\PassOptionsToPackage{dvipsnames, table}{xcolor}
\documentclass[12pt,reqno]{amsart}
\usepackage{amsmath,amssymb,graphicx,mathrsfs,amsopn,stmaryrd,amsthm,color,bm,extarrows,ytableau}
\usepackage{newtxtext}
\usepackage{float}
\usepackage[dvipsnames]{xcolor}
\usepackage[colorlinks=true,citecolor=blue,linkcolor=blue]{hyperref}
\usepackage[alphabetic,nobysame]{amsrefs}
\renewcommand{\MR}[1]{} 
\usepackage[english]{babel}
\usepackage[mathscr]{euscript}
\usepackage[left=2.5cm,
            right=2.5cm,
            bottom=2.75cm
            ]{geometry}
\DeclareFontFamily{OT1}{pzc}{}
\DeclareFontShape{OT1}{pzc}{m}{it}{ <-> s*[1.0] pzcmi7t }{}
\DeclareMathAlphabet{\mathpzc}{OT1}{pzc}{m}{it}

\usepackage{collectbox}


\let\emptyset\varnothing
\usepackage{dsfont}
\usepackage{tikz}    
\usetikzlibrary{arrows,matrix,decorations.markings,shapes.geometric}   
\usetikzlibrary{decorations.pathreplacing}

\usepackage[all]{xy}
\allowdisplaybreaks
\xyoption{all}

\numberwithin{equation}{section}
\newtheorem{thm}{Theorem}[section]
\newtheorem{prop}[thm]{Proposition}
\newtheorem{lem}[thm]{Lemma}
\newtheorem{cor}[thm]{Corollary}
\newtheorem{conj}[thm]{Conjecture}
\newtheorem{problem}[thm]{Problem}

\theoremstyle{definition} 
\newtheorem{eg}[thm]{Example}
\newtheorem{dfn}[thm]{Definition}

\newtheorem{alphatheorem}{Theorem}

\newtheorem{alphaconjecture}[alphatheorem]{Conjecture}

\theoremstyle{remark}
\newtheorem{rem}[thm]{Remark}

\newcommand{\beq}{\begin{equation}}
\newcommand{\eeq}{\end{equation}}
\newcommand{\be}{\begin{equation*}}
\newcommand{\ee}{\end{equation*}}

\newcommand{\bA}{\mathbb{A}}
\newcommand{\bC}{\mathbb{C}}

\newcommand{\bZ}{\mathbb{Z}}

\newcommand{\bN}{\mathbb{N}}

\newcommand{\mc}{\mathcal}

\newcommand{\cW}{\mathcal{W}}

\newcommand{\hf}{\tfrac12}

\newcommand{\g}{\mathfrak{g}}
\newcommand{\gl}{\mathfrak{gl}}

\newcommand{\fksl}{\mathfrak{sl}}

\newcommand{\fka}{\mathfrak{a}}

\newcommand{\fkv}{\mathfrak{v}}
\newcommand{\fkw}{\mathfrak{w}}

\newcommand{\id}{{\mathrm{id}}}   
   
\newcommand{\gr}{{\mathrm{gr}}}

\newcommand{\ParNepsilon}{\operatorname{Par}_\epsilon(N)}
\newcommand{\ParN}{\operatorname{Par}(N)}

\newcommand{\wtl}{\widetilde}
\newcommand{\gge}{\geqslant}
\newcommand{\lle}{\leqslant}
\newcommand{\la}{\lambda}

\newcommand{\C}{{\mathbb C}}
\newcommand{\W}{{\mathbf W}}
\newcommand{\I}{{\mathbb I}}
\newcommand{\iI}{{}^\imath\I}
\newcommand{\Z}{{\mathbf Z}}

\newcommand{\Qbar}{\overline{Q}}
\newcommand{\Qfbar}{\overline{Q}^f}

\newcommand{\iEVW}{E_{V,W}^\imath}
\newcommand{\Hom}{{\mathrm{Hom}}}

\newcommand{\Y}{{\mathcal{Y}}}
\newcommand{\Yt}{{^{\imath}\mathcal{Y}}}

\newcommand{\yt}{{}^{\imath} \mathscr{Y}}
\newcommand{\U}{\mathscr{U}}
\newcommand{\iW}{{}^\imath \cW}
\newcommand{\iWbar}{{}^\imath \overline{\cW}}

\newcommand{\Tt}{T^\imath}
\newcommand{\Qt}{Q^\imath}
\newcommand{\dfo}{\eth}
\newcommand{\bv}{\mathbf{v}}
\newcommand{\bw}{\mathbf{w}}

\newcommand{\gkloa}{\mathsf{a}}
\newcommand{\gklob}{\mathsf{b}}
\newcommand{\gkloc}{\mathsf{c}}
\newcommand{\gklod}{\mathsf{d}}
\newcommand{\ow}{w}
\newcommand{\ox}{y}

\newcommand{\arxiv}[1]{\href{http://arxiv.org/abs/#1}{\tt arXiv:\nolinkurl{#1}}}

\newcommand{\purple}[1]{{\color{purple}#1}}

\begin{document}
\pagestyle{myheadings}
\setcounter{page}{1}

\title[Shifted twisted Yangians and affine Grassmannian $\mathrm{i}$slices]{Shifted twisted Yangians and affine Grassmannian $\mathrm{i}$slices}

\author{Kang Lu}
\address{Shenzhen International Center for Mathematics and Department of Mathematics, Southern University of Science and Technology, Shenzhen, China}\email{kang.lu.math@gmail.com, luk@sustech.edu.cn}
\author{Weiqiang Wang}
\address{Department of Mathematics, University of Virginia, 
Charlottesville, VA 22903, USA}\email{ww9c@virginia.edu}
\author{Alex Weekes}
\address{Département de mathématiques, Université de  Sherbrooke, Sherbrooke, QC, Canada}
\email{alex.weekes@usherbrooke.ca}


\begin{abstract}
In a prequel we introduced the shifted iYangians ${}^\imath \mathcal Y_\mu$ associated to quasi-split Satake diagrams of type ADE and even spherical coweights $\mu$, and constructed the iGKLO representations of ${}^\imath \mathcal Y_\mu$, which factor through truncated shifted iYangians ${}^\imath \mathcal Y_\mu^\lambda$. In this paper, we show that ${}^\imath \mathcal Y_\mu$  quantizes the involutive fixed point locus ${}^\imath\mathcal W_\mu$ arising from affine Grassmannians of type ADE, and supply strong evidence toward the expectation that ${}^\imath \mathcal Y_\mu^\lambda$ quantizes a top-dimensional component of the affine Grassmannian islice ${}^\imath\overline{\mathcal W}_\mu^\lambda$. We identify the islices ${}^\imath\overline{\mathcal W}_\mu^\lambda$ in type AI with suitable nilpotent Slodowy slices of type BCD, building on the work of Lusztig and Mirkovi\'c-Vybornov in type A. We propose a framework for producing ortho-symplectic (and hybrid) Coulomb branches from split (and nonsplit) Satake framed double quivers, which are conjectured to relate closely to the islices ${}^\imath\overline{\mathcal W}_\mu^\lambda$ and the algebras ${}^\imath \mathcal{Y}_\mu^\lambda$. 
\end{abstract}
	
\maketitle

\setcounter{tocdepth}{2}
\tableofcontents

\thispagestyle{empty}

\section{Introduction}

A connection between truncated shifted Yangians and affine Grassmannian slices $\overline{\cW}_\mu^\la$, for dominant coweights $\mu$, was established in \cite{KWWY14}. A key input here is what has become known as the GKLO representations of shifted Yangians, inspired by \cite{GKLO}, whose images are known as truncated shifted Yangians $Y_\mu^\la$. These algebras  can be viewed as a vast generalization of the construction of Brundan-Kleshchev \cite{BK06}, who showed that truncated shifted Yangians $Y_\mu^\la$ in type A (with $\la =N\varpi_1^\vee$) are isomorphic to finite W-algebras of $\gl_N$, quantizing an isomorphism between certain affine Grassmannian slices and nilpotent Slodowy slices \cite{WWY20,MV22}.

The above connection extends naturally in the framework of Coulomb branches (see \cite{BFN19,NW23}). It is shown that a more general class of truncated shifted Yangians $Y_\mu^\la$ quantize the generalized affine Grassmannian slices $\overline{\cW}_\mu^\la$, for arbitrary (not necessarily dominant) coweights $\mu$. Moreover, the spaces $\overline{\cW}_\mu^\la$ are identified as Coulomb branches of cotangent type associated with framed quiver gauge theories (with symmetrizers). This story also admits a $q$-deformation: it is shown in \cite{FT19} that shifted affine quantum groups are mapped homomorphically into the quantized K-theoretic Coulomb branches of framed quiver gauge theories.

\vspace{2mm}
Building on the algebraic foundation in the prequel \cite{LWW25sty}, the goal of this paper is to formulate an $\imath$-fication of the above connections, a  framework in which new families of algebras provide a quantization of new Poisson varieties arising from the classical geometric setting. We formulated shifted iYangians and their iGKLO representations in \cite{LWW25sty}. In this paper we shall formulate and establish interconnections among 

$\blacktriangleright$ shifted iYangians 

$\blacktriangleright$ iGKLO representations 

$\blacktriangleright$ affine Grassmannian islices 

$\blacktriangleright$ type BCD nilpotent Slodowy slices  

$\blacktriangleright$   (conjecturally) iCoulomb branches. 

\vspace{2mm}
Some of the main players such as affine Grassmannian {\em islices} and iCoulomb branches are new, and there was little expectation in literature that shifted iYangians of quasi-split ADE type and their truncations admit geometric interpretations in such a generality as formulated in this paper. The connection between algebra and geometry is made through the iGKLO representations of shifted iYangians. Let us explain the main results of the paper. 

\subsection{Shifted iYangians and iGKLO}

This subsection provides a quick overview of the main constructions in the prequel \cite{LWW25sty}. 

Twisted Yangians admit R-matrix forms \cite{Ols92, MR02, GuayR16}, and more recently, twisted Yangians in Drinfeld presentations have been constructed in \cite{LWZ25, LWZ25degen} for split type and then for quasi-split type \cite{LZ24}. The Drinfeld presentations are determined by the underlying finite type quasi-split Satake diagrams $(\I,\tau)$, where $\tau$ is a diagram involution of the Dynkin diagram $\I$; the cases with $\tau=\id$ are called split. The identification of the two definitions of twisted Yangians (via R-matrix and via Drinfeld presentation) is only established in type AI and AIII; mostly we follow the ones in Drinfeld presentation in this paper and will refer to them as iYangians.

Associated to an arbitrary quasi-split Satake diagram $(\I,\tau)$, we have a symmetric pair $(\g, \g^{\omega_\tau})$, where $\omega_\tau =\omega_0\circ\tau$ for the Chevalley involution $\omega_0$, and the corresponding iYangian $\Yt=\Yt_0$. Generalizing the iYangians in Drinfeld presentation, we introduce the shifted iYangians $\Yt_\mu$ in Definition \ref{def:qsplit}, where $\mu$ is an arbitrary even spherical weight (see Definition \ref{def:spherical}). A dominant family of shifted iYangians $\Yt_\mu$ of type AI, for $\mu$ dominant and even (automatically spherical since $\tau=\id$), has played a fundamental role in \cite{LPTTW25}. 

The iGKLO homomorphisms $\Phi_\mu^\la$ from $\Yt_\mu$ to a ring of difference operators were constructed in \cite{LWW25sty}, and a truncated shifted iYangian (TSTY) of type $(\I,\tau)$ is by definition the image of the homomorphism $\Phi_\mu^\la$ and denoted by $\Yt_\mu^\la$.

\subsection{Twisted Yangians quantize loop symmetric spaces}

One can axiomize the notion of a Yangian $\U_\hbar(\g[z])$ as a quantization of the current algebra $U(\g[z])$; see Definition \ref{def:Yangian:axioms}. Following the quantum duality principle \cite{Dr87, G02}, one defines a subalgebra $\U_\hbar(\g[z])'$ of such a Yangian. Inspired by \cite{KWWY14,S16,FT19b}, under a general technical Assumption (\ref{eq: QDP1}), one shows that $\U_\hbar(\g[z])'$ quantizes the Poisson group $G_1[\![z^{-1}]\!]$; see Proposition \ref{prop: QDP Yangian1}. 

Let $\omega$ be an involution and an isometry on $\g$. It induces an involution on $\g[z]$ and a Poisson involution $\sigma$ in \eqref{eq:sigma} on $G_1[\![z^{-1}]\!]$. One can axiomize a twisted Yangian $\U_\hbar(\g[z]^\omega)$, which is a coideal subalgebra of $\U_\hbar(\g[z])$ and quantizes the twisted current algebra $U(\g[z]^\omega)$ (see Definition~ \ref{def:tYaxioms}); one also defines its subalgebra $\U_\hbar(\g[z]^\omega)'$ following the quantum duality principle. Under a technical Assumption (\ref{eq: QDP2}), by applying a general recipe of Dirac reduction (see \cite{Xu03} or \cite{F94}) at the loop group level, we obtain an isomorphism of Poisson homogeneous  spaces
$G_1[\![z^{-1}]\!]/ G_1[\![z^{-1}]\!]^\omega \stackrel{\sim}{\rightarrow}  G_1[\![z^{-1}]\!]^\sigma$, which is compatible with the natural Poisson algebra of $G_1[\![z^{-1}]\!]$; see \eqref{eq: exp map} and Proposition~ \ref{prop:loopSymSpace}. 
     
\begin{alphatheorem}  [Theorem \ref{thm:tY:G1sigma}]
        \label{thm:tY:G1:Intr}
$\U_\hbar(\g[z]^\omega)'$ quantizes the Poisson symmetric space $G_1[\![z^{-1}]\!]/ G_1[\![z^{-1}]\!]^\omega$ or equivalently the affine scheme $G_1[\![z^{-1}]\!]^\sigma$ with its (doubled) Dirac Poisson structure, and the inclusion $\U_\hbar(\g[z]^\omega) \subset \U_\hbar(\g[z])$ quantizes the map $G_1[\![z^{-1}]\!] \rightarrow G_1[\![z^{-1}]\!]^\sigma$ defined by $g \mapsto g \sigma(g)$. The left coideal structure on $\U_\hbar(\g[z]^\omega)$ quantizes the left action of $G_1[\![z^{-1}]\!]$ on $G_1[\![z^{-1}]\!]^\sigma$ given by $g\cdot p = g p \sigma(g).$
\end{alphatheorem}

\subsection{Shifted iYangians as a quantization of $\iW_\mu$}

Let $G$ be the adjoint group of the simple Lie algebra $\g$. 
Following \cite{FKPRW}, for any coweight $\mu$, the closed subscheme $\cW_\mu$ of $G(\!(z^{-1})\!)$ defined in \eqref{eq:Wmu} is quantized by the shifted Yangian $Y_\mu(\g)$. 
Associated to the quasi-split Satake diagram $(\I,\tau)$, the involution $\omega_\tau=\omega_0\circ\tau$ of $\g$ leads to an anti-involution $\sigma$ on the loop algebra \eqref{eq:sigma:gz} as well as on the loop group $G(\!(z^{-1})\!)$. We show that $\sigma$ preserves the subscheme $\cW_\mu$ if and only $\mu$ is even spherical. In this case, $\sigma$ restricts to a Poisson involution on $\cW_\mu$, and hence the $\sigma$-fixed point locus of $\cW_\mu$, denoted $\iW_\mu$, inherits a Poisson structure from $\cW_\mu$ by (doubled) Dirac reduction. Our next main result is that the Poisson algebra $\C[\iW_\mu]$ is quantized by the shifted iYangian $\Yt_\mu$.

\begin{alphatheorem}  [Proposition \ref{prop:iWmu}, Theorem \ref{thm:iWmuPoisson}]
   \label{thm:iWmuPoisson:Intr}
Let $\mu$ be an even spherical coweight.
\begin{enumerate}
    \item For $\mu = \mu_1 + \tau \mu_1$, there is an isomorphism of $\Qt\times \bZ$--graded Poisson algebras 
    $\gr^{F_{\mu_1}^\bullet} \Yt_\mu \cong \bC[\iW_\mu],$ 
    which matches the corresponding generators.

    \item The coordinate ring $\bC[\iW_\mu]$ is the Poisson algebra generated by $h_i^{(r)}, b_i^{(s)}$ for $i \in \I$, $r \in \bZ$ and $s \gge 1$, with explicit defining Poisson relations. 

    \item $\bC[\iW_\mu]$ is a polynomial algebra with explicit PBW generators.

    \item For any antidominant weight $\nu$ such that $\nu+\tau\nu$ is even, the shift map $\iota_{\mu,\nu}^\tau: \iW_{\mu+\nu+\tau\nu} \rightarrow \iW_\mu$ is quantized by the shift homomorphism $\iota_{\mu,\nu}^\tau : \Yt_\mu \rightarrow \Yt_{\mu+\nu+\tau\nu}$.
    \end{enumerate}
\end{alphatheorem}

Theorem \ref{thm:tY:G1:Intr} and Theorem \ref{thm:iWmuPoisson:Intr}(1) complement each other, as the iYangians $\Yt_\mu$ are expected to be compatible with the ones used in Theorem \ref{thm:tY:G1:Intr}; see Remark \ref{rem:tYsame}. 

\subsection{Affine Grassmannian islices}

For coweights $\la \gge \mu$ with $\la $ dominant, the (generalized) affine Grassmannian slice $\overline{\cW}_\mu^\la$ in  \eqref{eq:Wmula} as a closed subscheme of $\cW_\mu$ was introduced in \cite{BFN19}. For $\mu$ dominant, $\overline{\cW}_\mu^\la$ are the usual affine Grassmannian slices between spherical Schubert varieties in the affine Grassmannian. The Poisson involution $\sigma$ on $\cW_\mu$ preserves $\overline{\cW}_\mu^\la$ if and only if $\la$ is $\tau$-invariant, and the affine Grassmannian islice $\iWbar_\mu^\la$ is defined to be the $\sigma$-fixed point locus of $\overline{\cW}_\mu^\la$; cf. \eqref{eq:islices}. Thus $\iWbar_\mu^\la$ inherits a Poisson structure from $\overline{\cW}_\mu^\la$ by (doubled) Dirac reduction.  Some basic properties of $\iWbar_\mu^\la$ such as its symplectic leaves can be derived from the counterparts for $\overline{\cW}_\mu^\la$; see Theorem \ref{thm:fixedWlm}.

\begin{rem}
    The varieties $\iWbar_\mu^\la$ are not twisted affine Grassmannian slices in the sense of \cite[\S 3.9]{BF17} or \cite{PR08}. In fact, in the cases where twisted affine Grassmannian slices are defined as fixed points for an involution (instead of a map of order 3), the involutions used are what we denote by $\omega$ in \S \ref{ssec:loop sym spaces}. As such, these twisted affine Grassmannian slices are fixed points under anti-Poisson involutions, and thus are naturally Lagrangians inside ordinary affine Grassmannian slices.  This clarifies and refutes an expectation from \cite[\S 1.1]{TT24}.
\end{rem}

Recall the filtration on $\Yt_\mu$ appearing in Theorem \ref{thm:iWmuPoisson:Intr}(1) and \eqref{eq:mu1},
and there is also a filtration on $\mc A_{\bm z=0}$ as in \eqref{filter:A}. With \cite{KWWY14,BFN19,W19} in mind, one is tempted to claim that the surjective iGKLO homomorphism $\Phi_\mu^\la:\Yt_\mu[\bm z]\twoheadrightarrow \Yt_\mu^\la$ arising from Theorem \ref{thm:GKLOquasisplit} quantizes the geometric embedding $\iWbar_\mu^\la \hookrightarrow \iW_\mu$. However, one complication arises from the fact that the islices $\iWbar_\mu^\la$ may be reducible and may not be normal (as it happens to fixed point loci of irreducible normal varieties). Such examples do occur in the setting of nilpotent Slodowy slices of type BCD, which turn out to be affine Grassmannian islices of type AI by Theorem \ref{thm:sliceBCD:Intr}. Recall notation $\iI$ from \eqref{eq:iI} and $\fkv_i$ from \eqref{ell_theta}.

\begin{alphatheorem}  [Theorems \ref{thm:ctgklo} and \ref{thm:ctgklo:deform}]   
\label{thm:islices:Intro}
    \begin{enumerate}
    \item The open subscheme $U_\mu^\la \subset \overline{\cW}_\mu^\la$ from \eqref{eq:U} is invariant under $\sigma$. Its fixed point locus ${}^\imath U_\mu^\la$ is nonempty if and only if the parity condition \eqref{parity} holds, and in this case we have 
    $\dim {}^\imath U_\mu^\la = 2 \sum_{i \in \iI} \fkv_i.$

    \item The homomorphism $\Phi_\mu^{\la, \bm z =0} : \Yt_\mu \rightarrow \mc A_{\bm z=0}$ is filtered, and the associated graded map
    $\operatorname{gr} \Phi_\mu^{\la, \bm z =0} : \bC[ \iW_\mu ]\longrightarrow \operatorname{gr} \mc A_{\bm z =0}$
    defines the closure $\overline{C}_\mu^\la \subseteq \iWbar_\mu^\la$, i.e.,~the kernel of this map is the defining ideal of $\overline{C}_\mu^\la$. Here $C_\mu^\la \subseteq {}^\imath U_\mu^\la$ is a top-dimensional irreducible component. In case when $\iWbar_\mu^\la$ is irreducible, $\overline{C}_\mu^\la$ is equal to $\iWbar_\mu^\la$.
    \end{enumerate}   
\end{alphatheorem}
In particular, if $\iWbar_\mu^\la$ is irreducible then it has dimension $2\sum_{i\in\iI}\fkv_i$, and we expect that this is true more generally.  It is quite plausible that $\iWbar_\mu^\la$ are often irreducible, though we do not know of a general criterion. We expect that $\overline{C}_\mu^\la$ is quantized by the TSTY $\Yt_\mu^\la$; see Conjecture \ref{conj:TSTY:islices}. This boils down to a technical issue of identifying two filtrations on $\Yt_\mu^\la$.

\subsection{Nilpotent Slodowy slices of type BCD}

In this subsection, we specialize to $G = \mathrm{PGL}_N$ and $\tau=\id$, which correspond to the Satake diagram of type AI. 

Building on Lusztig's isomorphism \cite{Lu81} between the nilpotent cone 
$\mathcal{N}_{\fksl_N}$ and the affine Grassmannian slice $\overline{\cW}_0^{N \varpi_1^\vee}$, Mirkovi\'c-Vybornov \cite{MV22} established a general identification between nilpotent Slodowy slices and affine Grassmannian slices of type A:
${\mathbb{O}}_{\la} \cap \mathcal{S}_{\mu} \cong {\cW}_\mu^\la,$ and $\overline{\mathbb{O}}_{\la} \cap \mathcal{S}_{\mu} \cong \overline{\cW}_\mu^\la$; see \eqref{weight=par} for the weight-partition correspondence. It was mentioned in \cite[Footnote 1]{MV22} that ``These observations clearly do not literally extend beyond type A". We offer a proper generalization of the MV isomorphism to classical type. 

Consider an involution $\sigma_\epsilon$ of $\fksl_N$ sending $X \mapsto - J_\epsilon^{-1} X^T J_\epsilon$, where $J_\epsilon$ is any invertible $N\times N$ matrix which is symmetric if $\epsilon=+$ and skew-symmetric if $\epsilon=-$. The fixed point subalgebra $\fksl_N^\epsilon = (\fksl_N)^{\sigma_\epsilon}$
is $\mathfrak{so}_N$ if $\epsilon=+$, and $\mathfrak{sp}_N$ if $\epsilon=-$. By the classic results of Gerstenhaber and of Hesselink (see Proposition \ref{prop:cloc}), intersecting the nilpotent cone, nilpotent orbits, and nilpotent orbit closures of $\fksl_N$ with $\fksl_N^\epsilon$ gives rise to the corresponding nilpotent cone, nilpotent orbits (if nonempty), and nilpotent orbit closures in the classical Lie algebra $\fksl_N^\epsilon$ compatibly. 

Given an orthogonal/symplectic partition $\mu$ of $N$ (for $\epsilon=+/-$), following \cite[\S 4.2]{T23} we can choose $J_\epsilon$ in \eqref{eq:tinvo} for an involution $\sigma_\epsilon$ of $\fksl_N$ so that there exists an $\fksl_2$-triple $\{e,h,f\} \subset \fksl_N^\epsilon$ (i.e., fixed pointwise by $\sigma_\epsilon$) and $e$ has Jordan form $\mu$. Then the involution $\sigma_\epsilon$ restricts to an involution on the Slodowy slice $ \mathcal{S}_{\mu}$, and the corresponding $\sigma_\epsilon$-fixed point locus can be identified with the Slodowy slice $ \mathcal{S}_{\mu}^\epsilon$ in $\fksl_N^\epsilon$. Using Topley's results, we show that the Mirkovi\'c-Vybornov isomorphism is compatible with taking fixed point loci on both sides. We refer to \eqref{eq:3Par} for notation $\ParN_{\lle n-1}$ and $\ParNepsilon^\diamond_{\lle n}$ on subsets of (symplectic/orthogonal) partitions.
    
\begin{alphatheorem}  [Theorem \ref{thm:loci=sliceBCD}]
\label{thm:sliceBCD:Intr}
Let $\tau =\id$. Let $\la \in \ParN_{\lle n-1}$ and $\mu \in\ParNepsilon^\diamond_{\lle n}$ be such that $\la \unrhd \mu$; they are identified with dominant coweights for $\mathrm{PGL}_n$ such that $\la\gge \mu$. Then we have Poisson isomorphisms
    \[
    \iW_\mu^{\la}\cong {\mathbb{O}^{\epsilon}_{\la}} \cap \mathcal{S}_{\mu}^\epsilon, \quad \quad \iWbar_\mu^{\la}\cong \overline{\mathbb{O}}^{\epsilon}_{\la} \cap \mathcal{S}_{\mu}^\epsilon.
    \]
    Moreover,
    \begin{enumerate}
        \item $\iW_\mu^\la$ is nonempty if and only if $\la \in \ParNepsilon$.  
        \item The variety $\iWbar_\mu^\la$ is the closure of its stratum $\iW_\mu^{\la'} \cong \mathbb{O}^\epsilon_{\la'} \cap \mc S^\epsilon_{\mu}$, where $\la'$ is  the unique maximal element $\la' \in \ParNepsilon$ satisfying $\la'\unlhd \la$.  
    \end{enumerate}
\end{alphatheorem}

It follows by Theorem \ref{thm:sliceBCD:Intr} that the normalizations of affine Grassmannian islices $\iWbar_\mu^{\la}$ of type AI are always symplectic singularities; cf., e.g., \cite[\S1.2]{FJLS17}; they may not always be irreducible though \cite{KP82} (cf. \cite[\S1.6.2]{FJLS17}).

A class of iquiver varieties was formulated by Li \cite{Li19} as fixed point loci of Nakajima's quiver varieties for ADE quivers. Moreover, certain iquiver varieties (generalizing the cotangent bundles of flag varieties of type B/C; cf. \cite{BKLW18}) are identified with involutive fixed point loci of nilpotent Slodowy slices of type A (which are claimed to be nilpotent Slodowy slices of type BCD if nonempty) and their partial resolutions \cite{Li19}; this identification is built on the isomorphism between Nakajima quiver varieties and nilpotent Slodowy slices of type A (a conjecture of Nakajima \cite{Nak94} proved by Maffei \cite{Maf05}).

\subsection{iCoulomb branches}

Let $Q=(\I,\Omega)$ be an ADE quiver. Let $(V_i; W_i)_{i\in \I}$ be a representation of its framed double quiver $\Qfbar$. Associated to such datum we have a symplectic vector space $E_{V,W}$ \eqref{E_VW} with the action of a gauge group $GL(V)$ and a flavor symmetric group $GL(W)$. This gives rise to Coulomb branch of cotangent type $\mc M_C(V,W)$ \cite{BFN19}. 

Let $(\I,\tau)$ be a Satake diagram as before. We impose $\tau$-symmetry and ``not-$2$-odd" parity conditions \eqref{symmetricVW:I1} on the dimension vectors of $V, W$, and fix a bipartite partition of $\I$ in \eqref{bipartite}. Given such datum, we associate with new vector spaces $V_i^\imath, W_i^\imath$, for $i\in \iI$, in \eqref{VWimath}. We explain in details the $\imath$-fication process on quivers and representations by reducing to the cases of the rank one and two Satake double quivers; see Tables \ref{tab:framed Satake:rk1} to \ref{tab:framed osp:rk1:VWi} in Section \ref{sec:Coulomb} for diagrammatic illustrations.

Then we formulate a new symplectic vector space $\iEVW$ in \eqref{iEVW} with the actions of a new gauge group $G^\imath(V^\imath)$ and a new flavor symmetry group $G^\imath(W^\imath)$; see Lemma \ref{lem:iEVW}, and we have a Coulomb branch $\mathcal M_C^\imath(V^\imath,W^\imath)$ (not of cotangent type in general) following \cite{BDFRT}. 
The component groups of 
$G^\imath(V^\imath)$ and $G^\imath(W^\imath)$ are classical type; in case of split Satake diagrams, all component groups are orthogonal or symplectic. Recall that generalized affine Grassmannian slices $\overline{\cW}_\mu^\lambda$ are realized as Coulomb branches $\mc M_C(V, W)$ \cite{BFN19}. 

\begin{alphaconjecture}
    [Conjecture \ref{conj:iCBr:islices}] \begin{enumerate}
    \item 
    The iCoulomb branch $\mathcal M_C^\imath(V^\imath,W^\imath)$ is a normalization of a top-dimensional component of the affine Grassmannian islice ${}^\imath\overline{\cW}_\mu^\la$;
    \item 
    Truncated shifted iYangians are (subalgebras of) quantized Coulomb branches.  
\end{enumerate}
\end{alphaconjecture}

Here is some numerology: the rank vectors of the component groups of $G^\imath(V^\imath)$ and $G^\imath(W^\imath)$ are the dimension vectors $(\fkv_i)_{i\in\iI}$ and $(\fkw_i)_{i\in\iI}$ in \eqref{ell_theta}--\eqref{varsigma} which arise from truncated shifted iYangians and affine Grassmannian islices. In particular, the iCoulomb branch $\mathcal M_C^\imath(V^\imath,W^\imath)$ and the affine Grassmannian islice ${}^\imath\overline{\cW}_\mu^\la$ have the same dimension. We refer to Remarks \ref{rem:physics} and \ref{rem:subtle} for more discussions on related physics predictions and subtleties on the conjecture.  

\subsection{Related works and perspectives}

This paper provides a new framework where a family of algebras (shifted iYangians and TSTY's) associated with quasi-split Satake diagrams quantize Poisson varieties which are naturally constructed. The ordinary (truncated) shifted Yangians and affine Grassmannian slices can be viewed to be associated with Satake diagrams of diagonal type. These new algebras are expected to admit rich representation theories which will be very interesting to develop. On the other hand, it will also be exciting to see if the normalizations of (top-dimensional components of) affine Grassmannian islices $\iWbar_\mu^{\la}$ provide new symplectic singularities. For type AI, these are indeed well-studied symplectic singularities thanks to the identification with nilpotent Slodowy slices of type BCD in Theorem~ \ref{thm:sliceBCD:Intr}. 

There are many papers in the mathematical physics literature on Coulomb branches with type A or non-type A gauge groups, and their connections to nilpotent Slodowy slices of classical types; see \cite{GaioW09, CDT13,CHMZ15,HK16,CHZ17,CHananyK19} for samples and references therein. Finkelberg, Hanany and Nakajima \cite{FHN25} have also been working on ortho-symplectic Coulomb branches and nilpotent Slodowy slices of classical type among other topics; we hope our $\imath$-perspectives on shifted iYangians and affine Grassmannian islices can be complementary to theirs and those in the math physics literature. 

In \cite{SSX25}, the authors also formulated (a variant of) shifted iYangians of type AIII$_{2n}$ and defined a homomorphism from it to a Coulomb branch. Their paper provides a first example of Coulomb branches which relates precisely to shifted iYangians (through iGKLO like ours), supporting our general proposal that truncated shifted iYangians are intimately related to iCoulomb branches through iGKLO and affine Grassmannian islices. Note that the Coulomb branch of type AIII$_{2n}$ {\em loc. cit.} is the only one among all iCoulomb branches which is of cotangent type and has a purely type A gauge group. 

The $\imath$-fication process often leads to type BCD features. The icanonical bases arising from iquantum groups of quasi-split type AIII have played a fundamental role in Kazhdan-Lusztig theory of type BCD \cite{BW18KL}, and these (affine) iquantum groups admit a geometric realization via flag varieties of classical type in \cite{BKLW18, FLLLW, SuW24}; Theorem \ref{thm:sliceBCD:Intr} on nilpotent Slodowy slices is a new example. The root of such type BCD phenomenon can be traced down to the split and quasi-split rank one iquantum groups. The geometric realization of twisted Yangians in this paper, which applies to all (quasi-split) ADE type, appears in very different forms. In light of the conjectural iCoulomb branch connection, the $\imath$-fication manifests itself again in the component gauge groups being type (A)BCD. This should be contrasted with the fact that the gauge groups are always of type A for the quiver gauge theories giving rise to Nakajima quiver varieties as Higgs branches and the corresponding Coulomb branches \cite{Nak94, BFN19}. 

The present paper is also naturally related to integrable systems.  Indeed, the algebras $\Yt_\mu^\la$ come equipped with natural commutative subalgebras, defined in \S \ref{ssec:commsub}, which we expect to be quantum integrable systems. The classical limits of these subalgebras give integrable systems in many examples, and we conjecture that this is always the case.  In particular, in the setting of Theorem \ref{thm:sliceBCD:Intr} we obtain integrable systems on the varieties $\mathbb{O}_{\la}^\epsilon \cap \mc S_{\mu}^\epsilon$.  We are unaware if these have been studied previously in general, though some cases are related to work of Harada \cite{Harada} who constructed integrable systems on generic coadjoint orbits for symplectic Lie algebras.

It is natural to expect a $q$-deformation of the main constructions in this paper. Drinfeld presentations of quasi-split affine iquantum groups have been constructed in \cite{LW21, mLWZ24}. We have been able to formulate the shifted affine iquantum groups of split and quasi-split types accordingly (compare \cite{FT19}), construct their iGKLO representations and thus define the truncated shifted affine iquantum groups. These are then expected to be related to the K-theory of affine Grassmannian islices and iCoulomb branches. 

The $\imath$-fication framework can be further enlarged. 
Some main constructions in this paper will be extended beyond quasi-split Satake diagrams in a sequel \cite{LWW}. In particular, we shall identify the affine Grassmannian islices of type AII and others with nilpotent Slodowy slices of classical types, complementary to the results in Section \ref{sec:slodowy}. 


The paper is organized as follows.
Section \ref{sec:S.T.Yangians} is mostly a summary of the prequel \cite{LWW25sty}, including (truncated) shifted iYangians $\Yt_\mu$ and iGKLO representations. The construction of filtrations on $\Yt_\mu$ in \S\ref{ssec:filtrations on stY} is new. 
In Section \ref{sec:tYdualityPrinciple}, we show that a version of twisted Yangians from quantum duality principle quantizes the Poisson loop symmetric space. 
In Section \ref{sec:stYloci}, we formulate Poisson subschemes $\cW_\mu$ of a loop group, and show that the shifted iYangian $\Yt_\mu$ quantizes $\cW_\mu$, using the filtrations from Section \ref{sec:S.T.Yangians}. 
In Section \ref{sec:islices}, we construct the affine Grassmannian slices $\iWbar_\mu^\la$, and supply strong evidence on the conjecture that the truncated shifted iYangians $\Yt_\mu^\la$ quantize a top-dimensional irreducible component of $\iWbar_\mu^\la$. 
In Section \ref{sec:slodowy}, specializing to type AI, we show that the affine Grassmannian slices $\iWbar_\mu^\la$ are isomorphic to explicit nilpotent Slodowy slices of type BCD.
In Section \ref{sec:Coulomb}, we formulate iCoulomb branches starting from data on type A quiver gauge theory together with an involution $\tau$. We speculate that the iCoulomb branch provides a normalization of a top-dimensional component of the affine Grassmannian islice $\iWbar_\mu^\la$, noting they have the same dimension.

\vspace{2mm}
\noindent {\bf Acknowledgement.}
K.L. and W.W. are partially supported by the NSF grant DMS-2401351. A.W.~is supported by an NSERC Discovery Grant.  W.W. thanks ICMS at Edinburgh for providing a stimulating atmosphere in August 2023 when some of the connections was conceived, and National University of Singapore (Department of Mathematics and IMS) for providing an excellent research environment and support at the final stage of this work. The main key results of this work were stated in W.W.'s conference talk at NCSU in October 2024 and this project was also announced in \cite{LPTTW25}. We thank Yaolong Shen, Changjian Su and Rui Xiong for communicating their work with us; we thank them, Gwyn Bellamy, Michael Finkelberg, Amihay Hanany, Pengcheng Li, Hiraku Nakajima, Eric Sommers, Ben Webster and Curtis Wendlandt for their many helpful discussions, comments and correspondences.


\section{Shifted iYangians and truncations}
\label{sec:S.T.Yangians}

In this section, we review from \cite{LWW25sty} shifted iYangians $\Yt_\mu$ associated to quasi-split ADE Satake diagrams $(\I,\tau)$ and even spherical coweights $\mu$ and their basic properties. We also recall the iGKLO representation of $\Yt_\mu$ which allows us to define the truncated shifted iYangians $\Yt_\mu^\la$. Section \ref{ssec:filtrations on stY} on filtrations of $\Yt_\mu$ is new, which is then used in \eqref{eq:quofilt}; the latter plays a crucial role in this paper. 

\subsection{Shifted iYangians of quasi-split ADE type}

Let $C=(c_{ij})_{i,j\in \I}$ be the Cartan matrix of type ADE, and let $\g$ be the corresponding simple Lie algebra. We fix a simple system $\{\alpha_i\mid i\in \I\}$ with corresponding set $\Delta^+$ of positive roots. Let $\tau$ be an involution of the  Dynkin diagram of $\g$, i.e., $c_{ij}=c_{\tau i,\tau j}$ such that $\tau^2=\mathrm{id}$; note that $\tau=\id$ is allowed. We refer to $(\I,\tau)$ as quasi-split Satake diagrams and call the Satake diagrams split if $\tau=\id$. The split Satake diagrams formally look identical to Dynkin diagrams, and the quasi-split Satake diagrams $(\I,\tau)$ with $\tau \neq \id$ can be found in Table \ref{tab:Satakediag}, where the index denotes the number of nodes in a given diagram. 

\begin{table}[H]  
\begin{center}
\centering
\setlength{\unitlength}{0.6mm}			\begin{picture}(70,35)(0,5)
    \put(-70,20){${\rm AIII}_{2r-1} \,(r\gge 1)$}
    \put(-70,-15){${\rm AIII}_{2r} \,(r\gge 1)$}
    \put(-70,-50){${\rm DI}_{r} \,(r\gge 4)$}
    \put(-70,-85){${\rm EII}_{6} $}
   
				\put(0,10){$\circ$}
				\put(0,30){$\circ$}
				\put(50,10){$\circ$}
				\put(50,30){$\circ$}
				\put(72,10){$\circ$}
				\put(72,30){$\circ$}
				\put(92,20){$\circ$}

                \put(3,11.5){\line(1,0){16}}
				\put(3,31.5){\line(1,0){16}}
				\put(23,10){$\cdots$}
				\put(23,30){$\cdots$}
				\put(33.5,11.5){\line(1,0){16}}
				\put(33.5,31.5){\line(1,0){16}}
				\put(53,11.5){\line(1,0){18.5}}
				\put(53,31.5){\line(1,0){18.5}}
				
				\put(75,12){\line(2,1){17}}
				\put(75,31){\line(2,-1){17}}

				\color{red}
                \put(-7,20){$\tau$}
				\qbezier(0,13.5)(-4,21.5)(0,29.5)
				\put(-0.25,14){\vector(1,-2){0.5}}
				\put(-0.25,29){\vector(1,2){0.5}}
				
				\qbezier(50,13.5)(46,21.5)(50,29.5)
				\put(49.75,14){\vector(1,-2){0.5}}
				\put(49.75,29){\vector(1,2){0.5}}

				\qbezier(72,13.5)(68,21.5)(72,29.5)
				\put(71.75,14){\vector(1,-2){0.5}}
				\put(71.75,29){\vector(1,2){0.5}}
			\end{picture}
            \\
\setlength{\unitlength}{0.6mm}			\begin{picture}(70,35)(0,5)

   
				\put(0,10){$\circ$}
				\put(0,30){$\circ$}
				\put(50,10){$\circ$}
				\put(50,30){$\circ$}
				\put(72,10){$\circ$}
				\put(72,30){$\circ$}
				
				\put(3,11.5){\line(1,0){16}}
				\put(3,31.5){\line(1,0){16}}
				\put(23,10){$\cdots$}
				\put(23,30){$\cdots$}
				\put(33.5,11.5){\line(1,0){16}}
				\put(33.5,31.5){\line(1,0){16}}
				\put(53,11.5){\line(1,0){18.5}}
				\put(53,31.5){\line(1,0){18.5}}
				
				\put(73.5,13.6){\line(0,1){16}}
				
				\color{red}
                \put(-7,20){$\tau$}
				\qbezier(0,13.5)(-4,21.5)(0,29.5)
				\put(-0.25,14){\vector(1,-2){0.5}}
				\put(-0.25,29){\vector(1,2){0.5}}
				
				\qbezier(50,13.5)(46,21.5)(50,29.5)
				\put(49.75,14){\vector(1,-2){0.5}}
				\put(49.75,29){\vector(1,2){0.5}}

				\qbezier(76,13.5)(80,21.5)(76,29.5)
				\put(75.85,14){\vector(-1,-2){0.5}}
				\put(75.85,29){\vector(-1,2){0.5}}
			\end{picture}
\\
\setlength{\unitlength}{0.6mm}	\begin{picture}(40,35)(20,-15)
				\put(0,-1){$\circ$}
				\put(3,0){\line(1,0){16.5}}
				\put(20,-1){$\circ$}
				\put(64,-1){$\circ$}
				\put(84,-10){$\circ$}
				\put(84,9.5){$\circ$}

				\put(38,0){\line(-1,0){15.5}}
				\put(64,0){\line(-1,0){15}}
				
				\put(40,-1){$\cdots$}
				
				\put(83.5,9.5){\line(-2,-1){16.5}}
				\put(83.5,-8.5){\line(-2,1){16.5}}
				
				\put(12,-20.5){\begin{picture}(100,40)	\color{red}
                \put(79,20){$\tau$}
						\qbezier(75,13.5)(79,21.5)(75,29.5)
						\put(75.25,14){\vector(-1,-2){0.5}}
						\put(75.25,29){\vector(-1,2){0.5}}
				\end{picture}}
			\end{picture}
        \\
\setlength{\unitlength}{0.6mm}		\begin{picture}(70,35)(20,7)
                \put(10,35){\rotatebox[origin=c]{180}{\begin{picture}(100,10)
							
							\put(10,10){$\circ$}
							
							\put(32,10){$\circ$}
							
							\put(10,30){$\circ$}
							
							\put(32,30){$\circ$}

							
							\put(31.5,11){\line(-1,0){19}}
							\put(31.5,31){\line(-1,0){19}}
							
							\put(52,22){\line(-2,1){17.5}}
							\put(52,20){\line(-2,-1){17.5}}
							
							\put(54.7,21.2){\line(1,0){19}}

							\put(52,20){$\circ$}
							
							\put(74,20){$\circ$}
					\end{picture}}

					\put(-37,-32.5){\begin{picture}(100,40)	\color{red}
							\qbezier(5,13.5)(9,21.5)(5,29.5)
							\put(5.25,14){\vector(-1,-2){0.5}}
							\put(5.25,29){\vector(-1,2){0.5}}
					\end{picture}}
					\put(-15,-32.5){\begin{picture}(100,40)	\color{red}
                    \put(9,20){$\tau$}
							\qbezier(5,13.5)(9,21.5)(5,29.5)
							\put(5.25,14){\vector(-1,-2){0.5}}
							\put(5.25,29){\vector(-1,2){0.5}}
					\end{picture}}
				}
				
			\end{picture}
\end{center}
\vspace{0.5cm}
\caption{Quasi-split Satake diagrams with $\tau \neq {\rm Id}$}
\label{tab:Satakediag}
\end{table}

Denote $\I_0$ the set of fixed points of $\tau$ in $\I$, i.e., $\I_0 =\{i\in \I \mid \tau i=i\}$. Let $\I_1$ be a set of representatives for $\tau$-orbits in $\I$ of length 2 and define $\I_{-1}=\tau \I_1$; $\I_1$ can be conveniently chosen so that it is underlying Dynkin subdiagram is connected. Then $\I =\I_1 \sqcup \I_0 \sqcup \I_{-1}$. Set \begin{align}  \label{eq:iI}
\iI=\I_1 \sqcup \I_0.
\end{align}

The involution $\tau$ naturally acts on the (co)root and (co)weight lattices of $\g$. A weight/coweight in this paper is always meant to be integral.

\begin{dfn} \label{def:spherical}
A coweight $\mu$ is called {\em spherical} if $\mu =\mu_1 +\tau \mu_1$ for some coweight $\mu_1$. A coweight $\mu$ is \emph{even} if $\mu = 2\mu'$ for some coweight $\mu'$, i.e., $\langle \mu, \alpha_i\rangle \in 2\bZ$ for all $i\in \I$. 
\end{dfn}

We shall assume that a shift coweight $\mu$ to be even spherical in the context of shifted iYangians below, and remarkably, the same condition will be needed in consideration of fixed point loci of affine Grassmannian slices later on. 
Denote $[A,B]_+ =AB+BA$. 

\begin{dfn} \cite[Definition 2.2]{LWW25sty}
\label{def:qsplit}
Let $\mu$ be an even spherical coweight. The \textit{shifted iYangians $\Yt_{\mu}:=\Yt_\mu(\g)$ of quasi-split type} is the $\bC$-algebra with generators $H_{i}^{(r)}$, $ B_{i}^{(s)}$, for $i\in \I$, $r\in\bZ$, and $s\in\bZ_{>0}$, subject to the following relations, for $r,r_1,r_2\in \bZ$ and $s,s_1,s_2 \in \bZ_{>0}$: 
\begin{align}
& H_i^{(r)}=0~\text{ for }~r<-\langle \mu,\alpha_i\rangle,\quad H_{i}^{(-\langle \mu,\alpha_i\rangle)}=1,\label{def0}\\
&[H_i^{(r_1)},H_j^{(r_2)}]=0,\label{hhIII}\\
&[H_i^{(r+2)},B_j^{(s)}]-[H_i^{(r)},B_j^{(s+2)}]=\frac{c_{ij}-c_{\tau i,j}}{2} [H_{i}^{(r+1)},B_j^{(s)}]_+\label{hbNqs} \\
& \hskip4.92cm +\frac{c_{ij}+c_{\tau i,j}}{2} [H_{i}^{(r)},B_j^{(s+1)}]_+ +\frac{c_{ij}c_{\tau i,j}}{4}[H_i^{(r)},B_j^{(s)}], \nonumber\\
&[B_i^{(s_1+1)},B_j^{(s_2)}]-[B_i^{(s_1)},B_j^{(s_2+1)}]=\frac{c_{ij}}2 [B_i^{(s_1)},B_j^{(s_2)}]_+ + 2\delta_{\tau i,j}(-1)^{s_1}H_j^{(s_1+s_2)},\label{bbNqs}
\end{align}
and the Serre relations: for $c_{ij}=0$,
\beq\label{bbtau}
[B_i^{(s_1)},B_{j}^{(s_2)}]=(-1)^{s_1-1}\delta_{\tau i,j}H_{j}^{(s_1+s_2-1)},\\
\eeq
and for $c_{ij}=-1$, $i\ne \tau i\ne j$,
\beq\label{eq:Serre-ord}
\mathrm{Sym}_{s_1,s_2}\big[B_{i}^{(s_1)},[B_{i}^{(s_2)},B_{j}^{(s)}]\big]=0,
\eeq
and for $c_{ij}=-1$, $i=\tau i$,
\beq\label{SerreIII}
\begin{split}
\mathrm{Sym}_{s_1,s_2}\big[B_{i}^{(s_1)},[B_{i}^{(s_2)},B_{j}^{(s)}]\big] 
=
(-1)^{s_1-1}[H_{i}^{(s_1+s_2)},B_j^{(s-1)}],
\end{split}
\eeq
and for $c_{i,\tau i}=-1$,
\beq\label{SerreIII2}
\begin{split}
\mathrm{Sym}_{s_1,s_2}\big[B_{i}^{(s_1)},[B_{i}^{(s_2)},B_{\tau i}^{(s)}]\big]
=
4\,\mathrm{Sym}_{s_1,s_2}(-1)^{s_1-1}\sum_{p\gge 0} 3^{-p-1}[B_{i}^{(s_2+p)},H_{\tau i}^{(s_1+s-p-1)}].
\end{split}
\eeq
Here, if $c_{ij}=-1$ and $s=1$, we use the following convention in \eqref{SerreIII}:
\begin{align*}
[H_{i}^{(r)},B_j^{(s-1)}]
=\sum_{p\gge 0}2^{-2p}\big([H_{i}^{(r-2p-2)},B_{j}^{(s+1)}]-
[H_{i}^{(r-2p-2)},B_{j}^{(s)}]_+\big),
\end{align*}
which follows from \eqref{hbNqs} if $s>1$, cf. \cite[Lem. 4.11]{LWZ25}.
\end{dfn}
In particular, if $\tau=\mathrm{id}$, we call $\Yt_\mu$  the \textit{shifted iYangian of split type}.

Note that $\Yt_\mu$ has a generating set
\beq\label{rvset3}
\big\{B_\beta^{(r)}| \beta\in\Delta^+, r>0\big\} \cup\big\{H_i^{(2p)}| i\in \I_0, 2p>-\langle \mu,\alpha_i\rangle\big\} \cup\big\{H_i^{(p)}| i\in \I_1, p>-\langle \mu,\alpha_i\rangle\big\}.    
\eeq

Let $Q$ be the root lattice for $\g$. Consider the abelian group (called the $\imath$root lattice)
    \begin{align}  \label{Qtau}
    \Qt = Q / \langle \beta+\tau \beta \mid \beta \in Q\rangle.
    \end{align}
See \cite[(3.3)]{BW18iCB} for a similar notion of $\imath$weight lattice.
For $\beta \in Q$, denote its image by $\overline{\beta} \in \Qt$. The algebra $\Yt_\mu$ admits a grading by $\Qt$ defined by assigning degrees $\deg H_i^{(r)} = \overline{0}$ and $\deg B_i^{(s)} = \overline{\alpha_i}$.  Indeed, it is not hard to see that the defining relations of $\Yt_\mu$ are all homogeneous.  Let $T = \operatorname{Spec} \bC[Q]$ denote the torus whose character lattice is $Q$.  Then the grading on $\Yt_\mu$ by $\Qt$ corresponds to an action of a subgroup $\Tt $ of $T$:
    \begin{align} \label{eq:def of Tt}
    \Tt = \{ t \in T \mid \tau(t) = t^{-1}\}  = \operatorname{Spec} \bC[\Qt].
    \end{align}
Since $\tau$ permutes the basis $\{\alpha_i\}_{i \in \I}$ for  $Q$, we can identify $\Qt \cong \bZ^{\I_1} \times (\bZ/2\bZ)^{\I_0}$, where the $\bZ^{\I_1}$-factor has basis $\{\overline{\alpha_i}\}_{i \in \I_1}$ while $(\bZ/2\bZ)^{\I_0}$ is generated by  $\{\overline{\alpha_i} \}_{i \in \I_0}$.  If we quotient $\Qt$ by its torsion subgroup $(\bZ/2\bZ)^{\I_0}$, we obtain an induced grading on $\Yt_\mu$ by $\bZ^{\I_1}$. This corresponds to the adjoint action of the elements  $\{H_i^{(-\langle \mu, \alpha_i\rangle+1)} 
\mid i \in \I_{1}\}$.

\begin{lem} \cite[Lemma 2.7]{LWW25sty}
\label{lem:shiftqs}
Let $\mu,\nu$ be coweights such that both $\mu$ and $\nu+\tau \nu$ are even spherical. Suppose further that $\nu$ is anti-dominant. Then there exists a homomorphism 
\begin{align}
    \label{shift_hom}
    \iota^\tau_{\mu,\nu}:\Yt_\mu \longrightarrow \Yt_{\mu+\nu+\tau \nu}
\end{align}
defined by
\[
H_i^{(r)}\mapsto  H_i^{(r- \langle\nu+\tau \nu,\alpha_i\rangle )}, \qquad B_i^{(s)}\mapsto \begin{cases}B_i^{(s-\langle\nu,\alpha_i\rangle)}, &\text{ if }\langle\nu,\alpha_i\rangle \text{ is even},\\
\sqrt{-1}B_i^{(s-\langle\nu,\alpha_i\rangle)}, &\text{ if }\langle\nu,\alpha_i\rangle \text{ is odd},
\end{cases}
\]
for $r\in\bZ$ and $s\in \bZ_{>0}$.
\end{lem}

\subsection{PBW basis}
\label{ssec:PBW for sty}

Consider a positive root $\beta$ and fix an arbitrary ordered decomposition $\beta=\alpha_{i_1}+\ldots+\alpha_{i_\ell}$ into simple roots such that the elements $[e_{i_1},[e_{i_2},\cdots[e_{i_{\ell-1}},e_{i_\ell}]\cdots]]$ is a nonzero element in the root subspace $\g_\beta$. For any $r>0$, we define a root vector in $\mc Y_\mu(\g)$
\be
B_\beta^{(r)}:=\Big[B_{i_1}^{(r)},\big[B_{i_2}^{(1)},\cdots[B_{i_{\ell-1}}^{(1)},B_{i_\ell}^{(1)}]\cdots\big]\Big].
\ee

\begin{thm} \cite[Theorem 2.16]{LWW25sty}  
\label{thm:pbw_arb}
Let $\mu$ be an arbitrary even spherical coweight. 
\begin{enumerate}
\item 
The set of ordered monomials in the root vectors \eqref{rvset3} forms a basis for $\Yt_\mu$. 
\item 
For any anti-dominant $\nu$ such that $\nu+\tau \nu$ is even spherical, the shift homomorphism $\iota^\tau_{\mu,\nu}:\Yt_\mu\to \Yt_{\mu+\nu+\tau \nu}$  in \eqref{shift_hom} is injective.
\end{enumerate}
\end{thm}
 
\subsection{Filtrations on shifted iYangians}
\label{ssec:filtrations on stY}

Let us define filtrations on shifted iYangians $\Yt_\mu$ via the PBW bases introduced in Theorem \ref{thm:pbw_arb}. (These are analogous to the filtrations $F_{\mu_1,\mu_2}^\bullet Y_\mu$ on shifted Yangians \cite[\S 5.4]{FKPRW} discussed just before Proposition \ref{prop: Wmu Poisson gens}.)

Choose any decomposition $\mu = \mu_1 + \tau \mu_1,$ where $\mu_1$ is a coweight. Then we may define an increasing filtration $F_{\mu_1}^\bullet \mc \Yt_\mu$ on $ \Yt_\mu$ by defining the degrees of PBW generators as follows:
\begin{equation}
    \label{eq: grading on iYmu}
    \deg H_i^{(r)} = r + \langle \mu, \alpha_i\rangle, \qquad \deg B_\beta^{(s)} = s + \langle \mu_1, \beta \rangle.
\end{equation}
That is, for each $k\in \bZ$ we define the $k$th filtered piece $F_{\mu_1}^k (\Yt_\mu)$ to be the span of all ordered monomials in the PBW generators whose total degree is at most $k$. This defines a separated and exhaustive filtration on the vector space $\Yt_\mu$, i.e.,~$\bigcap_k F_{\mu_1}^k(\Yt_\mu) = 0$ and $\bigcup_k F_{\mu_1}^k(\Yt_\mu) = \Yt_\mu$.  The following nontrivial claim will be established   below in Proposition \ref{prop: equal filtrations}:
\begin{prop}
    \label{prop:algfilt}
    $F_{\mu_1}^\bullet \Yt_\mu$ is an algebra filtration, is independent of the choice of PBW basis, and the associated graded algebra $\gr^{F_{\mu_1}^\bullet} \Yt_\mu$ is commutative.
\end{prop}

Since these filtrations are defined in terms of PBW basis vectors,  we note a simple consequence of Theorem \ref{thm:pbw_arb}.
\begin{lem}
    \label{lem:strictlyfiltered}
    Let $\mu$ be an even spherical coweight and fix a decomposition $\mu = \mu_1 + \tau \mu_1$. For any anti-dominant coweight $\nu$ such that $\nu+\tau \nu$ is even, the shift homomorphism $\iota_{\mu,\nu}^\tau: \Yt_\mu \hookrightarrow \Yt_{\mu+\nu+\tau \nu}$ respects the filtrations $F_{\mu_1}^\bullet \Yt_\mu$ and $F_{\mu_1+\nu}^\bullet \Yt_{\mu+\nu+\tau \nu}$.  Moreover, it is strictly filtered in the sense that
    $$
    \iota^{\tau}_{\mu,\nu} \big( F_{\mu_1}^k (\Yt_\mu) \big) = \iota_{\mu,\nu}^\tau( \Yt_\mu) \cap F_{\mu_1+\nu}^k( \Yt_{\mu+\nu+\tau\nu}).
    $$
\end{lem}

\begin{rem}
    \label{rem: change of filtration}
    Similarly to \cite[\S 5.4]{FKPRW}, various choices of filtrations on $\Yt_\mu$ defined above are related to one another by $\bZ$-gradings.  More precisely, if we have two decompositions $\mu = \mu_1 + \tau \mu_1 = \mu_2 + \tau \mu_2$, then we may define a $\bZ$--grading of $\Yt_\mu$ by 
    \begin{equation}\label{eq:collapsed grading}
    \deg H_i^{(r)} = 0, \qquad \deg B_i^{(s)} = \langle \mu_2-\mu_1, \alpha_i \rangle.
    \end{equation}
    The filtrations $F_{\mu_1}^\bullet \Yt_\mu$ and $F_{\mu_2}^\bullet \Yt_\mu$ are then related by the construction of \cite[Lemma 5.1]{FKPRW}, and in particular there are (non-graded) algebra isomorphisms
    $$
    \operatorname{Rees}^{F_{\mu_1}^\bullet} \Yt_\mu \cong \operatorname{Rees}^{F_{\mu_2}^\bullet} \Yt_\mu \quad\text{and} \quad \operatorname{gr}^{F_{\mu_1}^\bullet} \Yt_\mu \cong \operatorname{gr}^{F_{\mu_2}^\bullet} \Yt_\mu.
    $$
Thanks to $\tau(\mu_2-\mu_1)=-(\mu_2-\mu_1)$, the pairing of $\mu_2-\mu_1$ with the root lattice $Q$ naturally descends to a pairing with $\Qt$ in \eqref{Qtau}, allowing us to collapse the $\Qt$--grading on $\Yt_\mu$ to a $\bZ$--grading defined in \eqref{eq:collapsed grading}.
\end{rem}

\subsection{$\mathrm{i}$GKLO representations}
\label{sec:GKLO}

In this subsection, we recall a family of iGKLO representations of shifted iYangians of arbitrary quasi-split ADE type. 

\subsubsection{Ring of difference operators}
\label{ssec:quantumtorus}

Fix a dominant $\tau$-invariant coweight $\la$, i.e., $\tau \la =\la$,  and an even spherical coweight $\mu$ such that $\lambda \gge \mu$. We denote
\begin{align}\label{ell}
\la-\mu=\sum_{i\in \I}\bv_i\alpha_i^\vee,
\end{align}
where $\bv_i\in\bN$ for $i\in \I$. Denoting $\mathbb Z_2 =\{\bar 0, \bar 1\}$, we set
\beq \label{ell_theta}
\fkv_i =
\begin{cases} 
\bv_i, & \text{if } \tau i \neq i \\ 
\lfloor\tfrac{1}{2} \bv_i \rfloor, & \text{if } \tau i = i,
\end{cases}
\qquad 
\theta_i =
\begin{cases} 
0, & \text{if } \tau i \neq i \\ 
\delta_{\overline{\mathbf v}_i, \bar{1}}, & \text{if } \tau i = i. 
\end{cases}
\eeq
Introduce $\vartheta_i$ by
\beq\label{vartheta}
\vartheta_{i}=\begin{cases}
\theta_i, & \text{if }\tau i\ne i,\\
\max\{ \theta_j\mid j\in\I \text{ and }c_{ij}\ne 0\},& \text{if }\tau i= i.
\end{cases}
\eeq
Denote $\mathbf w_i=\langle \lambda,\alpha_i\rangle$. We also set
\beq
\label{varsigma}
\fkw_i =
\begin{cases} 
\bw_i, & \text{if } \tau i \neq i \\ 
\lfloor\tfrac{1}{2} \bw_i \rfloor, & \text{if } \tau i = i,
\end{cases}
\qquad 
\varsigma_i =
\begin{cases} 
0, & \text{if } \tau i \neq i \\ 
\delta_{\overline{\bw}_i, \bar{1}}, & \text{if } \tau i = i. 
\end{cases}
\eeq

Note that 
\[
\fkv_{\tau i}=\fkv_i,\qquad \fkw_{\tau i}=\fkw_i, \qquad \text{for all }i\in \I. 
\]
Recall $C=(c_{ij})$ is the Cartan matrix. Throughout the paper, we shall impose the following fundamental parity condition on the dimension vector $\bv =(\bv_i)_{i\in\I}$; see \eqref{ell}--\eqref{ell_theta}:
\beq\label{parity}
c_{ij}\theta_i\theta_j=0, \qquad \text{ for } i\neq j \in \I,
\eeq
that is, at least one of $\bv_i$ and $\bv_j$ is even when $i\neq j\in \I$ are connected. 
This condition is needed in this section and also turns out to be required for geometric constructions in Section \ref{sec:islices}; see, e.g., Theorem \ref{thm:ctgklo}. 

\begin{rem}  \label{rem:parity}
  Let $i\in \I$. If $\theta_i=1$, then the evenness of $\mu$ and the parity condition \eqref{parity} imply that $\mathbf w_i =\langle \lambda,\alpha_i\rangle$ is even and hence $\varsigma_i=0$. 
\end{rem}

Let $\bm z:=(z_{i,s})_{i\in\iI,1\lle s\lle \fkw_i}$ be formal variables and denote the polynomial ring
\[
\bC[\bm z]=\bC[z_{i,s}]_{i\in\iI,1\lle s\lle \fkw_i}
\]
and define the new $\C$-algebra
\[
\Yt_\mu(\g)[\bm z]:=\Yt_\mu(\g)\otimes \bC[\bm z],
\]
with new central elements $z_{i,s}$.
Consider the $\bC$-algebra
\begin{align} \label{A:generators}
\mc A:=\bC[\bm z]\langle w_{i,r},\dfo_{i,r}^{\pm 1}, (w_{i,r}\pm w_{i,r'}+m)^{-1},(w_{i,r}+\hf m)^{-1}\rangle_{i \in \iI,1\lle r\ne r'\lle \fkv_i,m\in\bZ},
\end{align}
subject to the relations
\begin{align} \label{A:relations}
[ \dfo_{i,r}^{\pm 1},w_{j,r'}] =\pm\delta_{ij}\delta_{r,r'}\dfo_{i,r}^{\pm 1},\qquad [w_{i,r},w_{j,r'}]=[\dfo_{i,r},\dfo_{j,r'}]=0,\qquad \dfo_{i,r}^{\pm 1}\dfo_{i,r}^{\mp 1}=1. 
\end{align}
It is convenient to extend the notation $z_{i,s}$, $w_{i,r}$, and $\dfo_{i,r}$ to all $i\in \I$ as follows. First set 
\beq\label{signzw2}
z_{\tau i,s}=-z_{i,s}, \qquad \text{ for } i\in \I_1 \text{ and } 1\lle s\lle \fkw_i. 
\eeq
We further set
\begin{align}
 \label{signzw}
 w_{\tau i,r} :=-w_{i,\fkv_i+1-r},
\qquad
\dfo_{\tau i,r} :=\dfo_{i,\fkv_i+1-r}^{-1}, 
 \qquad \text{ for } i\in \I_1 \text{ and } 1\lle r\lle \fkv_i.
\end{align}

Given a monic polynomial $f(u)$ in $u$, we define
\begin{align}  \label{eq:f^-}
f^-(u):=(-1)^{\deg f}f(-u)
\end{align}
to be the monic polynomial whose roots are the opposite of the roots of $f(u)$.

For each $i\in \I_0$, define
\begin{align}  
W_i(u)&=\prod_{r=1}^{\fkv_i}(u- w_{i,r}),\qquad\quad~~~ Z_i(u)=\prod_{s=1}^{\fkw_i}(u-z_{i,s}),
\label{WiZi} \\
\W_i(u)&=u^{\theta_i}\prod_{r=1}^{\fkv_i}(u^2- w^2_{i,r}),\qquad \Z_i(u)=u^{\varsigma_i}\prod_{s=1}^{\fkw_i}(u^2-z^2_{i,s}).
\label{WiZibold}
\end{align}
Then we have $\deg \W_i(u) = \bv_i$ and $\W_i(u) = \W_i^-(u)$,  and similarly $\deg \Z_i(u) = \mathbf w_i$ and $\Z_i^-(u) = \Z_i(u)$. We also define 
\begin{align} \label{Wir}
\W_i^\circ(u)= \prod_{r=1}^{\fkv_i}(u^2- w^2_{i,r}),\quad\W_{i,r}(u)=u^{\theta_i}(u+w_{i,r})\prod_{s=1,s\ne r}^{\fkv_i}(u^2- w^2_{i,s}).
\end{align}
Introduce
\begin{align} \label{WiZibar}
\overline W^-_i(u):=u^{\theta_i}W^-_i(u),\qquad \overline Z^-_i(u):=u^{\varsigma_i}Z^-_i(u).
\end{align}
For simplicity, set
\beq \label{varkappa}
\varkappa(u)=1-\frac{1}{2u},\qquad \bm\varkappa(u)=1-\frac{1}{4u^2}.
\eeq

For each $i\in \I_1$, we fix a choice of $\zeta_i\in \bN$ such that $1\lle \zeta_i\lle \fkv_i$ and extend it to $i\in \I_1\cup\I_{-1}$ by 
\begin{align*} 
\zeta_{\tau i}=\fkv_i-\zeta_{i}.
\end{align*}

For $i\in \I_1\cup \I_{-1}$, set
\begin{align}
\label{WiZi_qs}
\begin{split}
W_i(u)=\prod_{r=1}^{\zeta_i}(u- w_{i,r}), &\qquad 
\W_i(u)=\prod_{r=1}^{\fkv_i}(u- w_{i,r}),
\\
\Z_i(u)=\prod_{s=1}^{\fkw_i}(u-z_{i,s}),
& \qquad
\W_{i,r}(u) =\prod_{s=1,s\ne r}^{\fkv_i}(u- w_{i,s}).
\end{split}
\end{align}
It follows by \eqref{signzw2} and \eqref{signzw} that for $i\in \I_1\cup \I_{-1}$ we have
\[
\W_i(u)=W_i(u)W_{\tau i}^-(u),\qquad 
\W_{\tau i}(u)=\W_i^-(u),\qquad 
\Z_{\tau i}(u)=\Z_i^-(u).
\]
We pick a monic polynomial $Z_i(u)$, for each $i\in \I_1\cup \I_{-1}$, such that
\[
\Z_i(u)=Z_i(u)Z_{\tau i}^-(u)=(-1)^{\deg Z_{\tau i}}Z_i(u)Z_{\tau i}(-u).
\]

\subsubsection{The iGKLO representationse}
\label{ssec:GKLOqs}

Recall the polynomials $\W_i(u)$ and $\Z_i(u)$ from \S\ref{ssec:quantumtorus}.
Fix an arbitrary orientation of the diagram $\I$ such that for each $i\in \I$ with $i\ne \tau i$, if $i\to j$, then $\tau j\to \tau i$, or if $j\to i$, then $\tau i\to \tau j$. Let
\beq\label{s-def}
\wp_{i}=\begin{cases}
1, & \text{if }~i\leftarrow \tau i,\\
-1, & \text{if }~i\rightarrow \tau i,\\
0, & \text{if }~c_{i,\tau i}=0, 2.
\end{cases}
\eeq

\begin{thm} \cite[Theorem 3.6]{LWW25sty}
\label{thm:GKLOquasisplit}
Let $(\I,\tau)$ be any quasi-split Satake diagram. 
Let $\la$ be a dominant $\tau$-invariant coweight and $\mu$ be an even spherical coweight subject to the constraint \eqref{parity} such that $\la\gge \mu$. 
Then there exists a homomorphism 
\[
\Phi_{\mu}^\la:\Yt_\mu[\bm z]\longrightarrow \mc A
\]
such that (see \eqref{ell}--\eqref{varsigma} and \eqref{WiZi}--\eqref{WiZi_qs} for notations)
\begin{align*}
H_i(u)&\mapsto \Big(1+\frac{\wp_i}{4u}\Big) \frac{\bm\varkappa(u)^{\vartheta_i}\Z_i(u)}{\W_{i}(u-\tfrac{1}2)\W_{i}(u+\tfrac{1}2)}\prod_{j\leftrightarrow i}\W_j(u),\qquad  \text{ for }\ i\in \I,
\\
B_i(u)&\mapsto -\sum_{r=1}^{\zeta_i}\frac{Z_i(w_{i,r}-\tfrac12)}{(u+\tfrac{1}{2}-w_{i,r})\W_{i,r}(w_{i,r})}\prod_{j\rightarrow i}\W_{j}(w_{i,r}-\tfrac12)\eth_{i,r}^{-1}\\
&\quad\  - \sum_{r=1}^{\zeta_{\tau i}}\frac{Z_i^-(w_{\tau i,r}+\tfrac12)}{(u+\tfrac{1}{2}+w_{\tau i,r})\W_{\tau i,r}( w_{\tau i,r})}\prod_{\tau j\leftarrow \tau i}\W_{\tau j}(w_{\tau i,r}+\tfrac12)\eth_{\tau i,r},
\qquad \text{ for }\  i\in \I\setminus\I_0,
\end{align*}
and 
\begin{align*}
B_i(u)&\mapsto -\sum_{r=1}^{\fkv_i}\frac{\varkappa(w_{i,r})^{-\vartheta_i}Z_i(w_{i,r}-\tfrac12)}{(u+\tfrac{1}{2}-w_{i,r})\W_{i,r}(w_{i,r})}\prod_{j\rightarrow i}\W_{j}(w_{i,r}-\tfrac12)\prod_{\substack{j\leftarrow i\\j\in\I_0}}\overline W_{j}^-(w_{i,r}-\tfrac12)\dfo_{i,r}^{-1}\\
&~\quad- \sum_{r=1}^{\fkv_i}\frac{\varkappa(-w_{i,r})^{-\vartheta_i}\overline Z_i^-(w_{i,r}+\tfrac12)}{(u+\tfrac{1}{2}+w_{i,r})\W_{i,r}(w_{i,r})}\prod_{\substack{j\rightarrow i\\ j\in \I_{\pm 1}}}\W_{j}(w_{i,r}+\tfrac12)\prod_{\substack{j\leftarrow i\\ j\in\I_0}}W_{j}(w_{i,r}+\tfrac12)\dfo_{i,r}
\\
&~\quad+\sqrt{(-1)^{\fkw_i+\sum_{i\leftrightarrow j\in \iI}\fkv_j}} \frac{\theta_i Z_i(0)}{u\W_i^\circ(\tfrac12)}\prod_{j\leftrightarrow i}W_{j}(0), \qquad\qquad\qquad \text{ for }\ i\in \I_0.
\end{align*}
\end{thm}
We refer to $\Phi_\mu^\la$ as the iGKLO homomorphisms.

\subsection{Truncated shifted iYangians}
\label{ssec:commsub}

The iGKLO homomorphisms allow us to formulate a new family of algebras. 

\begin{dfn} \cite[Definition 3.8]{LWW25sty}
\label{dfn:truncated stY}
Let $\la$ be a dominant $\tau$-invariant coweight and $\mu$ be an even spherical coweight with $\lambda \gge \mu$. The truncated shifted iYangian (TSTY), denoted $\Yt_\mu^\la$, is the $\C$-algebra given by the image of the iGKLO homomorphism $\Phi_{\mu}^\la:\Yt_\mu[\bm z]\rightarrow \mc A$. 
\end{dfn}

Define a ``Cartan" series $\mathsf{A}_i(u)$ in $\Yt_\mu[\bm z][\![u^{-1}]\!]$, for $i\in \I$, by
\beq\label{GKLO-A}
H_i(u)=\Big(1+\frac{\wp_{i}}{4u}\Big)\frac{\bm\varkappa(u)^{\vartheta_i}\Z_i(u)\prod_{j\leftrightarrow i}u^{\bv_j}}{(u^2-\frac14)^{\bv_i}}\frac{\prod_{j\leftrightarrow i} \mathsf A_j(u)}{\mathsf A_i(u-\frac12)\mathsf A_i(u+\frac12)},
\eeq
where $\wp_i$ is defined in \eqref{s-def}.
Expanding the series $\mathsf A_i(u)$ gives us a family of GKLO-type ``Cartan" elements $\mathsf{A}_i^{(r)}$ in $\Yt_\mu[\bm z]$, for $r>0, i\in \I$:
\beq\label{A-coeff}
\mathsf A_i(u)=1+\sum_{r>0}\mathsf A_i^{(r)}u^{-r}.
\eeq

Consider the subalgebra of $\Yt_\mu^\la$ generated over $\bC[\bm z]$ by the coefficients of all of the series $\mathsf A_i(u)$. (Equivalently, this subalgebra is generated over $\bC[\bm z]$ by the coefficients of the series $H_i(u)$.)  This commutative subalgebra is a polynomial ring, having the following algebraically independent generators over $\bC[ \bm z]$ (cf. \cite{LWW25sty}):
\be
    \{ \mathsf A_i^{(2r)} : i \in \I_0, ~ 1 \lle 2r \lle \bv_i \} \cup \{\mathsf A_i^{(r)} : i \in \I_1, 1 \lle r \lle \bv_i \}.
\ee
We call this the \emph{Gelfand-Tsetlin subalgebra} of $\Yt_\mu^\la$, and in many cases it is a maximal commutative subalgebra of $\Yt_\mu^\la$. 


Choose a coweight $\mu_1$ such that $\mu = \mu_1 + \tau \mu_1$, and recall the corresponding filtration $F_{\mu_1}^\bullet \Yt_\mu$ constructed in \S \ref{ssec:filtrations on stY}.  This extends to a filtration $F_{\mu_1}^\bullet \Yt_\mu[\bm z]$, by putting all variables $z_{i,s}$ in degree 1.  
Since $\Yt_\mu^\la$ is a quotient of $\Yt_\mu [ \bm z]$, it naturally inherits the quotient filtration.  By abuse of notation we will denote this filtration by $F_{\mu_1}^\bullet \Yt_\mu^\la$.  Then there is an epimorphism of associated graded algebras
\begin{equation}
    \label{eq:quofilt}
    \gr^{F_{\mu_1}^\bullet} \Yt_\mu[\bm z] \twoheadrightarrow \gr^{F_{\mu_1}^\bullet} \Yt_\mu^\la,
\end{equation}
and it follows from Proposition \ref{prop:algfilt} that $\gr^{F_{\mu_1}^\bullet} \Yt_\mu^\la$ is commutative. We will study the underlying geometry for this commutative algebra in \S \ref{ssec:GKLOfpl}.

\section{Twisted Yangians via quantum duality principle and Dirac reduction}
\label{sec:tYdualityPrinciple}

In this section, we establish a general relationship between twisted Yangians and fixed point loci in loop groups by a Poisson involution $\sigma$, by applying the quantum duality principle and Dirac reduction.  This serves as motivation for our constructions and results in subsequent sections, but it may be of independent interest.  

\subsection{Yangians and quantum duality principle}
\label{ssec: Yangians and loop groups}

Let $G$ be an affine algebraic group over $\bC$, with Lie algebra $\g$.  Then we may consider the corresponding loop groups such as $G(\!(z^{-1})\!), G[z]$ and $G[\![z^{-1}]\!]$, whose Lie algebras are $\g(\!(z^{-1})\!)$, $\g[z]$ and $\g[\![z^{-1}]\!]$, respectively. More precisely, these Lie algebras are defined via base change, such as $\g(\!(z^{-1})\!) = \g \otimes_\bC \bC(\!(z^{-1})\!)$, while the loop groups are defined as (ind-)schemes via their functors of points. For example, $G(\!(z^{-1})\!)$ represents the functor sending each $\bC$--algebra $R$ to the group $G\big( R(\!(z^{-1})\!) \big)$. Finally, we will denote by $G_1[\![z^{-1}]\!]$ the kernel of the evaluation $G[\![z^{-1}]\!] \rightarrow G$ at $z^{-1} = 0$.  Then $G_1[\![z^{-1}]\!]$ is an affine group scheme with Lie algebra $z^{-1}\mathfrak{g}[\![z^{-1}]\!]$, and the exponential map defines an isomorphism of affine schemes
\begin{equation}
\label{eq: exp map}
\exp : z^{-1}\g[\![z^{-1}]\!] \xrightarrow{\sim} G_1[\![z^{-1}]\!].
\end{equation}

Suppose that $\g$ is equipped with a non-degenerate invariant symmetric bilinear form $(\cdot,\cdot)_\g$. In the standard way \cite[Example 1.3.9]{CP94} we may form a Manin triple $\big(\g(\!(z^{-1})\!), \g[z], z^{-1}\g[\![z^{-1}]\!]\big)$, equipped with the non-degenerate symmetric invariant bilinear form
\begin{equation}
    \label{eq: loop form}
    ( x,y ) = - \operatorname{Res}_{z=0} (x, y)_\g, \qquad x,y \in \g(\!(z^{-1})\!).
\end{equation}
This endows $\g(\!(z^{-1})\!)$ with a Lie bialgebra structure, with $\g[z]$ as a sub-Lie bialgebra. The Lie algebra $\g[z]$ is $\bZ$--graded with respect to degree in $z$, and its Lie cobracket is homogeneous of degree $-1$. Its dual $z^{-1}\g[\![z^{-1}]\!]$ is also naturally a Lie bialgebra, corresponding to the Manin triple $\big(\g(\!(z^{-1})\!), z^{-1}\g[\![z^{-1}]\!], \g[z]\big)$ with respect to the same bilinear form \eqref{eq: loop form}. The groups $G(\!(z^{-1})\!)$, $G[z]$ and $G_1[\![z^{-1}]\!]$ each carry corresponding Poisson structures.

We now turn to quantization.  Consider $\bC[\hbar]$ as a graded ring with $\deg \hbar = 1$. 

\begin{dfn} \label{def:Yangian:axioms}
 A \emph{Yangian} for $(\g, (\cdot,\cdot)_\g)$ is a $\bZ$-graded Hopf algebra $\U_\hbar(\g[z])$ over $\bC[\hbar]$, which is free as a graded module over $\bC[\hbar]$ and quantizes the Lie bialgebra $\g[z]$ in the sense that 
 \begin{equation}
     \label{eq: yangian quant}
     \U_\hbar(\g[z]) / \hbar \U_\hbar(\g[z]) \cong U(\g[z])
 \end{equation}
 as graded co-Poisson-Hopf algebras. 
\end{dfn} 

For each integer $n \gge 0$ define a map $\Delta^n: \U_\hbar(\g[z]) \rightarrow \U_\hbar(\g[z])^{\otimes n}$, where $\Delta^0 = \varepsilon$ is the counit, $\Delta^1 = \operatorname{Id}$, and $\Delta^n = (\Delta\otimes \operatorname{Id}^{\otimes n-2})\circ \Delta^{n-1}$ for $n \gge 2$. The \emph{Quantum Duality Principle} \cite{Dr87b,G02} (also called \emph{Drinfeld-Gavarini Duality}) defines a subalgebra:
\begin{equation}
    \U_\hbar(\g[z])' = \left\{ a \in \U_\hbar(\g[z]) \mid (\operatorname{Id}-\varepsilon)^{\otimes n} \circ \Delta^n(a) \in \hbar^n \U_\hbar(\g[z])^{\otimes n} \text{ for all } n \gge 0 \right\}.
\end{equation}
Then $\U_\hbar(\g[z])'$ is a graded sub-Hopf algebra over $\bC[\hbar]$, which is almost-commutative in the sense that for any $a,b\in \U_\hbar(\g[z])'$ we have $[a,b] \in \hbar \U_\hbar(\g[z])'$.  Following \cite[\S A(ii)]{FT19b}, it will be useful to make the following technical assumption: 
\begin{equation}
    \label{eq: QDP1}
    \begin{array}{c}
    \text{There is a subset } \{x_\gamma\} \subset \U_\hbar(\g[z]) \text{ consisting of homogeneous elements,} \\ \text{which lifts  a basis } \{ \overline{x_\gamma} \} \text{ for } \g[z], \text{ and such that all } \hbar x_\gamma \in \U_\hbar(\g[z])'.
    \end{array}
\end{equation} 
Under this assumption, any homogeneous basis for $\g[z]$ admits a similar  lift.  Moreover, an induction on degree shows that $\U_\hbar(\g[z])$ has a PBW basis over $\bC[\hbar]$ consisting of ordered monomials in the set $\{x_\gamma\}$, with respect to any given total order.

\begin{rem}
    If $\g$ is a simple Lie algebra then the Yangian $\U_\hbar(\g[z])$ is unique up to isomorphism \cite[Theorem 2]{Dr85}, and thus Assumption (\ref{eq: QDP1}) holds by \cite[Lemma A.20]{FT19b}, see also \cite[\S 5.2]{W24}.  If $\g = \gl_n$ and $\U_\hbar(\g[z])$ is the usual Yangian defined in RTT presentation as in  \cite{Mol07}, then Assumption (\ref{eq: QDP1}) holds by \cite[Proposition A.25]{FT19b}.  
    
    In fact, based on related results for formal quantizations of finite-dimensional Lie bialgebras from \cite{G02}, it seems likely that Assumption (\ref{eq: QDP1}) holds for \emph{any} Lie algebra $\g$ with a non-degenerate invariant symmetric bilinear form and corresponding Yangian $\U_\hbar(\g[z])$. We will not address this question here.
\end{rem}

We have the following generalization of \cite[Theorem 3.9]{KWWY14} and \cite[Remark 1.12]{S16}:

\begin{prop}
\label{prop: QDP Yangian1}
Let $\U_\hbar(\g[z])$ be a Yangian satisfying Assumption (\ref{eq: QDP1}).  Then:
\begin{enumerate}
\item $\U_\hbar(\g[z])'$ is free over $\bC[\hbar]$, with a PBW basis consisting of ordered monomials in the elements $\{\hbar x_\gamma\}$ with respect to any given total order.

\item $\U_\hbar(\g[z])'$ is a quantization of the Poisson algebraic group $G_1[\![z^{-1}]\!]$.  In other words, there is an isomorphism of Poisson-Hopf algebras
$$
\U_\hbar(\g[z])' / \hbar \U_\hbar(\g[z])' \ \cong \ \bC\big[G_1[\![z^{-1}]\!] \big].
$$
\end{enumerate}
In particular, these results hold for Yangian $\U_\hbar(\g[z])$ when $\g$ is simple, and for the (RTT) Yangian of $\g = \gl_n.$
\end{prop}
\begin{proof}
    Part (1) is an application of \cite[Proposition A.5]{FT19b}.    
    
    For  Part (2), first note that the Hopf algebras  $\bC[ G_1[\![z^{-1}]\!]]$ and $U(z^{-1}\g[z^{-1}])$ are graded dual to one another, via the perfect Hopf pairing defined by  $(X, f) = (X\cdot f)(e)$ for any $X \in U(z^{-1}\g[z^{-1}])$ and $f \in \bC\big[G_1[\![z^{-1}]\!]\big]$. Here $X\cdot f$ denotes the action of $X$ thought of as a left-invariant differential operator and $e \in G_1[\![z^{-1}]\!]$ denotes the identity element. Meanwhile, the Quantum Duality Principle provides a non-degenerate graded Hopf pairing between $U(z^{-1}\g[z^{-1}])$ and the classical limit  $\U_\hbar(\g[z])' / \hbar \U_\hbar(\g[z])'$. As in \cite[Corollary~ 3.4]{KWWY14}, this pairing is  uniquely determined by the property that
    \begin{equation}
    \label{eq: QDP pairing}
    \langle y, \hbar x_\gamma  + \hbar \U_\hbar(\g[z])' \rangle = (y, \overline{x_\gamma})    
    \end{equation}
    for all $y \in z^{-1} \g[z^{-1}]$ and generators $\hbar x_\gamma \in U_\hbar(\g[z])'$.  This pairing defines an injective map  $ \U_\hbar(\g[z])' / \hbar \U_\hbar(\g[z])' \rightarrow \bC\big[G_1[\![z^{-1}]\!]\big]$ of graded Hopf algebras, which is an isomorphism (cf.~the proof of \cite[Theorem~ 3.9]{KWWY14}) since both have Hilbert series $\prod_{r\gge 1} (1-q^r)^{-\dim \g}$.
    Finally, the Quantum Duality Principle \cite{G02} identifies the Poisson structures.
\end{proof}

\begin{rem}
\label{rem: filtered 1}
    We may also specialize $\hbar = 1$ and work with the filtered $\bC$--algebra $\U_{\hbar=1}(\g[z])'$. Then $\bC\big[G_1[\![z^{-1}]\!]\big]$ is the associated graded algebra, and we recover $\U_\hbar(\g[z])'$ via the Rees algebra construction, see for example \cite[\S A(iii)]{FT19b}.
\end{rem}

\subsection{Dirac reduction and fixed points}
\label{ssec: dirac reduction and fixed points}

Let $X = \operatorname{Spec} A$ be an affine Poisson scheme over $\bC$.  In other words, $A$ is a commutative algebra over $\bC$ which is equipped with a Poisson bracket.  Let $\sigma$ be an algebra involution of $A$, which respects the Poisson structure; one may equivalently think of $\sigma$ as defining an action of the group $\bZ/2\bZ$ on $X$ by Poisson automorphisms.  The corresponding \emph{fixed point scheme} 
\begin{equation}
X^\sigma = X^{\bZ/2\bZ} = \operatorname{Spec} R(A, \sigma),
\end{equation}
is defined via the quotient algebra $R(A, \sigma) = A / \langle \sigma(a) - a : a \in A \rangle$. In particular $X^\sigma \subseteq X$ is a closed subscheme.
\begin{rem}
    \label{rem:fixedfunctorofpoints}
    Recall that schemes can also be described by their functors of points (see e.g.~\cite[\S 4.1]{GW10}). For any test scheme $S$, an element  $x \in X(S)$ is  simply a morphism $x: S \rightarrow X$. The functor corresponding to $X^\sigma$ has a very simple description \cite{F73}:
    $$
    X^\sigma(S)  = \big\{ x \in X(S) \mid \sigma (x) = x \big\}.
    $$
    Here $\sigma(x)$ denotes the composition $S \xrightarrow{x} X \xrightarrow{\sigma} X$.     In particular, this provides a definition of $X^\sigma$ even when $X$ is not affine.
\end{rem} 

\begin{lem}[\mbox{\cite[Prop.~3.4]{E92}}]
    \label{lem:fixedsm}
    If $X$ is smooth, then  $X^\sigma$ is smooth.
\end{lem}

The fixed-point scheme $X^\sigma$ has a natural Poisson structure defined via \emph{Dirac reduction}, see \cite[\S 4.1]{Xu03} or \cite[\S 2.1]{T23}.  To describe this structure, let $A = A(1) \oplus A(-1)$ denote the eigenspace decomposition for the action of $\sigma$.  Then the inclusion $A(1) \subset A$ induces a natural  identification
\begin{equation}
R(A, \sigma) \equiv A / \langle A(-1) \rangle \cong A(1) / A(-1)^2.
\end{equation}
It is easy to see that $A(1) \subset A$ is a Poisson subalgebra, and that $A(-1)^2 \subset A(1)$ is a Poisson ideal. In this way the quotient $A(1)/A(-1)^2$ inherits a Poisson structure.  

\begin{rem}
    \label{rem: Dirac reduction Poisson}
    Explicitly, given two functions $f,g$ on $X^\sigma$, their Poisson bracket is defined by choosing $\sigma$--invariant extensions $\tilde{f}, \tilde{g}$ to functions on $X$, computing the Poisson bracket $\{\tilde{f}, \tilde{g}\}$ on $X$, and then restricting the result to $X^\sigma$. 
\end{rem}

\begin{lem}
\label{lem: Dirac reduction leaves}
Assume that $X$ has finitely many symplectic leaves.  Then $X^\sigma$ also has finitely many symplectic leaves, i.e., the connected components of the fixed point loci $\mathcal{S}^\sigma$, for those leaves $\mathcal{S}$ of $X$ stable under the action of $\sigma$.
\end{lem}
\begin{proof}
Since $X$ has finitely many symplectic leaves, each leaf  $\mathcal{S} \subseteq X$ is a smooth locally-closed subvariety by \cite[Theorem 3.7]{BG03}. Its image $\sigma(\mathcal{S})$ is also a leaf, and thus the involution $\sigma$ permutes the set of all symplectic leaves.  It follows that the fixed point locus $X^\sigma$ is the union of fixed point loci $\mathcal{S}^\sigma$, this union being taken over the set of all leaves $\mathcal{S}$ satisfying $\sigma(\mathcal{S}) = \mathcal{S}$. For each leaf satisfying $\sigma(\mathcal{S}) = \mathcal{S}$, since $\mathcal{S}$ is smooth,  $\mathcal{S}^\sigma$ is also smooth by Lemma \ref{lem:fixedsm}. Moreover the symplectic leaves of $\mathcal{S}^\sigma$ are precisely its (finitely many) connected components, by \cite[Theorem~ 2.3]{Xu03}. Since the inclusions $\mathcal{S}^\sigma \subset X^\sigma$ are Poisson by construction, the claim follows.
\end{proof}

We conclude with two results for later use.  For simplicity we ignore Poisson structures.
\begin{lem}[\mbox{\cite[Lemma 2.2]{T23}.}]
    \label{lem:symdirac}
    Let $V$ be a vector space equipped with an involution $\sigma$, with corresponding eigenspaces $V = V(1) \oplus V(-1)$. Extend $\sigma$ to an involution of the symmetric algebra $S(V)$ in the natural way.  If $U\subset V$ is a subspace complementary to $V(-1)$,  then the  composition 
    $$
    S(U) \hookrightarrow S(V) \twoheadrightarrow R\big( S(V), \sigma\big)
    $$
     is an isomorphism of algebras. In particular, this applies to $U = V(1).$
\end{lem}

\begin{lem}
    \label{lem:fixedfp}
    Let $X, Y, Z$ be schemes, each equipped with an involution $\sigma$, and let $\alpha:X\rightarrow Z$ and $\beta:Y\rightarrow Z$ be $\sigma$--equivariant maps.  Then there is an equality of schemes:
    $$
    (X \times_Z Y)^\sigma = X^\sigma \times_{Z^\sigma} Y^\sigma.
    $$
\end{lem}
\begin{proof}
    Recall the description of a fiber product as a functor of points, see e.g.~\cite[Chapter 4]{GW10}: 
    $$
    (X \times_Z Y)(S) = \left\{ (x,y) \in X(S) \times Y(S) : \alpha(x) = \beta(y) \in Z(S) \right\}.
    $$
    Using Remark \ref{rem:fixedfunctorofpoints}, we see that:
    $$
    (X\times_Z Y)^\sigma(S) = \left\{ (x,y) \in X(S) \times Y(S) : \alpha(x) = \beta(y) \in Z(S), \ \sigma(x) = x, \ \sigma(y) = y \right\}.
    $$
    For  any pair $(x,y)$ on the right side, since $\alpha,\beta$ are $\sigma$-equivariant the  image $z = \alpha(x) = \beta(y)$ satisfies $\sigma(z) = z$, so $z \in Z^\sigma(S)$.  Therefore $(X^\sigma \times_{Z^\sigma} Y^\sigma)(S) = (X\times_Z Y)^\sigma(S)$.
\end{proof}

\subsection{Loop symmetric spaces} 
\label{ssec:loop sym spaces}

Now let  $\omega$ be an involution of $G$, and  denote the corresponding involution of $\g$ by the same notation. We will assume that the bilinear form $(\cdot,\cdot)_\g$ on $\g$ is $\omega$-invariant:
\begin{equation}
    (\omega(x), \omega(y))_\g = (x,y)_\g, \qquad \forall x,y \in \g.
\end{equation}
Extend $\omega$ to an involution, again denoted by $\omega$, of $\g(\!(z^{-1})\!)$ via
\begin{equation}
    \omega( x \otimes z^k)  = \omega(x) \otimes (-z)^k.
\end{equation}
This involution preserves the subalgebras $\g[z]$ and $z^{-1}\g[\![z^{-1}]\!]$, and there is a natural corresponding involution $\omega$ of the group $G(\!(z^{-1})\!)$, preserving its subgroups $G[z]$ and $G_1[\![z^{-1}]\!]$.  Viewed as an involution of any of these groups $\omega$ is anti-Poisson, since $\omega$ is anti-symmetric for the bilinear form (\ref{eq: loop form}) on $\g(\!(z^{-1})\!)$.
\begin{lem}
The fixed-point subgroup $G_1[\![z^{-1}]\!]^\omega$ is coisotropic, i.e., its defining ideal $J \subset \bC\big[G_1[\![z^{-1}]\!]\big]$ satisfies $\{J, J\} \subseteq J$. Consequently, the quotient 
$$
G_1[\![z^{-1}]\!] / G_1[\![z^{-1}]\!]^\omega
$$  
has the structure of a Poisson homogeneous space.
\end{lem}

\begin{proof}
    On the Lie algebra level the orthogonal complement of $(z^{-1}\g[\![z^{-1}]\!])^\omega$ is $\g[z]^\omega \subset \g[z]$. Since this orthogonal complement is a Lie subalgebra, it follows that the subgroup $G_1[\![z^{-1}]\!]^\omega$ is coisotropic and the corresponding quotient space is a Poisson homogeneous space, see e.g.~\cite[\S 1.3]{CG06}.  
\end{proof}

\begin{rem}
We may think of $G_1[\![z^{-1}]\!] / G_1[\![z^{-1}]\!]^\omega$ as a symmetric space for $G_1[\![z^{-1}]\!]$, and thus an infinite-dimensional \emph{Poisson symmetric space} in the sense of \cite[\S 5.3]{Xu03} or \cite[\S 4]{F94}. 
\end{rem}

This homogeneous space admits a useful alternative description as a fixed point scheme. To this end, define a map on the group $G(\!(z^{-1})\!)$ by
\begin{equation} 
\label{eq:sigma}
\sigma = \operatorname{inv} \circ \  \omega = \omega \circ \operatorname{inv},
\end{equation} 
where $\operatorname{inv}$ denotes the inverse map. Clearly $\sigma$ is an involution and a group anti-homomorphism. We use the same notation for the restriction of $\sigma$ to the subgroups $G[z]$ and $G_1[\![z^{-1}]\!]$.  

\begin{lem}
    The map $\sigma$ is a Poisson involution of $G(\!(z^{-1})\!), G[z]$ and $ G_1[\![z^{-1}]\!]$.
\end{lem}

\begin{proof}
    Follows by definition \eqref{eq:sigma} since both $\omega$ and $\operatorname{inv}$ are anti-Poisson maps.
\end{proof}

We will be interested in the fixed-point locus
\begin{equation}
    G_1[\![z^{-1}]\!]^\sigma =  \left\{ g \in G_1[\![z^{-1}]\!] \mid \sigma(g) = g\right\}= \left\{ g \in G_1[\![z^{-1}]\!] \mid  \omega(g) = g^{-1} \right\}.
\end{equation}
Since $\sigma$ is a Poisson involution, this fixed-point locus carries a natural Poisson structure via Dirac reduction as in \S \ref{ssec: dirac reduction and fixed points}. If we take this Poisson structure and rescale it by 2, we obtain another Poisson structure. We will always consider this doubled Poisson structure on $G_1[\![z^{-1}]\!]^\sigma$, because of the following result. Let $e\in G_1[\![z^{-1}]\!]$ denote the identity element. 

\begin{prop} 
\label{prop:loopSymSpace}
    \begin{enumerate}
        \item Group multiplication provides an isomorphism of affine schemes
        $$
        G_1[\![z^{-1}]\!]^\sigma \times G_1[\![z^{-1}]\!]^\omega \rightarrow G_1[\![z^{-1}]\!].
        $$
        
        \item There is a transitive Poisson action of $G_1[\![z^{-1}]\!]$ on $G_1[\![z^{-1}]\!]^\sigma$ (with its doubled Poisson structure) defined by $$g\cdot p = g p \sigma(g) = g p \omega(g)^{-1}.$$
        \item There is an induced  isomorphism of Poisson homogeneous  spaces
        $$
        G_1[\![z^{-1}]\!]/ G_1[\![z^{-1}]\!]^\omega \xrightarrow{\ \ \sim\ \ }  G_1[\![z^{-1}]\!]^\sigma, \qquad  [g] \mapsto g\cdot e = g \sigma(g).
        $$
        In particular, there is a corresponding isomorphism
        $$
        \bC\big[ G_1[\![z^{-1}]\!]^\sigma\big] \xrightarrow{\ \ \sim\ \ } \bC\big[ G_1[\![z^{-1}]\!]\big]^{G_1[\![z^{-1}]\!]^\omega}
        $$
        of graded Poisson algebras with left co-actions of $\bC\big[G_1[\![z^{-1}]\!]\big]$.
    \end{enumerate}
\end{prop}

\begin{proof}  
We remark that every $g\in G_1[\![z^{-1}]\!]$ admits a unique square root $h \in G_1[\![z^{-1}]\!]$ satisfying $g = h^2$, as can be seen for example via the exponential map (\ref{eq: exp map}).  Moreover if $g \in G_1[\![z^{-1}]\!]^\omega$ then $h \in G_1[\![z^{-1}]\!]^\omega$, and similarly if $g\in G_1[\![z^{-1}]\!]^\sigma$ then $h \in G_1[\![z^{-1}]\!]^\sigma.$

    Given $g \in G_1[\![z^{-1}]\!]$ we can thus uniquely find $p \in G_1[\![z^{-1}]\!]^\sigma$ such that $p^2 = g \sigma(g)$.  Defining $k = p^{-1} g$, we have $k \in G_1[\![z^{-1}]\!]^\omega$ and $g = pk$, and this decomposition is easily seen to be unique. Since these constructions are all algebraic, as is the multiplication map, this proves Part (1).
    
    Next, consider the group action of $G_1[\![z^{-1}]\!]$ on $G_1[\![z^{-1}]\!]^\sigma$ given by $g\cdot p = g p \sigma(g)$. If $p \in G_1[\![z^{-1}]\!]^\sigma$, then there is a unique $h \in G_1[\![z^{-1}]\!]^\sigma \subset G_1[\![z^{-1}]\!]$ such that $p = h^2$. Since
    $$
    h \cdot e = h \sigma(h)  = h^2 = p,
    $$
   this group action is transitive, and the stabilizer of the identity element $e \in G_1[\![z^{-1}]\!]^\sigma$ is precisely $G_1[\![z^{-1}]\!]^\omega$.  The action is Poisson by \cite[Theorem 5.9]{Xu03}, the application of which requires doubling the Poisson structure on $G_1[\![z^{-1}]\!]^\sigma$ coming from Dirac reduction, which proves (2).   

    On the level of sets and Poisson structures, Part (3) follows from Part (2).  Regarding the ring isomorphism in the end, first consider the simpler map $G_1[\![z^{-1}]\!]^\sigma \rightarrow G_1[\![z^{-1}]\!]/G_1[\![z^{-1}]\!]^\omega$ defined by $ p \mapsto [p]$. This is a bijection of sets by Part (1), and in fact of affine schemes, via a sequence of  ring isomorphisms:
    \begin{align*}
    \bC\big[ G_1[\![z^{-1}]\!]^\sigma\big] &\cong \bC\big[ G_1[\![z^{-1}]\!]^\sigma\big]\otimes \bC \\&= \bC\big[ G_1[\![z^{-1}]\!]^\sigma\big] \otimes \bC\big[ G_1[\![z^{-1}]\!]^\omega\big]^{G_1[\![z^{-1}]\!]^\omega} \cong  \bC\big[ G_1[\![z^{-1}]\!]\big]^{G_1[\![z^{-1}]\!]^\omega}.
    \end{align*}
    Now, composing this simpler map with the claimed isomorphism of homogeneous spaces, one obtains the map from $G_1[\![z^{-1}]\!]^\sigma$ to itself defined by $p \mapsto p^2$.  This is an automorphism of this affine scheme, which proves the desired ring isomorphism.  
\end{proof}

\subsection{Twisted Yangians as quantizations} 
\label{ssec:tYangian:symSpaces}

We now turn to quantization of this story.  Fix a Yangian $\U_\hbar(\g[z])$ for $(\g, (\cdot,\cdot)_\g)$ as in Definition \ref{def:Yangian:axioms}.
\begin{dfn}  \label{def:tYaxioms}
    We say that a graded left coideal subalgebra $\U_\hbar(\g[z]^\omega) \subset \U_\hbar(\g[z])$ over $\bC[\hbar]$ is a \emph{twisted Yangian} for $(\g, (\cdot,\cdot)_\g, \omega)$, if the following conditions hold:
    \begin{enumerate}
    \item $\U_\hbar(\g[z]^\omega)$ is free as a graded module over $\bC[\hbar]$,
    \item  $ \U_\hbar(\g[z]^\omega) \cap \hbar \U_\hbar(\g[z]) = \hbar \U_\hbar(\g[z]^\omega)$,
    \item $\U_\hbar(\g[z]^\omega) \subset \U_\hbar(\g[z])$  quantizes $\g[z]^\omega \subset \g[z]$, in the sense that the isomorphism (\ref{eq: yangian quant}) restricts to an isomorphism of graded algebras 
    $$
    \U_\hbar(\g[z]^\omega) / \hbar \U_\hbar(\g[z]^\omega) \cong U(\g[z]^\omega).
    $$
    \end{enumerate}
\end{dfn}

\begin{rem}
    In fact $\U_\hbar(\g[z]^\omega) / \hbar \U_\hbar(\g[z]^\omega) \cong U(\g[z]^\omega)$ are isomorphic as Hopf algebras. Indeed, if we lift $\overline{x} \in \g[z]^\omega$ to an element $x \in \U_\hbar(\g[z]^\omega),$ then since $\U_\hbar(\g[z])$ quantizes $U(\g[z])$ we see that 
    $$
    \Delta(x) \in\, x \otimes 1 + 1 \otimes x + \hbar \ \U_\hbar(\g[z]) \otimes \U_\hbar(\g[z]^\sigma).
    $$
    The classical limit is thus the usual coproduct on $U(\g[z]^\omega)$, and similarly for the counit and antipode.   
\end{rem}

\begin{rem}
    By imposing constraints on the coproduct $\Delta(x) \pmod{\hbar^2}$, a related (but stricter) notion of twisted Yangian is given in \cite[Definition 5.3]{BR17}.
\end{rem}

In \cite{CG06,CG14}, the quantum duality principle was extended to incorporate (coisotropic) subgroups and homogeneous spaces.  Following their approach, define
\begin{align*}
    \U_\hbar(\g[z]^\omega)' = \left\{ c \in \U_\hbar(\g[z]^\omega)  \mid (\operatorname{Id}-\varepsilon)^{\otimes n} \circ\Delta^n(c) \in \hbar^n \U_\hbar(\g[z])^{\otimes(n-1)} \otimes \U_\hbar(\g[z]^\omega), \forall  n \gge 0 \right\}.    
\end{align*}
(Notation $\U_\hbar(\g[z]^\omega)^\Lsh$ was used {\em loc. cit.}) Then $\U_\hbar(\g[z]^\omega)' \subset \U_\hbar(\g[z])'$ is a graded left coideal subalgebra over $\bC[\hbar]$.  Similarly to (\ref{eq: QDP1}), it will be useful to make the following technical assumption:
\begin{equation}
\label{eq: QDP2}
    \begin{array}{c}
    \text{There is a subset } \{y_\eta\} \subset \U_\hbar(\g[z]^\omega) \text{ consisting of homogeneous elements,} \\ \text{which lifts  a basis } \{ \overline{y_\eta} \} \text{ for } \g[z]^\omega, \text{ and such that all } \hbar y_\eta \in \U_\hbar(\g[z]^\omega)'.
    \end{array}
\end{equation} 
In this case any homogeneous basis for $\g[z]^\omega$ admits a similar lift, and $\U_\hbar(\g[z]^\omega)$ admits a PBW basis over $\bC[\hbar]$ consisting of ordered monomials in the set $\{y_\eta\}$, with respect to any total order.

\begin{thm}
\label{thm: QDP Yangian2}
    Let $\U_\hbar(\g[z]^\omega) \subset \U_\hbar(\g[z])$ be a twisted Yangian such that Assumptions (\ref{eq: QDP1}) and (\ref{eq: QDP2}) both hold.  Then:
    \begin{enumerate}
        \item $\U_\hbar(\g[z]^\omega)'$ is free over $\bC[\hbar]$, with PBW basis consisting of ordered monomials in the elements $\{\hbar y_\eta\}$.
       \item $\U_\hbar(\g[z]^\omega)' \cap \hbar \U_\hbar(\g[z])' = \hbar \U_\hbar(\g[z]^\omega)'$.
        \item $\U_\hbar(\g[z]^\omega)'$ is a quantization of the Poisson homogeneous space $ G_1[\![z^{-1}]\!]/ G_1[\![z^{-1}]\!]^\omega$.         More precisely, the isomorphism $\U_\hbar(\g[z])' / \hbar \U_\hbar(\g[z])' \cong \bC\big[G_1[\![z^{-1}]\!] \big]$ from Proposition \ref{prop: QDP Yangian1} restricts to an isomorphism of graded Poisson left coideal subalgebras
        $$
        \U_\hbar(\g[z]^\omega)' / \hbar \U_\hbar(\g[z]^\omega)' \ \cong \ \bC\big[ G_1[\![z^{-1}]\!]\big]^{G_1[\![z^{-1}]\!]^\omega}.
        $$
    \end{enumerate}
\end{thm}
    
\begin{proof}
    By Assumption (\ref{eq: QDP2}) we have lifts $\{y_\eta\}$ of some homogeneous basis $\{\overline{y_\eta}\}$ for $\g[z]^\omega$.  If we enlarge $\{\overline{y_\eta}\}$ to a homogeneous basis $\{\overline{x_\gamma}\}$ for $\g[z]$, then using Assumption (\ref{eq: QDP1}) we may also enlarge $\{y_\eta\}$ to a set of lifts $\{x_\gamma\}$. Choose compatible total orders on these sets. By Proposition~ \ref{prop: QDP Yangian1}(1), along with the PBW property for $\U_\hbar(\g[z])^\omega)$, it follows that $\U_\hbar(\g[z]^\omega) \cap \U_\hbar(\g[z])'$ has a basis consisting of ordered monomials in $\{\hbar y_\eta\}$. But by assumption these monomials all lie in the subset $\U_\hbar(\g[z]^\omega)'$.  Therefore $\U_\hbar(\g[z]^\omega)' = \U_\hbar(\g[z]^\omega) \cap \U_\hbar(\g[z])'$, and this discussion proves (1) and (2).

Part (3) formally follows from the Quantum Duality Principle from \cite{CG06, CG14}.  However, since we work in an infinite-dimensional graded situation, we give a detailed direct argument.

    First, the classical limit of $\U_\hbar(\g[z]^\omega)'$ certainly corresponds to some graded Poisson left coideal subalgebra $C \subset \bC\big[ G_1[\![z^{-1}]\!] \big]$. Let  $C^\perp \subset U(z^{-1}\g[z^{-1}])$ be its orthogonal complement, under the duality discussed in the proof of Proposition \ref{prop: QDP Yangian1}. Then $C^\perp$ is a left ideal, and a two-sided coideal. Since $\U_\hbar(\g[z]^\omega)'$ is a left coideal spanned by monomials in $\{\hbar y_\eta\}$, and the classical limits $\overline{y_\eta} \in \g[z]^\omega$ are orthogonal to $(z^{-1}\g[z^{-1}])^\omega$, using the  pairing (\ref{eq: QDP pairing}) it is not hard to see that $(z^{-1} \g[z^{-1}])^\omega \subset C^\perp$. Thus $C^\perp$ contains the left ideal $J$ generated by $(z^{-1}\g[z^{-1}])^\omega$. We claim that in fact $C^\perp = J$.
    
    Indeed, decompose $\g = \mathfrak{k}\oplus \mathfrak{p}$ into $(\pm 1)$-eigenspaces for the involution $\omega$. The quotient $U(z^{-1}\g[z^{-1}]) / J$ is isomorphic as a vector space to the symmetric algebra on the $(-1)$--eigenspace for $\omega$ acting on $z^{-1}\g[z^{-1}]$, which is $z\bC[z^2] \mathfrak{k} \oplus \bC[z^2] \mathfrak{p}$. Thus this quotient has Hilbert series
    $$
    \frac{1}{\prod_{\substack{r \gge 1, \\ r \text{ odd}}}(1-q^{-r})^{\dim \mathfrak{k}} \cdot \prod_{\substack{r \gge 1, \\ r \text{ even}}}(1-q^{-r})^{\dim \mathfrak{p}}}.
    $$
    But the PBW basis for $\U_\hbar(\g[z]^\omega)'$ from Part (1) implies that $C$ has exactly the same Hilbert series, up to replacing $q$ by $q^{-1}$ (i.e.,~up to taking the graded dual vector space). This is only possible if $C^\perp = J$, proving the claim. 

    Finally,  we show that the orthogonal complement of the left ideal $J = C^\perp$ is precisely the coordinate ring of $G_1[\![z^{-1}]\!]/ G_1[\![z^{-1}]\!]^\omega$. A containment in one direction is straightforward, since the Lie algebra of $G_1[\![z^{-1}]\!]^\sigma$ contains $(z^{-1}\g[z^{-1}])^\sigma$. Equality follows from Proposition \ref{prop:loopSymSpace}(3), since the Hilbert series of the affine scheme $G_1[\![z^{-1}]\!]^\sigma \cong (z^{-1}\g[\![z^{-1}]\!])^\sigma$ is given once again by the Hilbert series above, up to replacing $q$ by $q^{-1}$. This proves Part (3).
\end{proof}


Combining Proposition \ref{prop:loopSymSpace} and Theorem \ref{thm: QDP Yangian2}, we obtain the main result of this section.
\begin{thm}    
    \label{thm:tY:G1sigma}
     Let $\U_\hbar(\g[z]^\omega) \subset \U_\hbar(\g[z])$ be a twisted Yangian such that Assumptions (\ref{eq: QDP1}) and (\ref{eq: QDP2}) both hold.  Then $\U_\hbar(\g[z]^\omega)'$ quantizes the affine scheme $G_1[\![z^{-1}]\!]^\sigma$ with its (doubled) Dirac Poisson structure, and the inclusion $\U_\hbar(\g[z]^\omega) \subset \U_\hbar(\g[z])$ quantizes the map $G_1[\![z^{-1}]\!] \rightarrow G_1[\![z^{-1}]\!]^\sigma$  defined by $g \mapsto g \sigma(g)$. The left coideal structure on $\U_\hbar(\g[z]^\omega)$ corresponds to the left action of $G_1[\![z^{-1}]\!]$ on $G_1[\![z^{-1}]\!]^\sigma$ given by $g\cdot p = g p \sigma(g).$
\end{thm}

\begin{rem}
    \label{rem: filtered 2}
    As in Remark \ref{rem: filtered 1} we may consider instead the filtered algebra $\U_{\hbar=1}(\g[z]^\omega)'$, from which we recover $\U_\hbar(\g[z]^\omega)'$ via the Rees algebra construction. This filtered algebra perspective is in keeping with the approach to (shifted) twisted Yangians taken in the main body of this paper, and much of the literature.
\end{rem}

Assumptions \eqref{eq: QDP1} and \eqref{eq: QDP2} hold for many of the twisted Yangians studied in the literature. 

\begin{eg}
    The classic twisted Yangians were first introduced mathematically by Olshanski \cite{Ols92}. These algebras fit the framework of this section.  More precisely, equip $\bC^n$ with a nondegenerate bilinear form which we assume to be either orthogonal or symplectic, encoded by a matrix $Q$. As in \cite[Definition 2.1.1]{Mol07}, we may associate a corresponding  algebra denoted $Y_Q(\g_n)$. It follows from the results of \cite[\S 2]{Mol07} that these algebras are twisted Yangians in the conventions of this section (specialized at $\hbar=1$).  In this case we take $G = GL_n$, and the anti-involution $\sigma$ of $G_1[\![z^{-1}]\!]$ is defined by $g(z) \mapsto Q^{-1} g(-z)^T Q$ where $A \mapsto A^T$ denotes matrix transpose.  By the above results, it follows that $Y_Q(\g_n)$ quantizes a Poisson structure on the fixed point locus 
    $$
    G_1[\![z^{-1}]\!]^\sigma = \left\{ g(z) \in G_1[\![z^{-1}]\!] \mid g(z) = Q^{-1} g(-z)^T Q \right\}.
    $$
    This locus is naturally a homogeneous space for $G_1[\![z^{-1}]\!]$, quantized by the coideal structure on the twisted Yangian.  
\end{eg}

\begin{rem}
    In the setting of the previous example, the symplectic case was previously studied by Harada \cite[\S 4.1, 4.2]{Harada} in connection with integrable systems and Gelfand-Tsetlin bases for symplectic Lie algebras.  Harada views the twisted Yangian as a quantization of the homogeneous space $G_1[\![z^{-1}]\!] / G_1[\![z^{-1}]\!]^\omega$, which is isomorphic to $G_1[\![z^{-1}]\!]^\sigma$ by Proposition \ref{prop:loopSymSpace}.  Thus, we can think of Theorem \ref{thm: QDP Yangian2} and Theorem \ref{thm:tY:G1sigma} as a generalization of aspects of Harada's work.
\end{rem}

\begin{rem}  \label{rem:tYsame}
We conjecture that a variant of the algebra $\Yt_0$ over the ring $\C[\hbar]$ in Drinfeld presentation considered in other sections is always a twisted Yangian in the sense of Definition~ \ref{def:tYaxioms}. This conjecture holds for split type (excluding type G$_2$) and quasi-split type AIII. In the cases of quasi-split type A, the iYangians \cite{LWZ25, LZ24} are isomorphic to the twisted Yangians defined via R-matrix  \cite{Ols92, MR02}. In the case of split type, the iYangians \cite{Lu25} are isomorphic to the twisted Yangians defined via $J$-presentation \cite{Mac02,BR17}. Hence they are coideal subalgebras of the Yangian. It is nontrivial but possible to verify more cases. 
\end{rem}


\section{Shifted iYangians and fixed point loci}
\label{sec:stYloci}

We show that a Poisson involution $\sigma$ descending from the loop group preserves the closed subscheme $\cW_\mu$ exactly when $\mu$ is even spherical. In this case, we give a Poisson algebra presentation for the coordinate algebra of the $\sigma$-fixed point locus ${}^\imath\cW_\mu$, and show that it is isomorphic to the associated graded of the shifted iYangian $\Yt_\mu$ with respect to a natural filtration.

\subsection{Spaces $\cW_\mu$ from loop groups}
\label{sec:Wmu}

We briefly recall the notion of presentations of Poisson algebras by generators and relations, following \cite[\S 3.1]{T23}.
The free Poisson algebra generated by a set $X$ is the symmetric algebra $S(L_X)$, where $L_X$ is the free Lie algebra generated by $X$. Here $S(L_X)$ is endowed with a Poisson structure in the standard way.  For any Poisson algebra $A$, we say that $A$ is \emph{Poisson generated} by its subset $X \subset A$ if there is a surjective homomorphism of Poisson algebras $S(L_X) \twoheadrightarrow A$.  Given any set $R \subset S(L_X)$ which we refer to as \emph{Poisson relations}, we may consider the Poisson ideal $\langle R\rangle_{Poisson} \subset S(L_X)$ which they generate.  If the homomorphism $S(L_X) \twoheadrightarrow A$ induces an isomorphism $S(L_X) / \langle R\rangle_{Poisson} \cong A$, then we refer to this as  presentation of $A$ as a Poisson algebra by generators and relations.

For any affine algebraic group $H$ over $\bC$, there is a loop group $H(\!(z^{-1})\!)$ and its subgroups $H_1[\![z^{-1}]\!]$ and $H[z]$; see \S \ref{ssec: Yangians and loop groups}.

Recall that $\g$ is a simple Lie algebra of type ADE, and let $G$ be a corresponding connected  algebraic group of adjoint type\footnote{We make the assumption that $G$ is adjoint type in order that it has all fundamental coweights. However this assumption is largely unnecessary, see \cite[Remark 2.10]{MW24}.}.  Fix Chevalley generators $\{e_i, h_i, f_i\}_{i\in \I}$  for $\g$, yielding a triangular decomposition $\mathfrak{g} = \mathfrak{n}^+ \oplus \mathfrak{h} \oplus \mathfrak{n}^-$, and let $U^+, U^-, T \subset G$ be the corresponding unipotent subgroups and maximal torus.

Each coweight $\mu : \bC^\times \rightarrow T$ yields a point $z^\mu \in T(\!(z^{-1})\!)$. Following \cite[\S 5.9]{FKPRW}, we define a closed subscheme of $G(\!(z^{-1})\!)$ by:
\begin{equation}
\label{eq:Wmu}
\cW_\mu  \ = \ U_1^+[\![z^{-1}]\!] T_1[\![z^{-1}]\!] z^\mu U_1^-[\![z^{-1}]\!].
\end{equation}
This is an affine scheme and is the product of its factors. In particular, $\cW_0 = G_1[\![z^{-1}]\!].$   Following \cite[\S 5.9]{FKPRW}, for antidominant coweights $\nu_1,\nu_2$, we define a (surjective) shift map
\begin{equation}
\label{eq:shift map}
\iota_{\mu,\nu_1,\nu_2} : \cW_{\mu+\nu_1+\nu_2} \longrightarrow \cW_\mu, \qquad g \mapsto \pi ( z^{-\nu_1} g z^{-\nu_2}).
\end{equation}
Here $\pi$ is the surjective map from $U^+(\!(z^{-1})\!) T_1[\![z^{-1}]\!]z^\mu U^-(\!(z^{-1})\!)$ onto $U_1^+[\![z^{-1}]\!]T_1[\![z^{-1}]\!]z^\mu U_1^-[\![z^{-1}]\!] $ defined by projecting along the appropriate factors in the decompositions $U^\star(\!(z^{-1})\!) = U^\star[z] \times U_1^\star[\![z^{-1}]\!]$ (for $\star=+$ or $-$) induced by multiplication.

There is a natural action of $\bC^\times$ on $G(\!(z^{-1})\!)$ by loop rotation, however this action does not preserve the subscheme $\cW_\mu$ unless $\mu = 0$.  Following \cite[\S 5.9]{FKPRW} we can correct this via a slight modification: pick any decomposition $\mu = \mu_1 + \mu_2$ into coweights $\mu_1,\mu_2$, and define an action of $\bC^\times$ on $\cW_\mu$ by
\begin{equation}
\label{eq:loop rotation action}
    s\cdot g(z) = s^{-\mu_1} g(sz) s^{-\mu_2}.
\end{equation}

For each $i\in\I$ there are homomorphisms $\psi^\pm_i: U^\pm \rightarrow \mathbb{G}_a$, normalized by $\psi_i^+( \exp(t e_i )) = t$ and $\psi_i^-(\exp(t f_i )) = t$.  We also have homomorphisms $\alpha_i : T \rightarrow \mathbb{G}_m$ corresponding to the simple roots. Define distinguished functions on $\cW_\mu$ as follows: given $g = n^+ h z^\mu n^- \in \cW_\mu$, we have
\begin{equation}
\label{eq: Drinfeld gens Wmu}
e_i^{(r)}( g) = [z^r] \ \psi_i^+(n^+), \qquad h_i^{(r)}(g) = [z^r] \ \alpha_i(hz^\mu), \qquad f_i^{(r)} (g) = [z^r] \ \psi_i^-(n^-).
\end{equation}
Here, for a series $F(z) = \sum F_n z^n $ we denote by $[z^r]\ F(z) = F_r$ the coefficient of $z^r.$  We note that $e_i^{(r)} =  f_i^{(r)} = 0$ for $r \lle 0$, $h_i^{(s)} = 0$ for $s<-\langle \mu,\alpha_i\rangle$, and $h_i^{(-\langle\mu,\alpha_i\rangle)} = 1$.  Under the $\bZ$--grading on $\bC[\cW_\mu]$ corresponding to a $\bC^\times$--action as in (\ref{eq:loop rotation action}), we have
\begin{equation}
\label{eq:degrees under Wmu}
\deg h_i^{(r)} = r + \langle \mu,\alpha_i\rangle, \qquad \deg e_i^{(s)} = s + \langle \mu_1,\alpha_i\rangle, \qquad \deg f_i^{(s)} = s+ \langle \mu_2,\alpha_i\rangle.
\end{equation}
The torus $T$ acts on $\cW_\mu$ by conjugation, inducing a grading on $\bC[\cW_\mu]$ by the root lattice $Q$ of $\g$. For each $\beta\in \Delta^+$ and $s\gge 1$ we fix arbitrary decompositions $\beta=\alpha_{i_1}+\ldots + \alpha_{i_\ell}$ as in \S \ref{ssec:PBW for sty}, and set
$$
e_\beta^{(s)} = \big\{e_{i_1}^{(s)},\{e_{i_2}^{(1)},\cdots\{e_{i_{\ell-1}}^{(1)}, e_{i_\ell}^{(1)}\}\cdots \}\big\}, 
\qquad 
f_\beta^{(s)} = \big\{f_{i_1}^{(s)},\{f_{i_2}^{(1)},\cdots\{f_{i_{\ell-1}}^{(1)}, f_{i_\ell}^{(1)}\}\cdots \}\big\}.
$$

Following \cite[Definition B.2]{BFN19} and \cite[Definition 3.5]{FKPRW}, one defines algebras $Y_\mu = Y_\mu(\g)$ called \emph{shifted Yangians}.  The algebra $Y_\mu$ has a PBW basis in elements $H_i^{(r)}, E_\alpha^{(s)}, F_\alpha^{(s)}$ by \cite[Corollary 3.15]{FKPRW}. There are shift homomorphisms $\iota_{\mu,\nu_1,\nu_1}: Y_\mu \rightarrow Y_{\mu+\nu_1+\nu_2}$ by \cite[Proposition~3.8]{FKPRW}, and filtrations $F_{\mu_1,\mu_2}^\bullet Y_\mu$ by \cite[\S 5.4]{FKPRW} (compare with degrees in \eqref{eq:degrees under Wmu}).  There is also a grading of $Y_\mu$ by the root lattice, defined by $\deg E_\alpha^{(s)} = \alpha, \deg H_i^{(r)}=  0$ and $\deg F_\alpha^{(s)} = -\alpha$, which respects the filtrations $F_{\mu_1, \mu_2}^\bullet Y_\mu$.

As a preparation for Proposition \ref{prop: Wmu Poisson gens}, we list the Poisson relations satisfied by generators $h_i^{(r)}, e_i^{(s)}, f_i^{(s)}$ in $\bC[\cW_\mu]$, for $i\in \mathbb I$, $r \in \bZ $ and $s \gge 1$:
\begin{align}
h_i^{(r)} = 0~\text{ for }~r<&- \langle\mu,\alpha_i\rangle, \quad h_i^{(-\langle \mu, \alpha_i\rangle)} = 1,  \label{eq:csy0} \\
\{h_i^{(r_1)}, h_j^{(r_2)}\} &= 0, \label{eq:csy1}\\
\{e_i^{(s_1)}, f_j^{(s_2)}\} &= \delta_{ij} h_i^{(s_1+s_2-1)}, \label{eq:csy2}\\
 \{ h_i^{(r)}, e_j^{(s)}\}  &=  c_{ij} \sum_{p\gge 0} h_i^{(r-p-1)} e_j^{(s+q)}, \label{eq:csy3}\\
 \{h_i^{(r)}, f_j^{(s)}\}  &=  -  c_{ij} \sum_{p\gge 0} h_i^{(r-p-1)} f_j^{(s+p)}, \label{eq:csy4}\\
\{ e_i^{(r_1+1)}, e_j^{(r_2)}\} - \{ e_i^{(r_1)}, e_j^{(r_2+1)}\} &=  c_{ij} e_i^{(r_1)} e_j^{(r_2)}, \label{eq:csy5}\\
 \{ f_i^{(r_1+1)}, f_j^{(r_2)}\} - \{ f_i^{(r_1)}, f_j^{(r_2+1)}\} &= - c_{ij} f_i^{(r_1)} f_j^{(r_2)}, \label{eq:csy6}\\
i \neq j, N = 1-c_{ij} \Longrightarrow \operatorname{Sym}_{r_1,\ldots,r_N}& \{e_i^{(r_1)},\{e_i^{(r_2)},\cdots\{e_i^{(r_N)}, e_j^{(s)}\}\cdots\}\} = 0, \label{eq:csy7}\\
i \neq j, N = 1-c_{ij} \Longrightarrow \operatorname{Sym}_{r_1,\ldots,r_N} &\{f_i^{(r_1)},\{f_i^{(r_2)},\cdots\{f_i^{(r_N)}, f_j^{(s)}\}\cdots\}\} = 0. \label{eq:csy8}
\end{align}
The next result is an extension of the results of \cite{FKPRW}, inspired by the definition of the semiclassical shifted Yangian from \cite[\S 3.3]{T23} and \cite[Proposition~3.4]{TT24}. 

\begin{prop}
\label{prop: Wmu Poisson gens}
\begin{enumerate}
\item For each decomposition $\mu = \mu_1 + \mu_2$ there is an isomorphism of $Q\times \bZ$--graded algebras $$\gr^{F^\bullet_{\mu_1,\mu_2}} Y_\mu \cong \bC[\cW_\mu],$$ where the $\bZ$--grading on $\bC[\cW_\mu]$ corresponds to the $\bC^\times$--action (\ref{eq:loop rotation action}).   This isomorphism identifies the classes of $E_i^{(s)}, H_i^{(r)}, F_i^{(s)}$ with the elements $e_i^{(s)}, h_i^{(r)}, f_i^{(s)}$. It endows $\cW_\mu$ with a Poisson structure, independent of $\mu_1, \mu_2$, with the Poisson bracket having degree $(0,-1)$.

\item The coordinate ring $\bC[\cW_\mu]$ is the Poisson algebra generated by $h_i^{(r)}, e_i^{(s)}, f_i^{(s)}$ for $i\in \mathbb I$, $r \in \bZ $ and $s \gge 1$, with the defining Poisson relations \eqref{eq:csy0}--\eqref{eq:csy8}.

\item $\bC[\cW_\mu]$ is a polynomial algebra  on the PBW generators 
$$ 
 \{
 e_\beta^{(s)} : \beta \in\Delta^+, s \gge 1 \} \cup \{h_i^{(r)} : i \in \I, r > -\langle \mu,\alpha_i\rangle\} \cup\{f_\beta^{(s)} : \beta \in \Delta^+, s \gge 1
 \}.
$$

\item The shift map (\ref{eq:shift map}) is Poisson, and it is quantized by the shift homomorphism $\iota_{\mu,\nu_1,\nu_2}$ of shifted Yangians.  On the Poisson generators the corresponding map $\iota_{\mu,\nu_1,\nu_2} : \bC[\cW_\mu] \rightarrow \bC[\cW_{\mu+\nu_1+\nu_2}]$ is given by
$$
e_i^{(r)} \mapsto e_i^{(r-\langle \nu_1, \alpha_i\rangle)},\qquad h_i^{(r)} \mapsto h_i^{(r-\langle \nu_1+\nu_2, \alpha_i\rangle)},\qquad f_i^{(r)} \mapsto f_i^{(r-\langle \nu_2, \alpha_i\rangle)}.
$$
\end{enumerate}
\end{prop}

\begin{proof}
Parts (1), (3) and (4) follow from \cite[Theorem 5.15]{FKPRW}, which also implies that $\bC[\cW_\mu]$ is Poisson generated by the elements in Part (2).  Furthermore it follows that the claimed relations from (2) hold in $\bC[\cW_\mu]$, since these are simply the classical limits of the defining relations of $Y_\mu$. 

Let us temporarily denote by $\mathscr{Y}_\mu$ the Poisson algebra defined by generators and relations as in Part (2), so we have $\mathscr{Y}_\mu \twoheadrightarrow \bC[\cW_\mu]$. To complete the proof, it remains to prove the claim that $\mathscr{Y}_\mu$ is spanned by the same PBW monomials from Part (3).   For $\mu$ dominant, this follows from \cite[Theorem~ 3.4]{T23}.  (Although that result is stated for $\gl_n$, it holds for any simple $\g$ with minor modifications.)  For general $\mu$, consider the subalgebras $\mathscr{Y}^{>}_\mu, \mathscr{Y}^0_\mu, \mathscr{Y}^<_\mu$ Poisson generated by the elements $e_i^{(q)}$ (resp.~$h_i^{(p)}$, resp.~$f_i^{(q)}$). Then $\mathscr{Y}_\mu^0$ is a quotient of the polynomial ring generated by the $h_i^{(p)}$. Similarly to \cite[Remark~ 3.3]{FKPRW} there are surjections $\mathscr{Y}^>_0 \twoheadrightarrow \mathscr{Y}^>_\mu$ and $\mathscr{Y}^<_0 \twoheadrightarrow \mathscr{Y}^<_\mu$, so the spanning set for $\mathscr{Y}^>_0$ yields one for $\mathscr{Y}^>_\mu$, and similarly for $\mathscr{Y}^<_\mu$.  Finally using the relations one can see that $\mathscr{Y}^<_\mu \otimes \mathscr{Y}^0_\mu \otimes \mathscr{Y}^>_\mu \twoheadrightarrow \mathscr{Y}_\mu$, cf.~the first half of the proof of Lemma~\ref{lem:iWmu_pres}, which proves the claim.
\end{proof}

\begin{rem}
     $\cW_\mu$ inherits a Poisson structure from its quantization $Y_\mu$, by Proposition \ref{prop: Wmu Poisson gens}(1).  If $\mu$ is antidominant, then this Poisson structure admits an alternate description by \cite[Theorem~ A.8]{KPW22}: it is restricted from the Poisson structure on the group $G(\!(z^{-1})\!)$ corresponding to a Manin triple as in \S \ref{ssec: Yangians and loop groups}.
\end{rem}

\subsection{Fixed point loci of the spaces $\cW_\mu$}
\label{sec:fixed points in slices}
Recall that  $\tau$ is a (possibly trivial) involution of the Dynkin diagram of  $\g$. Define a corresponding involution $\omega_\tau$ of the Lie algebra $\g$, given on Chevalley generators by $e_i \mapsto -f_{\tau i}$, $h_i \mapsto - h_{\tau i}$ and $f_i \mapsto - e_{\tau i} $. In other words, $\omega_\tau$ is the composition of the Chevalley involution $\omega_0$ on $\g$ and the involution induced by $\tau$ on $\g$. 

We shall work from now on with this explicit involution $\omega_\tau$ in place of $\omega$ in Section \ref{sec:tYdualityPrinciple}. We define an anti-involution $\sigma$ of the loop algebra $\g(\!(z^{-1})\!)$ by 
\begin{align} \label{sigma:gz}
\sigma(x\otimes z^k) = - \omega_\tau(x) \otimes (-z)^k.
\end{align}
Explicitly, this anti-involution $\sigma$ is determined by
\begin{equation}
    \label{eq:sigma:gz}
    e_i \otimes z^k \mapsto (-1)^k f_{\tau i} \otimes z^k, \qquad h_i \otimes t^k \mapsto (-1)^k h_{\tau i} \otimes z^k, \qquad  f_i \otimes z^k \mapsto (-1)^k e_{\tau i} \otimes z^k.
\end{equation}
The corresponding anti-involution $\sigma$ of the group $G(\!(z^{-1})\!)$ (see \S \ref{ssec:loop sym spaces} and \eqref{eq:sigma}) satisfies $\sigma(T_1[\![z^{-1}]\!])= T_1[\![z^{-1}]\!]$ and $\sigma(U^\pm_1[\![z^{-1}]\!]) = U^\mp_1[\![z^{-1}]\!]$, as well as
\begin{equation}
\label{eq:sigma:zmu}
\sigma(z^\mu) = (-z)^{\tau \mu} = (-1)^{\tau \mu} z^{\tau \mu}.
\end{equation}
Similarly, if $s\in \bC^\times$ then the element $s^\mu\in T$ satisfies $\sigma(s^\mu) = s^{\tau \mu}.$  

Inside the maximal torus $T \subset G$, recall from \eqref{eq:def of Tt} the following subgroup 
\begin{equation*}
    \Tt  = \{ t \in T \mid \sigma(t) = t^{-1} \} = \{ t \in T \mid \tau(t) = t^{-1} \}.
\end{equation*}
The coordinate ring of $\Tt$ is the group algebra $\bC[\Qt]$; see \eqref{Qtau}.  

\begin{eg}
\label{eg:transpose}
Take $\g = \fksl_n$ and $\tau = \id$, then  $\omega_\tau$ is given by negative transpose $\omega_\tau(x) = - x^T$.  For any $g(z) \in \mathrm{PGL}_n(\!(z^{-1})\!)$ we then have $\sigma\big( g(z) \big) = g(-z)^T$.  
\end{eg}

Recall the notion of spherical and even coweights from Definition \ref{def:spherical}.

\begin{lem}
\label{lem:evenspherical}
    The following are equivalent, for any coweight $\mu$:
    \begin{enumerate}
        \item $\sigma$ preserves the subscheme $\cW_\mu \subset G(\!(z^{-1})\!)$;
        \item $\sigma(z^\mu) = z^\mu$;
        \item $\mu$ is even and $\tau\mu = \mu$;
        \item $\mu$ is even spherical.
    \end{enumerate}
\end{lem}

\begin{proof}
    Each $g \in \cW_\mu$ has a unique Gauss decomposition $g = n^+ h z^\mu n^-$ in \eqref{eq:Wmu}.  Then 
    $$
    \sigma(g) = \sigma(n^-) \sigma(h) \sigma(z^\mu) \sigma(n^+) \in U_1^+[\![z^{-1}]\!]T_1[\![z^{-1}]\!] \sigma(z^\mu) U_1^-[\![z^{-1}]\!].
    $$
    Since Gauss decompositions are unique, we have $\sigma(g) \in \cW_\mu$ if and only if $\sigma(z^\mu) = z^\mu.$  Comparing with (\ref{eq:sigma:zmu}), this  holds exactly when $(-1)^{\tau \mu} = 1$ and $\tau \mu = \mu$.  But $(-1)^{\tau \mu} = 1$ means precisely that $\tau \mu$ is even, which is equivalent to $\mu$ itself being even.  

    The equivalence of (3) and (4) is also clear. 
\end{proof}

\begin{lem}
    \label{lem: gradings and shifts iWmu}
    Suppose that $\mu$ is an even spherical coweight.
    \begin{enumerate}
        \item Consider a $\bC^\times$--action on $\cW_\mu$ as in (\ref{eq:loop rotation action}), such that $\mu_2 = \tau \mu_1$, and hence $\mu = \mu_1 + \tau \mu_1.$ Then this $\bC^\times$--action commutes with $\sigma$.

        \item Under the adjoint action of the group $T$ on $\cW_\mu$, the subgroup $\Tt$ commutes with $\sigma$.
        
        \item Let $\nu$ be an even antidominant coweight. Then the shift map $\iota_{\mu,\nu,\tau \nu}: \cW_{\mu+\nu+\tau\nu} \rightarrow \cW_\mu$ from (\ref{eq:shift map}) intertwines the respective involutions $\sigma|_{\cW_\mu} \circ \iota_{\mu, \nu, \tau \nu} = \iota_{\mu,\nu,\tau \nu} \circ \sigma|_{\cW_{\mu+\nu+\tau \nu}}$. 

        \item The restriction $\sigma|_{\cW_\mu}$ is a Poisson involution, and given on the Poisson generators by
    $$
    e_i^{(r)} \mapsto (-1)^r f_{\tau i}^{(r)}, \qquad h_i^{(s)} \mapsto (-1)^s h_{\tau i}^{(s)} , \qquad f_i^{(r)} \mapsto (-1)^r e_{\tau i}^{(r)}.
    $$
    \end{enumerate}
\end{lem}
\begin{proof}
    For Part (1), let $g = g(z) \in \cW_\mu$. Then we compute:
    \begin{align*}
    \sigma( s\cdot g(z) ) = \sigma( s^{-\mu_1}g(sz) s^{-\tau \mu_1}) &= \sigma( s^{-\tau \mu_1}) \sigma( g(sz)) \sigma( s^{-\mu_1})\\& = s^{-\mu_1} \sigma( g(sz)) s^{-\tau\mu_1} = s \cdot \sigma( g(z)).
    \end{align*}
    Parts (2) and (3) follow from similar direct calculations.
    
    It remains to prove Part (4). It follows from the results of \S \ref{ssec:loop sym spaces} that $\sigma$ is a Poisson involution of $G(\!(z^{-1})\!)$ in a natural way. For $\mu$ antidominant, $\cW_\mu \subset G(\!(z^{-1})\!)$ is a Poisson subscheme by \cite[Theorem~ A.8]{KPW22}, and therefore the restriction  $\sigma|_{\cW_\mu}$ is Poisson.  For general $\mu$, pick an even antidominant weight $\nu$ such that $\mu+\nu+\tau\nu$ is antidominant, and by (3) the corresponding Poisson embedding $ \bC[\cW_\mu] \hookrightarrow \bC[\cW_{\mu+\nu+\tau\nu}]$ intertwines the action of $\sigma$, which proves that $\sigma|_{\cW_\mu}$ is Poisson.  The explicit description in coordinates in (4) follows from the formula $\sigma(n^+ h z^\mu n^-) = \sigma(n^-) \sigma(h) z^\mu \sigma(n^+)$ together with equations  (\ref{eq: Drinfeld gens Wmu}) and (\ref{eq:sigma:gz}).   
\end{proof}

Recall an element in $\cW_\mu$ \eqref{eq:Wmu} is written naturally as $u h z^\mu u_-$. 
    
\begin{dfn}  \label{def:iWmu}
For any even spherical coweight $\mu$, we define the fixed point locus
\[
    \iW_\mu \ := \ (\cW_\mu)^\sigma 
    = \left\{ u h z^\mu u_- \in \cW_\mu \mid \sigma(u) = u_-, \ \sigma(h) = h \right \}.
\]
\end{dfn}

Since $\sigma$ defines a Poisson involution of $\cW_\mu$, the fixed point locus $\iW_\mu$ inherits a Poisson structure via Dirac reduction as in \S\ref{ssec: dirac reduction and fixed points}.  We may rescale this Poisson structure by a factor of 2, i.e., $\{a,b\} := 2 \{a,b\}_{Dirac}$; see Proposition \ref{prop:loopSymSpace} for an explanation. From now on, we will always equip $\iW_\mu$ with this doubled Poisson bracket.

\begin{prop}
\label{prop:iWmu}
    $\iW_\mu$ has a natural Poisson structure inherited from $\cW_\mu$ via (doubled) Dirac reduction. The group $\Tt$ acts on $\iW_\mu$ by conjugation, and for each choice of decomposition $\mu = \mu_1 + \tau \mu_1$ there is a $\bC^\times$--action on $\iW_\mu$ by
    $s\cdot g(z) = s^{-\mu_1} g(sz) s^{-\tau \mu_1}$
    such that the Poisson bracket has degree $-1$.    Finally, for each antidominant weight $\nu$ such that $\nu+\tau\nu$ is even, the shift map $\iota_{\mu,\nu}^\tau : \iW_{\mu+\nu+\tau\nu} \rightarrow \iW_\mu$, 
    $g \mapsto \pi( z^{-\nu} g z^{-\tau \nu})$,
    is surjective and Poisson.
\end{prop}

\begin{proof}
    This follows from Lemma \ref{lem: gradings and shifts iWmu} and Definition \ref{def:iWmu}: the subscheme  $\iW_\mu$ is preserved by the $\bC^\times$--action on $\cW_\mu$ corresponding to $\mu = \mu_1 + \mu_2$ with $\mu_2 = \tau \mu_1$, as well as the action of the group $\Tt$, and finally the map $\iota_{\mu,\nu}^\tau : \iW_{\mu+\nu+\tau\nu} \rightarrow \iW_\mu$ is the restriction of the Poisson map $\iota_{\mu,\nu,\tau\nu} : \cW_{\mu+\nu+\tau\nu} \rightarrow \cW_\mu$.
\end{proof}

We introduce natural functions on $\iW_\mu$ by restriction from $\cW_\mu$:
\begin{equation}
    \label{eq: Drinfeld gens iWmu}
    h_i^{(r)} := h_i^{(r)}\big|_{\iW_\mu} = (-1)^{(r)} h_{\tau i}^{(r)}\big|_{\iW_\mu}, \qquad b_i^{(s)} := e_i^{(s)}\big|_{\iW_\mu} = (-1)^{s} f_{\tau i}^{(s)}\big|_{\iW_\mu}
\end{equation}
Consider the action of $\bC^\times$ on $\iW_\mu$ associated with a decomposition $\mu = \mu_1 + \tau\mu_1$.  Under the corresponding $\bZ$--grading of $\bC[\iW_\mu]$ we have degrees
$$
\deg h_i^{(r)} = r + \langle \mu, \alpha_i \rangle, \qquad \deg e_i^{(s)} = s + \langle \mu_1, \alpha_i\rangle
$$
as follows from (\ref{eq:degrees under Wmu}); compare with the filtration $F_{\mu_1}^\bullet \mc Y_\mu$ defined in (\ref{eq: grading on iYmu}). At the same time, the adjoint action of $\Tt$ induces a grading on $\bC[\iW_\mu]$ by $\Qt$, with $h_i^{(r)}$ in degree $\overline{0}$ and $b_j^{(s)}$ in degree $\overline{\alpha_i}$. Since the actions of $\bC^\times$ and $\Tt$ commute, we obtain a grading on $\bC[\iW_\mu]$ by $\Qt\times \bZ$.

As a preparation for Theorem~\ref{thm:iWmuPoisson} below, we list the Poisson relation  \eqref{eq:csty0}--\eqref{eq:csty7} satisfied by generators $h_i^{(r)}, b_i^{(s)}$ for $i \in \I$, $r \in \bZ$ and $s \gge 1$: 
    \begin{align}
         h_i^{(r)} = 0~\text{ for }~r < &-\langle \mu,\alpha_i\rangle, \quad h_i^{(-\langle \mu, \alpha_i\rangle)} = 1,  \label{eq:csty0}\\
         \{h_i^{(r_1)}, h_j^{(r_2)} \}  &= 0, \label{eq:csty1}\\
        \{h_i^{(r)}, b_j^{(s)}\} &= \sum_{p \gge 0} \big( c_{ij} +(-1)^{p+1} c_{\tau i, j} \big) h_i^{(r-p-1)} b_j^{(s+p)}, \label{eq:csty2}\\
                 \{b_i^{(s_1+1)}, b_j^{(s_2)}\} - \{b_i^{(s_1)}, b_j^{(s_2+1)}\} &= c_{ij} b_i^{(s_1)} b_j^{(s_2)} + 2 \delta_{\tau i, j} (-1)^{s_1} h_j^{(s_1+s_2)} \label{eq:csty3}
    \end{align}
    as well as the Serre relations: 
    \begin{align}
    \{b_i^{(s_1)},b_{\tau i}^{(s_2)}\} &=(-1)^{s_1-1}h_{\tau i}^{(s_1+s_2-1)}, \quad
    \text{for } c_{ij} = 0, \, \tau i \neq j,
        \label{eq:csty4}
        \\
    \mathrm{Sym}_{s_1,s_2}\big\{b_{i}^{(s_1)},\{b_{i}^{(s_2)},b_{j}^{(s)}\}\big\}& =0, \qquad\qquad\qquad   
    \text{for } c_{ij}=-1, \, i\ne \tau i\ne j,
    \label{eq:csty5}
    \\
    \mathrm{Sym}_{s_1,s_2}\big\{b_{i}^{(s_1)},\{b_{i}^{(s_2)},b_{j}^{(s)}\}\big\}
        &= 2c_{ij} (-1)^{s_1-1} \sum_{p\gge 0} h_{i}^{(s_1+s_2-2p-1)}b_j^{(s+2p-1)}, 
        \notag \\
    &\qquad\qquad\qquad\qquad\text{ for }c_{ij}=-1, \, i=\tau i,
    \label{eq:csty6}
    \\
     \mathrm{Sym}_{s_1,s_2}\big\{b_{i}^{(s_1)},\{b_{i}^{(s_2)},b_{\tau i}^{(s)}\}\big\}
        &=
        4~\mathrm{Sym}_{s_1,s_2}(-1)^{s_1-1}\sum_{p \gge 0}h_{\tau i}^{(s_1+s-2p-2)} b_{i}^{(s_2+2p)},
        \notag \\
        &\qquad\qquad\qquad\qquad \text{for } c_{i,\tau i}=-1.
    \label{eq:csty7}
    \end{align}
Since $\tau$ is an involution, it follows from \eqref{eq:csty4} that 
\beq \label{hi=htaui}
h_i^{(r)}=(-1)^{r}h_{\tau i}^{(r)},
\eeq
and in particular,
\beq \label{h=0:i=taui}
h_i^{(2r+1)}=0, \qquad \text{ if } \tau i=i.
\eeq
For each $\beta\in \Delta^+$ and $s\gge 1$ we choose an element $b_\beta^{(s)} \in \bC[\iW_\mu]$ similarly to \S\ref{ssec:PBW for sty}: 
    $$
    b_\beta^{(s)} = \Big\{b_{i_1}^{(s)}, \big\{b_{i_2}^{(1)},\cdots \{ b_{i_{\ell-1}}^{(1)}, b_{i_\ell}^{(1)}\}\cdots\big\} \Big\}.
    $$
Now we can formulate the main result of this section.

\begin{thm}
\label{thm:iWmuPoisson}
Let $\mu$ be an even spherical coweight.
\begin{enumerate}
    \item For each decomposition $\mu = \mu_1 + \tau \mu_1$, there is an isomorphism of $\Qt\times \bZ$--graded Poisson algebras 
    $$
    \gr^{F_{\mu_1}^\bullet} \Yt_\mu \cong \bC[\iW_\mu],
    $$ 
    which identifies the classes  $\bar H_i^{(r)}, \bar B_j^{(s)}$ with the elements $h_i^{(r)}, b_j^{(s)}$.

    \item The coordinate ring $\bC[\iW_\mu]$ is the Poisson algebra generated by $h_i^{(r)}, b_i^{(s)}$ for $i \in \I$, $r \in \bZ$ and $s \gge 1$, with defining Poisson relations \eqref{eq:csty0}--\eqref{eq:csty7}.

    \item $\bC[\iW_\mu]$ is a polynomial algebra  on PBW generators
    $$
    \{ b_\beta^{(s)} : \beta \in \Delta^+, s \gge 1\} \cup\{h_i^{(2p)} : i \in \I_0, 2p > -\langle \mu,\alpha_i\rangle\} \cup \{h_i^{(s)} : i \in \I_1, s > -\langle \mu,\alpha_i\rangle\}.
    $$
    
    \item For any antidominant weight $\nu$ such that $\nu+\tau\nu$ is even, the shift map $\iota_{\mu,\nu}^\tau: \iW_{\mu+\nu+\tau\nu} \rightarrow \iW_\mu$ is quantized by the shift homomorphism $\iota_{\mu,\nu}^\tau : \Yt_\mu \rightarrow \Yt_{\mu+\nu+\tau\nu}$ from Lemma \ref{lem:shiftqs}.
\end{enumerate}
\end{thm}
The proof of this theorem will be given in \S \ref{ssec:proof of theorem on quantizing fixed locus} below.

\begin{rem}
Assume that $\Yt_0$ are identified with the twisted Yangians (specialized at $\hbar =1$) in Definition \ref{def:tYaxioms}; this is conjectured to be always true and known in type AI and AIII, see Remark~ \ref{rem:tYsame}. Then Theorem \ref{thm:tY:G1sigma} is compatible with Theorem \ref{thm:iWmuPoisson}(1), stating that the twisted Yangian $\Yt_0$ quantizes $G_1[\![z^{-1}]\!]^\sigma =\iW_0$. We expect that Theorem \ref{thm:iWmuPoisson} could then be proven along the lines of \cite[Theorem 5.15]{FKPRW}. Note that our proof below is not reliant on any such assumptions.
\end{rem}

\subsection{Proof of quantization Theorem \ref{thm:iWmuPoisson} of $\iW_\mu$}
\label{ssec:proof of theorem on quantizing fixed locus}

We will break the proof into several steps.

\begin{lem}
    \label{lem:iWmuDrinfeld}
    The elements $h_i^{(r)}$ and $b_i^{(s)}$ of $\bC[\iW_\mu]$ satisfy the Poisson relations (\ref{eq:csty0})--(\ref{eq:csty7}).
\end{lem}
\begin{proof}
     We remind the reader that we use the \emph{doubled} Dirac reduction Poisson bracket on $\iW_\mu$. By Remark \ref{rem: Dirac reduction Poisson}, this means that for $f,g \in \bC[\iW_\mu]$ we choose $\sigma$--invariant lifts $\widetilde{f}, \widetilde{g} \in \bC[\cW_\mu]$, and then define $\{f,g\} = 2 \{\widetilde{f}, \widetilde{g}\}\big|_{\iW_\mu}$. A similar formula applies for iterated Poisson brackets.  With this in mind, the claim is a direct calculation  using the relations (\ref{eq:csy0})--(\ref{eq:csy7}); similar calculations also appeared in the proof of \cite[Theorem 3.7]{T23}, except for the analogues of (\ref{eq:csty5}) and (\ref{eq:csty7}) which are not included since $\tau = \id$ in \cite{T23}.

     We illustrate this process with (\ref{eq:csty7}). Lift each $b_i^{(s)} \in \bC[\iW_\mu]$ to the $\sigma$--invariant element $\widetilde{b}_i^{(s)} = \frac{1}{2}(e_i^{(s)}+(-1)^s f_{\tau i}^{(s)})\in \bC[\cW_\mu]$. Then $\{b_i^{(s_1)}, \{b_i^{(s_2)}, b_{\tau i}^{(s)}\}\}$ is determined by a calculation in $\bC[\cW_\mu]$:
     \begin{align*}
        &~2\big\{ \widetilde{b}_i^{(s_1)}, 2 \big\{ \widetilde{b}_i^{(s_2)}, \widetilde{b}_{\tau i}^{(s)} \big\} \big\} \\
         = &~ \tfrac{1}{2} \Big\{ e_i^{(s_1)}+(-1)^{s_1} f_{\tau i}^{(s_1)}, \big\{ e_i^{(s_2)}+(-1)^{s_2} f_{\tau i}^{(s_2)}, e_{\tau i}^{(s)} + (-1)^s f_i^{(s)}\big\}\Big\}     \\
         =&~ \tfrac{1}{2} \Big\{ e_i^{(s_1)}+(-1)^{s_1} f_{\tau i}^{(s_1)},\\
         &\quad~~ \{e_i^{(s_2)}, e_{\tau i}^{(s)}\} + (-1)^s h_i^{(s_2+s-1)} - (-1)^{s_2} h_{\tau i}^{(s_2+s-1)} +(-1)^{s_2+s} \{f_{\tau i}^{(s_2)}, f_i^{(s)}\}\Big\}.
     \end{align*}
     We will expand this Poisson bracket.  Observe that the term $\{e_i^{(s_1)}, \{e_i^{(s_2)}, e_{\tau i}^{(s)}\}\}$ will vanish once we  symmetrize in $s_1$ and $s_2$ due to (\ref{eq:csy7}), and similarly for the term $(-1)^{s_1+s_2+s}\{f_{\tau i}^{(s_1)}, \{f_{\tau i}^{(s_2)}, f_{i}^{(s)}\}\}$. We thus omit these terms below.  Using the relations (\ref{eq:csy2})--(\ref{eq:csy4}), the remaining terms are seen to be:
     \begin{align*}
        &~- (-1)^s \sum_{p \gge 0} h_i^{(s_2+s-p-2)} e_i^{(s_1+p)} - (-1)^{s_2} \tfrac{1}{2}\sum_{p \gge 0}h_{\tau i}^{(s_2+s-p-2)} e_i^{(s_1+p)}\\&~ + (-1)^{s_2+s} \tfrac{1}{2} \{ f_{\tau i}^{(s_2)}, h_i^{(s_1+s-1)}\} 
        +(-1)^{s_1} \tfrac{1}{2} \{ e_i^{(s_2)}, - h_{\tau i}^{(s_1+s-1)}\}\\ &~-(-1)^{s_1+s} \tfrac{1}{2}\sum_{p \gge 0} h_i^{(s_2+s-p-2)} f_{\tau i}^{(s_1+p)} - (-1)^{s_1+s_2} \sum_{p \gge 0} h_{\tau i}^{(s_2+s-p-2)} f_{\tau i}^{(s_1+p)}\\
        = &~- (-1)^s \sum_{p \gge 0} h_i^{(s_2+s-p-2)} e_i^{(s_1+p)} - (-1)^{s_2} \tfrac{1}{2}\sum_{p \gge 0}h_{\tau i}^{(s_2+s-p-2)} e_i^{(s_1+p)}\\
        &~ - (-1)^{s_2+s} \tfrac{1}{2} \sum_{p \gge 0} h_i^{(s_1+s-p-2)} f_{\tau i}^{(s_2+p)} -(-1)^{s_1} \tfrac{1}{2}\sum_{p \gge 0} h_{\tau i}^{(s_1+s-p-2)} e_i^{(s_2+p)}\\
        &~-(-1)^{s_1+s} \tfrac{1}{2}\sum_{p \gge 0} h_i^{(s_2+s-p-2)} f_{\tau i}^{(s_1+p)} - (-1)^{s_1+s_2} \sum_{p \gge 0} h_{\tau i}^{(s_2+s-p-2)} f_{\tau i}^{(s_1+p)}.
     \end{align*}
     Restricting to $\iW_\mu$ and applying (\ref{eq: Drinfeld gens iWmu}) yields
     $$
      \sum_{p\gge 0}\left( 2(-1)^{s_2-p-1} + (-1)^{s_2-1}\right) h_{\tau i}^{(s_2+s-p-2)} b_i^{(s_1+p)}+  \sum_{p\gge 0}(-1)^{s_1-1} h_{\tau i}^{(s_1+s-p-2)} b_i^{(s_2+p)}.
     $$
     Finally, after symmetrizing in $s_1$ and $s_2$ we get
     $$
      \operatorname{Sym}_{s_1,s_2} \sum_{p\gge 0}\left( 2(-1)^{s_1-p-1} + 2(-1)^{s_1-1}\right) h_{\tau i}^{(s_1+s-p-2)} b_i^{(s_2+p)}.
     $$
     The coefficient appearing inside the sum vanishes unless $p$ is even, yielding (\ref{eq:csty7}).
\end{proof}

The next result establishes Part (3) of Theorem \ref{thm:iWmuPoisson}.
\begin{lem}
\label{lem:iWmuPBW}
    $\bC[\iW_\mu]$ is Poisson generated by the elements $h_i^{(r)}$ and $b_i^{(s)}$ for $i\in \I$, $r\in \bZ$ and $s \gge 1$, and is a polynomial ring in PBW generators 
    $$
    \{ b_\beta^{(s)} : \beta \in \Delta^+, s \gge 1\} \cup\{h_i^{(2p)} : i \in \I_0, 2p > -\langle \mu,\alpha_i\rangle\} \cup \{h_i^{(r)} : i \in \I_1, r > -\langle \mu,\alpha_i\rangle\}.
    $$
\end{lem}
\begin{proof}
    Inside $\bC[\cW_\mu]$, consider the subspace $V$ spanned by the PBW generators 
    $$ 
    \{e_\beta^{(s)} : \beta \in\Delta^+, s \gge 1 \} \cup \{h_i^{(r)} : i \in \I, r > -\langle \mu,\alpha_i\rangle\} \cup\{f_\beta^{(s)} : \beta \in \Delta^+, s \gge 1\}.
    $$
    In the definition of the PBW generators in $\bC[\cW_\mu]$,  we may choose them is such a way that $\sigma( e_\beta^{(s)}) = (-1)^{s+\operatorname{ht} \beta -1} f_{\tau \beta}^{(s)}$.  Then $\sigma$ acts on $V$, with invariants $V(1) \subset V$ having as basis the set
    $$
    \{ e_\beta^{(s)} + (-1)^{s+\operatorname{ht}\beta - 1} f_{\tau \beta}^{(s)} : \beta \in \Delta^+, s \gge 1\} $$
    $$
    \cup \ \{ h_i^{(2p)} : i \in \I_0, 2p > -\langle \mu, \alpha_i \rangle\} \ \cup \ \{ h_i^{(s)} + (-1)^s h_{\tau i}^{(r)} : i \in \I_1, r > -\langle \mu, \alpha_i\rangle \}.
    $$
    
Proposition \ref{prop: Wmu Poisson gens}(3) says that $\bC[\cW_\mu] = S(V)$ is a symmetric algebra on $V$.  By Lemma~ \ref{lem:symdirac} the composed map $S\big(V(1)\big) \hookrightarrow \bC[\cW_\mu] \twoheadrightarrow \bC[\iW_\mu]$ induces an isomorphism 
    $S\big(V(1)\big) \xrightarrow{\sim} \bC[\iW_\mu]$. 
To complete the proof, it suffices to prove that the images of the elements $e_\beta^{(s)} + (-1)^{s + \operatorname{ht} \beta -1} f_\beta^{(s)} $ of $V(1)$ can all be written in terms of the PBW generators from the statement of the lemma. This follows by an induction on $\operatorname{ht} \beta$, similarly to the proof of \cite[Lemma 2.3]{LWZ25degen}.  
\end{proof}

For the remainder of this section, we denote by $\yt_\mu$ the Poisson algebra with generators $h_i^{(r)}$ and $b_i^{(s)}$ for $i\in \I$, $r > -\langle \mu, \alpha_i\rangle$, and $s \gge 1$, with relations (\ref{eq:csty0})--(\ref{eq:csty7}).  By the previous two lemmas, there is a natural surjective Poisson map $\yt_\mu \twoheadrightarrow \bC[\iW_\mu]$ by sending $h_i^{(r)}, b_i^{(s)}$ to the same-named elements.

We can now prove Part (2) of Theorem \ref{thm:iWmuPoisson}.

\begin{lem}
\label{lem:iWmu_pres}
    The map $\yt_\mu \twoheadrightarrow \bC[\iW_\mu]$ is an isomorphism of Poisson algebras.
\end{lem}
\begin{proof}
    It is enough to show that $\yt_\mu$ is spanned by monomials in the same PBW generators from Lemma \ref{lem:iWmuPBW}.  This is essentially a  Poisson algebra variation on the proof of \cite[Proposition~2.13]{LWW25sty}.  

    First, we claim that $\yt_\mu$ is spanned by elements of the form
    $b_1 \cdots b_k h_{i_1}^{(m_1)} h_{i_2}^{(m_2)} \cdots h_{i_\ell}^{(m_\ell)},$
    where $b_1,\ldots,b_k$ are each nested commutators of the elements $b_i^{(s)}$. It suffices to prove that the subalgebra $S \subseteq \yt_\mu$ generated by these  elements is closed under Poisson brackets. By the Leibniz rule, it is in turn sufficient to show the Poisson bracket of any two generators of $S$ is still in $S$.  The only interesting case is to show that  $\{h_i^{(r)}, b\} \in S$, where $b$ is any nested commutator in the $b_i^{(s)}$. This follows by a straightforward induction on the length of $b$, using the relation (\ref{eq:csty2}) in $\yt_\mu$.
    
    Next, consider the filtration on $\yt_\mu$ defined by setting $\deg b_j^{(s)} = 1$ and $\deg h_i^{(r)}=0$. More precisely, these degrees define a grading on the free Lie algebra $L$ generated by symbols $h_i^{(r)}$ and $b_i^{(s)}$, and thus on the symmetric algebra $S(L)$. There is a surjection $S(L) \twoheadrightarrow \yt_\mu$ and the grading on $S(L)$ thus induces a filtration on $\yt_\mu$ (a similar and fuller discussion can be found in the paragraph right after this proof below).  Note that the Poisson bracket is degree zero.

    Similarly to the proof of \cite[Proposition~2.13]{LWW25sty}, one can argue that the associated graded algebra $\gr~\yt_\mu$ is spanned by monomials in the images of the PBW generators from Lemma \ref{lem:iWmuPBW}.  This is because the images $\overline{b_i^{(s)}} \in \gr~\yt_\mu$ satisfy the Poisson relations (\ref{eq:csy5}) and (\ref{eq:csy7}) which define the positive half of the algebra $\bC[\cW_\mu]$. This completes the proof.        
\end{proof}

To complete the proof of the remaining parts (1) and (4) of Theorem \ref{thm:iWmuPoisson}, we turn to the shifted iYangian $\Yt_\mu$ and its filtrations.  We take an indirect approach, inspired by the proof of \cite[Theorem 3.4]{T23}.  Consider the set
\begin{align}  \label{set:X}
X = \left\{ H_i^{(r)}, B_i^{(s)}  \mid i \in \I, r \in \bZ, s \gge 1\right\}
\end{align}
and the free Lie algebra $L$ on this set.  Then there is a surjection $U(L) \cong \bC\langle X\rangle \twoheadrightarrow \Yt_\mu$.  Define a grading on $L = \bigoplus_{n \in \bZ}L_n$ by declaring that $\deg H_i^{(r)} = r + \langle \mu, \alpha_i\rangle - 1$ and $\deg B_i^{(s)} = s + \langle \mu_1, \alpha_i\rangle -1$. Then we may define a filtration on the enveloping algebra $U(L)$, such that $L_n$ lives in filtered degree $n+1$. The PBW theorem for $U(L)$ ensures that $\gr~U(L) \cong S(L)$ with its standard Poisson structure.  In particular, $\gr~U(L)$ is Poisson generated by (the images of) the set $X$.
Via the surjection $U(L) \twoheadrightarrow \Yt_\mu$ we induce a corresponding (quotient) filtration $G^\bullet_{\mu_1} \Yt_\mu$ on $\Yt_\mu$.   Note that $\gr^{G_{\mu_1}^\bullet} \Yt_\mu$ is commutative and identified with a quotient of $S(L)$.

\begin{lem}
    There is a surjective homomorphism of graded algebras $\yt_\mu \twoheadrightarrow \gr^{G_{\mu_1}^\bullet}\Yt_\mu$, sending the generators $h_i^{(r)}, b_i^{(s)}$ to the classes  $\bar H_i^{(r)}, \bar B_i^{(s)}$.
\end{lem}
\begin{proof}
    Looking at the top degree parts of the defining relations \eqref{def0}--\eqref{SerreIII2} of $\Yt_\mu$, one can see that the above assignment on generators defines a homomorphism $\yt_\mu \rightarrow \gr^{G_{\mu_1}^\bullet}\Yt_\mu$. Let us denote the image of generators of $\Yt_\mu$ in $\gr^{G_{\mu_1}^\bullet}\Yt_\mu$ by $\bar H_i^{(r)}, \bar B_j^{(s)}$ and so on. 
    For example, modulo lower terms the  relation \eqref{hbNqs} becomes
   \begin{align}  
   \label{hbNqs_gr}
    [\bar H_i^{(r+2)},\bar B_j^{(s)}]-[\bar H_i^{(r)},\bar B_j^{(s+2)}]=(c_{ij}-c_{\tau i,j}) \bar H_{i}^{(r+1)}\bar B_j^{(s)} +(c_{ij}+c_{\tau i,j}) \bar H_{i}^{(r)}\bar B_j^{(s+1)}.
   \end{align}
    This is simply an inductive version of the relation \eqref{eq:csty2}. Similarly \eqref{SerreIII} reduces to \eqref{eq:csty6}, cf.~\cite[(6.18)]{LWZ25degen}. 
    Now we show that \eqref{SerreIII2} reduces to \eqref{eq:csty7}. Indeed, with the help of \eqref{eq:csty2} (corresponding to \eqref{hbNqs_gr}), we rewrite \eqref{SerreIII2} in the associated graded after cancellation of many summands as
    \begin{align*}
    & \mathrm{Sym}_{s_1,s_2}\big[\bar B_{i}^{(s_1)},[\bar B_{i}^{(s_2)},\bar B_{\tau i}^{(s)}]\big]
    \\
&=\frac{4}{3}\mathrm{Sym}_{s_1,s_2}(-1)^{s_1-1}
\sum_{p=0}^{s_1+s-2}3^{-p} 
\Big(\sum_{q\gge 0} 3 \bar H_{\tau i}^{(s_1+s-p-2-2q)}\bar B_j^{(s_2+p+2q)} \\&\hskip 6cm -\sum_{q\gge 0}  \bar H_{\tau i}^{(s_1+s-p-3-2q)}\bar B_j^{(s_2+p+2q+1)}\Big)
\\
&=\frac{4}{3}\mathrm{Sym}_{s_1,s_2}(-1)^{s_1-1}
\sum_{q\gge 0} 3 \bar H_{\tau i}^{(s_1+s-2-2q)}\bar B_j^{(s_2+2q)}.
    \end{align*}
This matches \eqref{eq:csty7}.    Finally, this homomorphism  $\yt_\mu \rightarrow \gr^{G_{\mu_1}^\bullet}\Yt_\mu$ is surjective since $\gr~U(L)$ is Poisson generated by the set $X$ in \eqref{set:X}, and thus so is its quotient $\gr~U(L) \twoheadrightarrow\gr^{G_{\mu_1}^\bullet}~\Yt_\mu$. 
    \end{proof}

Recall that  in \S \ref{ssec:filtrations on stY}, we defined a \emph{vector space} filtration $F_{\mu_1}^\bullet \Yt_\mu$, which a priori depends on a choice of PBW basis for $\Yt_\mu$.  The following technical result is key:
\begin{prop}
    \label{prop: equal filtrations}
    The filtrations $F_{\mu_1}^\bullet \Yt_\mu$ and $G_{\mu_1}^\bullet \Yt_\mu$ coincide.  
\end{prop}
In particular, this result establishes Proposition \ref{prop:algfilt}, since by construction $G_{\mu_1}^\bullet  \Yt_\mu$ is an algebra filtration and does not depend on any choice of PBW basis.
\begin{proof}
We adopt a general proof strategy used in \cite{FKPRW,KPW22}: the $\mu=0$ case first, then antidominant $\mu$, and finally all remaining $\mu$.

For each $k\in \bZ$ there is a natural inclusion
$F_{\mu_1}^k (\Yt_\mu) \subseteq G_{\mu_1}^k (\Yt_\mu),$
so we must establish  the opposite inclusion. Similarly to Remark \ref{rem: change of filtration}, for each $\mu$ is suffices to prove this claim for a single choice of decomposition $\mu = \mu_1 + \tau \mu_1$.  

First, consider the case of $\mu = 0$, with $\mu_1 = 0$. In this case all associated filtrations and gradings are non-negative: $F_0^k( \Yt_0) = G_0^k(\Yt_0)  = (\yt_0)_k = 0$ for $k < 0$.  By definition $F_{\mu_1}^k (\Yt)$ has a basis consisting of those PBW basis elements which have total degree $\lle k$. There are finitely many such monomials. In addition, analogous monomials label a basis for the degree $\lle k$ part of $ \yt_0 \cong \bC[\iW_0]$.  Due to the surjection of graded algebras $\yt_0 \twoheadrightarrow \gr^{G_0^\bullet}\Yt_0$, the dimension of $\dim G_0^k(\Yt_0)$ is at most that of the degree $\lle k$ part of $\yt_0$. Put together, we obtain the estimate
$$
\dim F_{0}^k (\Yt_0) \ \gge \ \dim G_{0}^k (\Yt_0).
$$
As $F_0^k(\Yt_0) \subseteq G_0^k(\Yt_0)$, we must have $F_0^k(\Yt_0) = G_0^k(\Yt_0)$ as claimed.

Next, let $\mu$ be antidominant. Since $\mu$ is even and $\tau \mu = \mu$, we can write $\mu = 2 \mu_1$ with $\tau \mu_1 = \mu_1$.  Recall the algebra ${}^\imath\wtl\Y$ from \cite[Definition.~2.17]{LWW25sty}, and define a filtration $F_{\mu_1}^k( {}^\imath\wtl\Y)$ as the span of PBW monomials of degree $\lle k$ where
$$
\deg \wtl H_i^{(r)} = r + \langle \mu, \alpha_i\rangle, \qquad \deg \wtl B_\beta^{(s)} = s + \langle \mu_1, \beta\rangle.
$$
Similarly, we may define a filtration $G_{\mu_1}^\bullet( {}^\imath\wtl\Y)$ analogously to the definition of $G_{\mu_1}^\bullet(\Yt_\mu)$.  Recall also from the proof of \cite[Lemma~2.18]{LWW25sty} that there is an algebra embedding ${}^\imath \wtl\Y \hookrightarrow \Yt_0 \otimes \bC[\xi_i \mid i \in \I]$. If we consider the tensor product of the filtration $F_0^\bullet(\Yt_0) = G_0^\bullet(\Yt_0)$ with the filtration on $\bC[\xi_i\mid i\in\I]$ defined by giving $\deg \xi_i = \langle \mu_1, \alpha_i\rangle$, then one can see that our algebra embedding is strictly filtered, for both filtrations. It follows that $F_{\mu_1}^\bullet({}^\imath \wtl \Y) = G_{\mu_1}^\bullet({}^\imath \wtl \Y)$.

Finally, for general $\mu$ we may again write $\mu = 2 \mu_1$. Choose any shift homomorphism $\iota_{\mu,\nu}^\tau:  \Yt_\mu \hookrightarrow \Yt_{\mu+\nu+\tau\nu}$ as in Lemma \ref{lem:shiftqs}, such that $\mu+\nu+\tau \nu$ is antidominant. Then it is easy to see that the image of $G_{\mu_1}^k(\Yt_\mu)$ lies in $G_{\mu_1+\nu}^k(\Yt_{\mu+\nu+\tau\nu})$. Since $\mu+\nu+\tau\nu$ is antidominant, we have already proven  that $F_{\mu_1+\nu}^k(\Yt_{\mu+\nu+\tau\nu}) = G_{\mu_1+\nu}^k(\Yt_{\mu+\nu+\tau\nu})$. Therefore:
\begin{align*}
\iota_{\mu,\nu}^\tau\big( F_{\mu_1}^k(\Yt_\mu)\big) \subseteq \iota_{\mu,\nu}^\tau\big( G_{\mu_1}^k(\Yt_\mu)\big)
 &\subseteq \iota_{\mu,\nu}^\tau (\Yt_\mu) \cap G_{\mu_1+\nu}^k(\Yt_{\mu+\nu+\tau\nu}) \\
&= \iota_{\mu,\nu}^\tau (\Yt_\mu) \cap F_{\mu_1+\nu}^k(\Yt_{\mu+\nu+\tau\nu}).
\end{align*}
But by Lemma \ref{lem:strictlyfiltered} the right hand side above is exactly $\iota_{\mu,\nu}^\tau\big( F_{\mu_1}^\bullet(\Yt_\mu)\big).$  As $\iota_{\mu,\nu}^\tau$ is injective we conclude that $F_{\mu_1}^k(\Yt_\mu) = G_{\mu_1}^k (\Yt_\mu)$, which completes the proof.
\end{proof}

We can now prove the remaining Parts (1) and (4) of Theorem \ref{thm:iWmuPoisson}. Since $F_{\mu_1}^\bullet \Yt_\mu$ is defined in terms of a PBW basis, the algebra $\gr^{F_{\mu_1}^\bullet} \Yt_\mu$ inherits a PBW basis. On the other hand, Lemmas \ref{lem:iWmuPBW} and \ref{lem:iWmu_pres} show that $\bC[\iW_\mu] \cong \yt_\mu $ has an analogous PBW basis.  These are mapped onto one another under the surjection (cf. Proposition~\ref{prop: equal filtrations})
$$
\bC[\iW_\mu] \cong \yt_\mu \twoheadrightarrow \gr^{G_{\mu_1}^\bullet}\Yt_\mu = \gr^{F_{\mu_1}^\bullet} \Yt_\mu.
$$
Therefore this map is a $\Qt$--grading preserving isomorphism, which proves Part (1). For Part (4), note that there are two shift homomorphisms $\bC[\iW_{\mu+\nu+\tau \nu}] \rightarrow \bC[\iW_\mu]$: one defined geometrically in Corollary \ref{prop:iWmu}, and the other inherited from Lemma \ref{lem:shiftqs} by taking associated graded.  They are both Poisson, and have the same effect on the Poisson generators of $\bC[\iW_\mu]$.  Therefore they are equal, which proves (4).   
This completes the proof of Theorem \ref{thm:iWmuPoisson}.

\section{Fixed point loci of affine Grassmannian slices}
\label{sec:islices}

In this section, for any even spherical coweight $\mu$, we show that the Poisson involution $\sigma$ preserves the generalized affine Grassmannian slice $\overline{\cW}_\mu^\lambda$ exactly when $\lambda$ is $\tau$-invariant. Then we study the affine Grassmannian {\em islices}, i.e., the fixed point loci $\iWbar_\mu^\lambda = (\overline{\cW}_\mu^\lambda)^\sigma$, and connect them to iGKLO.

\subsection{Generalized affine Grassmannian slices}

Let $\lambda \gge \mu$ be coweights with $\lambda $ dominant, and recall the space $\cW_\mu$ from \eqref{eq:Wmu}.   Define the \emph{generalized affine Grassmannian slice} as a closed subscheme of $\cW_\mu$:
\begin{equation}
\label{eq:Wmula}
\overline{\cW}_\mu^\lambda \ = \ \cW_\mu \cap \overline{ G[z] z^\lambda G[z]}.
\end{equation}
This is an irreducible affine variety over $\bC$ of dimension $\langle 2 \rho, \lambda -\mu\rangle$ \cite{BFN19}, and is a closed Poisson subvariety of $\cW_\mu$.  The $\bC^\times$--actions from (\ref{eq:loop rotation action}) all preserve $\overline{\cW}_\mu^\lambda \subset \cW_\mu$, as does the action of $T$ by conjugation. The shift maps (\ref{eq:shift map}) restrict to birational Poisson maps $\iota_{\mu,\nu_1,\nu_2}: \overline{\cW}_{\mu+\nu_1 + \nu_2}^{\lambda+\nu_1+\nu_2} \rightarrow \overline{\cW}_\mu^\lambda$ for any dominant coweight $\lambda \gge \mu$ such that $\lambda+\nu_1+\nu_2$ is also dominant \cite[Proposition~4.11]{KPW22}.

\begin{rem}
    \label{rem:modular}
    The varieties $\overline{\cW}_\mu^\lambda$ admit an alternate description as a moduli space of $G$-bundles on $\mathbb{P}^1$, see \cite[\S 2]{BFN19} for details.
\end{rem}

For $\mu$ dominant, $\overline{\cW}_\mu^\lambda$ are the usual affine Grassmannian slices between spherical Schubert varieties in the affine Grassmannian $\operatorname{Gr}_G$.
For general $\mu$, generalized affine Grassmannian slices are Coulomb branches for quiver gauge theories \cite{BFN19}.   They have symplectic singularities, as proven for $\mu$ dominant by \cite[Theorem 2.7]{KWWY14} and in general by \cite{Z20} (and also see \cite{B23}). In general, $\overline{\cW}_\mu^\lambda$ is a union of finitely many symplectic leaves \cite{MW23}: 
\begin{equation}
    \label{eq: Wlamu symplecic leaves}
    \overline{\cW}_\mu^\lambda  = \bigsqcup_{\substack{\mu \lle \nu \lle \lambda, \\ \nu \text{ dominant}}} \cW_\mu^{\nu}.
\end{equation}

\begin{eg}
    \label{eg:SL2slices}
    Let $G = \operatorname{PGL}_2$.  Given coweights $\la \gge \mu$ with $\la$ dominant, we extract non-negative integers $\bw = \langle \la, \alpha\rangle$ and $\bv = \langle \la-\mu, \varpi\rangle$.  The slice $\overline{\cW}_\mu^\la$ has an explicit description as matrices with polynomial entries due to Kamnitzer, see \cite[Proposition 2.17]{BFN19}:
    \begin{equation}
        \overline{\cW}_\mu^\la \cong \left\{  \begin{pmatrix} \gklod & \gklob \\ \gkloc & \gkloa \end{pmatrix} ~ :~ \begin{array}{cl} (i) & \gkloa, \gklob, \gkloc, \gklod \in \bC[z], \\ (ii) & \gkloa \text{ is monic of degree } \bv, \\ (iii) & \gklob, \gkloc \text{ have degree } < \bv, \\ (iv) & \gkloa \gklod - \gklob \gkloc= z^\bw. \end{array}\right\}.
    \end{equation}
    Note that the condition $(iv)$ entails that $\gkloa$ divides $z^\bw + \gklob \gkloc$ with quotient $\gklod$, and in particular $\gklod$ is uniquely determined by the other matrix entries. 
\end{eg}

The matrix entries from the previous example admit generalizations to all types, which we  consider in the next section.

\subsection{Classical GKLO homomorphism}
\label{sec:classGKLO1}

For each fundamental weight $\varpi_i$ of $\g$ let $V(-\varpi_i)$ denote the corresponding irreducible representation of $\g$ with \emph{lowest weight} $-\varpi_i$. Let $s_i \in W$ denote the simple reflection for $i \in \I$. Choose nonzero extremal weight vectors $v_{-\varpi_i} \in V(-\varpi_i)_{- \varpi_i}$ and $v_{-s_i \varpi_i} \in V(-\varpi_i)_{- s_i \varpi_i}$, with corresponding dual vectors $v_{-\varpi_i}^\ast, v_{-s_i\varpi_i}^\ast \in V(-\varpi_i)^\ast$.

Let $\lambda\gge \mu$ be coweights with $\lambda$ dominant, and write $\lambda - \mu = \sum_i \bv_i \alpha_i^\vee$ as in \S \ref{ssec:quantumtorus}. For any element $g = n^+ h z^\mu n^- \in \overline{\cW}^\lambda_\mu$, we can extract various formal Laurent series in $z^{-1}$:
\begin{align}
    \gkloa_i(z)  & = z^{\langle\lambda, \varpi_i \rangle}\langle v_{-\varpi_i}^\ast,  g v_{-\varpi_i}\rangle = z^{\bv_i} \langle v_{-\varpi_i}^\ast, h v_{-\varpi_i}\rangle,  \label{eq:gkloa} \\
    \gklob_i(z) & = z^{\langle\lambda, \varpi_i \rangle}\langle v_{-s_i \varpi_i}^\ast,  g v_{-\varpi_i}\rangle =  z^{\bv_i}\langle v_{-s_i \varpi_i}^\ast, n^+ h v_{-\varpi_i}\rangle,  \label{eq:gklob} \\
    \gkloc_i(z) & = z^{\langle\lambda, \varpi_i \rangle}\langle v_{-\varpi_i}^\ast,  g v_{-s_i \varpi_i}\rangle = z^{\bv_i}\langle v_{-\varpi_i}^\ast, h n^- v_{-s_i \varpi_i}\rangle.\label{eq:gkloc} 
\end{align}
In fact, these are all polynomials in $z$, with $\gkloa_i(z)$ being monic of degree $\bv_i$, while $\gklob_i(z)$ and $\gkloc_i(z)$ have degree strictly less than $\bv_i$, see \cite[\S 2.2.1]{MW24}.  Writing $\gkloa_i(z) = \sum_s z^{\bv_i - s} \gkloa_i^{(s)}$, and defining $\gklob_i^{(s)}$ and $\gkloc_i^{(s)}$ similarly, these coefficients define regular functions on $\overline{\cW}_\mu^\lambda$.  In fact, the coefficients $\gkloa_i^{(s)}$ and $\gklob_i^{(s)}$ for $i \in \I$ form a system of birational coordinates on $\overline{\cW}_\mu^\lambda$ by the results of \cite{FKMM99} and \cite[\S 2]{BFN19}.  These functions satisfy the following fundamental equation, which stems from a relation for generalized minors due to Fomin-Zelevinsky \cite[Theorem 1.17]{FZ99}, see also \cite[(2.70)]{MW24}: for each $i\in \I$, recalling that $\bw_i = \langle \lambda, \alpha_i\rangle$, we have  
\begin{equation}
    \label{eq:gklorel}
    \gkloa_i(z) \gklod_i(z) - \gklob_i(z) \gkloc_i(z) = z^{\bw_i} \prod_{\substack{j \leftrightarrow i,~j \in \I}} \gkloa_j(z).
\end{equation}
Here $\gklod_i(z)  = z^{\langle\lambda, \varpi_i \rangle}\langle v_{-s_i \varpi_i}^\ast,  g v_{-s_i \varpi_i}\rangle$ which is again a polynomial in $z$. 

\begin{rem}
    We have defined $G$ to be the \emph{adjoint} group associated to $\g$, however $V(-\varpi_i)$ is only a representation of the \emph{simply-connected} group associated to $\g$.  Nevertheless, the functions \eqref{eq:gkloa}--\eqref{eq:gkloc} are well defined on the variety $\overline{\cW}_\mu^\lambda$. One can see this via zastava spaces, as in \cite[\S 2]{BFN19}. Alternatively, following  \cite[Remark 2.10]{MW24}, note that $U_1^\pm[\![z^{-1}]\!]$ and $T_1[\![z^{-1}]\!]$ are naturally independent of the choice of group, so the factors $n^\pm, h$ of $g = n^+ h z^\mu n^-$ all act on $V(-\varpi_i)(\!(z^{-1})\!)$. Letting $D$ denote the determinant of the Cartan matrix, then the factor $z^\mu$ also has a well-defined action on $V(-\varpi_i)(\!(z^{-\frac{1}{D}})\!)$. Thus \eqref{eq:gkloa}--\eqref{eq:gkloc} are all well-defined as Laurent series in $z^{-\frac{1}{D}}$, and in fact are Laurent series in $z^{-1}$. 
\end{rem}
Let $\bA^{(\lambda-\mu)}$ denote the variety of tuples $\big(\gkloa_i(z)\big)_{i \in \I}$ of monic polynomials $\gkloa_i(z)$ of degree $\bv_i$. Then $\bA^{(\lambda-\mu)}$ is an affine space with coordinates $\gkloa_i^{(s)}$ for $i\in \I$ and $1 \lle s \lle \bv_i$. The functions \eqref{eq:gkloa} define a natural map $\overline{\cW}_\mu^\lambda \rightarrow \bA^{(\lambda-\mu)}$ which is faithfully flat by \cite[Lemma 2.7]{BFN19}.
We may identify  $\bA^{(\lambda-\mu)} = \bA^{|\lambda-\mu|} / S_{\lambda-\mu}$ as the quotient of an affine space $\bA^{|\lambda-\mu|} = \prod_i \bA^{\bv_i}$ by the product of symmetric groups $S_{\lambda-\mu} = \prod_i S_{\bv_i}$.  Indeed, we denote natural coordinates on $\bA^{|\lambda-\mu|}$ by $\ow_{i,r}$ for $i\in \I$ and $1\lle r\lle \bv_i$, defined so that 
\begin{equation}
    \label{eq:rootsofgkloa}
    \gkloa_i(z) = \prod_{r=1}^{\bv_i} (z-\ow_{i,r})
\end{equation}
with $S_{\bv_i}$ acting by permuting these roots. Define also 
\begin{equation}
\ox_{i,r}^- = \gklob_i(\ow_{i,r}) \quad \text{and} \quad \ox_{i,r}^+ = \gkloc_i(\ow_{i,r}),
\end{equation}
and observe that because of the relation \eqref{eq:gklorel} these  must satisfy
\begin{equation}
    \label{eq:gklorel3}
    \ox_{i,r}^+ \ox_{i,r}^- = - \ow_{i,r}^{\bw_i} \prod_{j\leftrightarrow i,~j \in \I} \gkloa_j(\ow_{i,r}).
\end{equation}
The polynomials $\gklob_i(z), \gkloc_i(z)$ can be recovered from the $\ox_{i,r}^\pm$ by Lagrange interpolation:
\begin{equation}
    \label{eq:lagrange}
    \gklob_i(z) = \sum_{r=1}^{\bv_i} \Big( \prod_{s=1, s \neq r}^{\bv_i} \frac{z - w_{i,s}}{w_{i,r}-w_{i,s}}\Big) \ox_{i,r}^-, \quad \gkloc_i(z) = \sum_{r=1}^{\bv_i} \Big( \prod_{s=1, s \neq r}^{\bv_i} \frac{z - w_{i,s}}{w_{i,r}-w_{i,s}}\Big) \ox_{i,r}^+.
\end{equation}

\begin{rem}
    \label{rem:ogcgklo}
    The functions from \eqref{eq: Drinfeld gens Wmu} can also be expressed in terms of the above coordinates:
\begin{equation}
    \label{eq:changeofgens}
    \sum_r e_i^{(r)} z^{-r} = \frac{\gklob_i(z)}{\gkloa_i(z)}, 
    \quad 
    \sum_r h_i^{(r)} z^{-r} =z^{\bw_i} \prod_{j\in \I} \gkloa_j(z)^{- c_{ji}}, 
    \quad 
    \sum_r f_i^{(r)} z^{-r} = \frac{\gkloc_i(z)}{\gkloa_i(z)}.
\end{equation}
This leads to variations on the Lagrange interpolation formulas above, such as
$$
 \sum_r e_i^{(r)} z^{-r} = \sum_{r=1}^{\bv_i} \frac{1}{(z-w_{i,r})  \prod_{s=1, s \neq r}^{\bv_i}{(w_{i,r}-w_{i,s})}} \ox_{i,r}^-,
$$
where the right side is expanded as a Laurent series in $z^{-1}$.
\end{rem}

In a slight variation on the notation of \cite[\S 2]{BFN19}, we define an open subset $\mathring{\bA}^{|\lambda-\mu|} \subset \bA^{|\lambda-\mu|}$ as the complement to all diagonals $\ow_{i,r} = \ow_{j,s}$  for pairs $(i,r ) \neq (j,s)$ with $c_{ij} \neq 0$. Let  $\mathring{\bA}^{(\lambda-\mu)} =  \mathring{\bA}^{|\lambda-\mu|} / S_{\lambda-\mu}$. Then $\mathring{\bA}^{(\lambda-\mu)} \subset \bA^{(\lambda-\mu)}$ is open, consisting of tuples $(\gkloa_i(z))_{i \in \I}$ where each $\gkloa_i(z)$ has no multiple roots, and where $\gkloa_i(z)$ and $\gkloa_j(z)$ have no common roots whenever $c_{ij}=-1$. Define an open subset $U_\mu^\lambda \subseteq \overline{\cW}_\mu^\lambda$ as the fiber product 
\begin{equation} \label{eq:U}
U_\mu^\lambda = \overline{\cW}_\mu^\lambda \times_{\bA^{(\lambda-\mu)}} \mathring{\bA}^{(\lambda-\mu)}.
\end{equation}
The following definition provides an étale covering of $U_\mu^\lambda$.

\begin{dfn}
    \label{def:X}
    Consider an affine space $\bA^{3|\lambda-\mu|}$ with coordinates $\ow_{i,r}, \ox_{i,r}^+, \ox_{i,r}^-$ for $i\in \I$ and $1\lle r \lle \bv_i$.
    Let $X_\mu^\lambda \subseteq\bA^{ 3|\lambda-\mu|}$ be the locally closed subvariety defined by imposing the relations \eqref{eq:gklorel3} for all pairs $(i,r)$, and by also imposing that $\ow_{i,r} - \ow_{j,s} \neq 0$ for all pairs $(i,r ) \neq (j,s)$ where $c_{ij} \neq 0$. 
\end{dfn}

 There is a natural map $X_\mu^\lambda \rightarrow \bA^{|\lambda-\mu|}$ defined by projecting onto the coordinates $\ow_{i,r}$. Note that the image of $X_\mu^\lambda$ actually lands in the open set $\mathring{\bA}^{|\lambda-\mu|}$.  
 %
The following result is a variant of results from \cite[\S2]{BFN19} describing the birational geometry of $\overline{\cW}_\mu^\lambda$. The Poisson structure on $X_\mu^\lambda$  is related to the Poisson structure for zastava spaces in the étale coordinates of \cite{FKMM99}. Denote 
\[
\delta_{i\leftrightarrow j} = \begin{cases} 1, & \text{if } c_{ij} = -1, \\ 0, & \text{else.} 
\end{cases}
\]

\begin{thm}[\mbox{cf. \cite{BFN19}}]
    \label{prop:gklobirat}
    \begin{enumerate}
        \item The variety $X_\mu^\lambda$ is equipped with a Poisson structure defined by:
        \begin{align*}
        \{\ow_{i,r}, \ow_{j,s}\} &= 0, \quad \{\ow_{i,r}, \ox_{j,s}^\pm\} = \mp \delta_{i,j} \delta_{r,s} \ox_{j,s}^\pm,
        \\
        \{\ox_{i,r}^-, \ox_{j,s}^+\} &= \delta_{i,j} \delta_{r,s} \frac{\partial}{\partial \ow_{i,r}} \Big( w_{i,r}^{\bw_i} \prod_{j \leftrightarrow i, ~j \in \I} \gkloa_i(\ow_{i,r}) \Big),
        \\
        \{\ox_{i,r}^\pm, \ox_{j,s}^\pm\} &=\pm \frac{\delta_{i\leftrightarrow j}}{\ow_{i,r} - \ow_{j,s}} \ox_{i,r}^\pm \ox_{j,s}^\pm.     
        \end{align*}
 
        \item   There is a Poisson map $X_\mu^\lambda \rightarrow U_\mu^\lambda$, uniquely determined by
        $$
        \gkloa_i(z) = \prod_r (z-\ow_{i,r}),$$ 
        $$\gklob_i(z) = \sum_r \Big( \prod_{s \neq r} \frac{z-\ow_{i,s}}{\ow_{i,r} - \ow_{i,s}}\Big) \ox_{i,r}^-
        , \qquad \gkloc_i(z) = \sum_r \Big( \prod_{s \neq r} \frac{z-\ow_{i,s}}{\ow_{i,r} - \ow_{i,s}}\Big) \ox_{i,r}^+.
        $$
        This map is finite and faithfully flat, and identifies $$X_\mu^\lambda = U_\mu^\lambda \times_{{\bA}^{(\lambda-\mu)}} {\bA}^{|\lambda-\mu|}  = U_\mu^\lambda \times_{\mathring{\bA}^{(\lambda-\mu)}} \mathring{\bA}^{|\lambda-\mu|}.$$
    \end{enumerate}
\end{thm}

\begin{proof}
    An isomorphism $\overline{\cW}_\mu^\la \cong \mc M_C(\operatorname{GL}(V),\mathbf{N})$ with the Coulomb branch for a quiver gauge theory is established in \cite[Theorem 3.10]{BFN19}.  It follows that the base change $\overline{\cW}_\mu^\la \times_{\mathbb A^{(\lambda-\mu)}} \mathring{\mathbb{A}}^{|\lambda-\mu|}$ can be described using Coulomb branch techniques from \cite[\S 5 and \S 6]{BFN18}. Indeed, fix a maximal torus $\mathbf T \subset \operatorname{GL}(V)$ and let ${\bf N}' = \bigoplus_{i \in \I} \Hom(W_i, V_i)$ be the sub-representation of $\bf N$ where only framing arrows are left (see \cite[\S 3(iii)]{BFN19} for the definition of $\bf N$), there is an isomorphism:
    $$
    \overline{\cW}_\mu^\la \times_{\mathbb A^{(\lambda-\mu)}} \mathring{\mathbb{A}}^{|\lambda-\mu|} \cong \mc M_C(\bf T, \bf N') \times_{\mathbb A^{|\lambda-\mu|}} \mathring{\mathbb{A}}^{|\lambda-\mu|}.
    $$
    Denoting this space by $\widetilde X_\mu^\la$, there is a natural map $\widetilde X_\mu^\la \rightarrow U_\mu^\la$. It is finite and faithfully flat since it arises via base change along $\mathbb{A}^{|\lambda-\mu|} \rightarrow \mathbb{A}^{(\lambda-\mu)}$. It is étale by an application of \cite[Theorem~ 10]{W22}.  
    
    To complete the proof, it suffices to show that $\widetilde X_\mu^\la = X_\mu^\la$ .  In fact, both can be identified with the same open subset of a product of Kleinian singularities of type A.  For $\mathcal M_C(\bf T, \bf N')$ this follows from \cite[\S 3(vii)(a)]{BFN18} and \cite[\S 4(iv)]{BFN18}, since $\bf T$ is a product of multiplicative groups $\bC^\times_{i,r}$, each acting independently by scaling a space $\Hom(W_i, \bC_{i,r})$.  Meanwhile, in Definition \ref{def:X} note that all linear factors of $\gkloa_j(z)$ in the defining relation \eqref{eq:gklorel3} are invertible.  Thus up to renormalizing  $\ox_{i,r}^\pm$ this relation is also simply the Kleinian singularity $\ox_{i,r}^+ \ox_{i,r}^- = w_{i,r}^{\bw_i}$.
\end{proof}

\subsection{Affine Grassmannian islices $\iWbar_\mu^\lambda$}

Recall from Lemma \ref{lem: gradings and shifts iWmu} a Poisson involution $\sigma$ of $\cW_\mu$, for any even spherical coweight $\mu$.  We are interested in the fixed point loci of the subvarieties $\overline{\cW}_\mu^\lambda \subset \cW_\mu$ under the action of $\sigma$:
\begin{align} \label{eq:islices}
    \iWbar_\mu^\lambda := (\overline{\cW}_\mu^\lambda)^\sigma.
\end{align}
We shall refer to $\iWbar_\mu^\lambda$ as an affine Grassmannian {\em islice}.

\begin{thm}
    \label{thm:fixedWlm}
    Let $\mu$ be an even spherical coweight, and let $\lambda$ be a dominant coweight with $\lambda \gge \mu$. Then $\sigma$ preserves the subvariety $\overline{\cW}_\mu^\lambda \subset \cW_\mu$ if and only if $\tau \lambda = \lambda.$  In this case, the islice 
    $\iWbar_\mu^\lambda$
    inherits a Poisson structure via (doubled) Dirac reduction, and $\iWbar_\mu^\lambda \subset\iW_\mu$ is a Poisson embedding.  There is a decomposition into smooth strata:
    $$
    \iWbar_\mu^\lambda = \bigsqcup_{\substack{\mu \lle \nu \lle \lambda, \\ \nu \text{ dominant,}\\ \tau \nu = \nu} } \iW_\mu^\nu,
    \qquad
    \text{where }\iW_\mu^\nu := (\cW_\mu^\nu)^\sigma.
    $$
    The symplectic leaves of $\ \iWbar_\mu^\lambda$ consist of the connected components of the strata $\iW_\mu^\nu$. Finally, the shift maps restrict to Poisson maps $\iota_{\mu,\nu}^\tau : {}^{\imath} \overline{\cW}_{\mu+\nu+\tau \nu}^{\lambda + \nu + \tau \nu} \longrightarrow {}^{\imath}\overline{\cW}_\mu^\lambda$.
\end{thm}
\begin{proof}
    For any dominant coweight $\nu$, using (\ref{eq:sigma:zmu}) we have 
    $$
    \sigma( G[z] z^\nu G[z]) = \sigma(G[z]) \sigma(z^\nu) \sigma( G[z]) = G[z]  (-1)^{\tau \nu} z^{\tau \nu} G[z] = G[z]z^{\tau \nu} G[z].
    $$
    Since $\tau \nu$ is dominant, we have $G[z] z^\nu G[z] \cap G[z] z^{\tau \nu} G[z] = \emptyset$ unless $\tau \nu = \nu$. It follows that $\sigma$ maps $\cW_\mu^\nu$ isomorphically onto $\cW_\mu^{\tau \nu}$, and preserves $\cW_\mu^\nu$ if and only if $\tau \nu = \nu.$ Similarly $\sigma$ preserves $\overline{\cW}_\mu^\lambda$ if and only if $\tau \lambda =\lambda$. In this case the fixed point locus $\iWbar_\mu^\lambda$ inherits a Poisson structure by (doubled) Dirac reduction, and $\iWbar_\mu^\lambda \subset \iW_\mu$ is Poisson since $\overline{\cW}_\mu^\lambda \subset \cW_\mu$ is Poisson and (doubled) Dirac reduction is functorial. The restricted shift maps are therefore Poisson by Theorem \ref{thm:iWmuPoisson}(4). Finally, the  decomposition of $\iWbar_\mu^\lambda$ into symplectic leaves follows from  Lemma \ref{lem: Dirac reduction leaves} and (\ref{eq: Wlamu symplecic leaves}).
\end{proof}

\begin{rem}
    Similarly to Remark \ref{rem:modular}, the islice $\iWbar_\mu^\la$ admits a modular description as a moduli space of $G$-bundles.  More precisely, using the symmetric definition $\overline{\cW}_{\mu_1, \tau \mu_1}^\la $ of the slice $\overline{\cW}_\mu^\la$ from \cite[\S 2(v)]{BFN19},  we can cut out $\iWbar_\mu^\la \subset \overline{\cW}_{\mu_1, \tau \mu_1}^\la$ by imposing  that the bundles denoted $\mathcal{P}_\pm$ are related by the involution $\sigma$.   We leave the details to the interested reader.  
\end{rem}

\begin{eg}  
    \label{eg:SL2slices2}
    Let $G =\operatorname{PGL}_2$ and $\tau = \id$. The involution $\sigma$ on $G(\!(z^{-1})\!)$ is given by $g(z) \mapsto g(-z)^T$.  Take coweights $\la \gge \mu$ for $G$, with $\la$ dominant.  We assume that $\mu$ is even; note that it is equivalent to assume that $\la$ is even, since $\langle \lambda-\mu, \alpha\rangle = \langle \la - \mu, 2 \varpi\rangle = 2 \bv$. Recall the explicit model for $\overline{\cW}_\mu^\la$ from Example \ref{eg:SL2slices}.  After renormalizing by a multiple of the identity matrix (this is harmless since $G = \operatorname{PGL}_2$), in this model the involution takes the form, cf.~\eqref{eq:sigmaingklo}:
    $$
    \begin{pmatrix}
        \gklod(z) & \gklob(z) \\ \gkloc(z) & \gkloa(z) 
    \end{pmatrix} \longmapsto 
    (-1)^\bv \begin{pmatrix}
        \gklod(-z) & \gkloc(-z) \\ \gklob(-z) & \gkloa(-z) 
    \end{pmatrix}
    $$
    The fixed point locus can thus be described analogously to Example \ref{eg:SL2slices}:
        \begin{equation}
        \iWbar_\mu^\la \cong \left\{ 
         \begin{pmatrix} \gklod & \gklob \\ \gkloc & \gkloa \end{pmatrix} ~ :~ \begin{array}{cl} (i) & \gkloa, \gklob, \gkloc,\gklod \in \bC[z], \\ (ii) & \gkloa \text{ is monic of degree } \bv, \\ (iii) & \gklob, \gkloc \text{ have degree } < \bv, \\ (iv) & \gkloa\gklod-\gklob\gkloc = z^\bw, \\ (v)& \gkloa(z) = (-1)^\bv \gkloa(-z), \ \gkloc(z) = (-1)^\bv \gklob(z). \end{array}\right\}
    \end{equation}
    Similarly to Example \ref{eg:SL2slices}, the coordinate $\gklod$ is uniquely determined by the others by condition $(iv)$. Note that it  automatically satisfies $\gklod(z) = (-1)^\bv \gklod(-z)$, because of $(v)$, which also determines $\gkloc$ uniquely from $\gklob$. So in the most simple terms, the variety $\iWbar_\mu^\la$ consists of pairs $(\gkloa, \gklob)$ of polynomials satisfying the degree conditions ($ii$) and ($iii$), such that $\gkloa(z) = (-1)^\bv \gkloa(-z)$, and such that $\gkloa(z)$ divides $z^\bw + (-1)^\bv \gklob(z) \gklob(-z)$.

    In this case $\iWbar_\mu^\la$ is normal and a complete intersection.   When $\mu$ is dominant, we will give an alternate description of $\iWbar_\mu^\la$ in terms of Slodowy slices, see Example \ref{ex:rank1:islice}.  In particular these cases are all symplectic singularities.
\end{eg}

\subsection{Classical iGKLO for $\iWbar_\mu^\lambda$} 
\label{ssec:GKLOfpl} 

Let  $\mu$ be an even spherical coweight, and let $\tau \lambda =\lambda$ be a dominant coweight with $\lambda \gge \mu$.  Recall the iGKLO homomorphism $\Phi_\mu^\lambda : \Yt_\mu[\bm z] \rightarrow \mc A$ from Theorem~\ref{thm:GKLOquasisplit}. This homomorphism depends on a choice of orientation of the Dynkin diagram $\I$, which we fix from now on.

\subsubsection{Classical iGKLO}

We show that the classical limit of $\Phi_\mu^\lambda$ (or more precisely, its specialization at $\bm  z = 0$) is related to the fixed point scheme $\iWbar_\mu^\lambda$. We first define appropriate filtrations on these algebras. 

On one hand, we consider a filtration $F_{\mu_1}^\bullet(\Yt_\mu)$ as in \S \ref{ssec:filtrations on stY}. Recall that $\operatorname{gr}^{F_{\mu_1}^\bullet} \Yt_\mu \cong \bC[\iW_\mu]$ by Theorem \ref{thm:iWmuPoisson}.  For our purposes here,  we define the coweight $\mu_1$ for $\g$ as follows:
\begin{equation}  \label{eq:mu1}
    \langle \mu_1, \alpha_i\rangle  =  \left\{ \begin{array}{cl} \deg Z_i - \bv_i + \sum_{\substack{j \leftrightarrow i, \\ j \in \I}} \bv_j, & \text{if } i \in \I_{ \pm 1}, \\ \fkw_i- 2 \fkv_i+ \sum_{\substack{j\leftrightarrow i, \\ j \in \iI}} \fkv_j + \tfrac{1}{2} \Big(\varsigma_i + 2 \theta_i + \sum_{\substack{j \leftrightarrow i, \\ j \in \I_0}}\theta_j \Big), & \text{if }i \in \I_0. \end{array} \right.
\end{equation}
The following result is necessary so that the filtration $F_{\mu_1}^\bullet(\Yt_\mu)$ is defined.
\begin{lem}
    $\mu_1$ is an integral coweight, and $\mu = \mu_1 + \tau \mu_1$.
\end{lem} 
\begin{proof}
    Both parts follow from direct calculation using the definitions from \S \ref{ssec:quantumtorus}.  
    In particular note that for $i \in \I_0$ we have:
    $$
\langle \mu, \alpha_i\rangle = \bw_i - 2\bv_i + \sum_{\substack{j \leftrightarrow i, \\ j \in \I}} \bv_j = (2 \fkw_i + \varsigma_i) - 2(2 \fkv_i + \theta_i) + 2\sum_{\substack{j \leftrightarrow i, \\ j \in \I_{1}}} \fkv_j  + \sum_{\substack{j \leftrightarrow i, \\ j \in \I_0}} (2 \fkv_j + \theta_j).
$$
This is an even integer since $\mu$ is even, and $\langle \mu_1, \alpha_i\rangle$ is forced to be half of this value to ensure that $\mu = \mu_1 + \tau \mu_1$.
\end{proof}

On the other hand, in the algebra $\mc A$ we specialize all (central) variables $z_{i,s} $ to $ 0$, obtaining its quotient algebra $\mc A_{\bm z= 0}$.  Similarly to \cite[\S B(ii)]{BFN19}, define a filtration on $\mc A_{\bm z=0}$ by setting $\deg w_{i,r} = 1$ for all pairs $(i,r)$, and setting
\begin{align}  \label{filter:A}
\deg \dfo_{i,r}^{\pm 1} = \left\{ \begin{array}{cl} 0, & \text{if } i \in \I_{\pm 1}, \\ 
\pm \sum_{\substack{j \rightarrow i, \\ j \in \I_0}} \fkv_j \pm \tfrac{1}{2}\Big(- \varsigma_i+\sum_{\substack{j \leftrightarrow i, \\ j \in \I_0}} \theta_j\Big), & \text{if } i \in \I_0. \end{array} \right.
\end{align} 
Similarly to the previous lemma, we note that $\deg \dfo_{i,r}^{\pm 1}$ is always an integer. By the same argument as \cite[Proposition 4.4]{KWWY14}, the associated graded algebra is a localized polynomial ring:
\begin{equation}
\operatorname{gr} \mc A_{\bm z=0} = \bC[\ow_{i,r}, \dfo_{i,r}^{\pm 1}, (w_{i,r}\pm w_{i,r'})^{-1}, w_{i,r}^{-1} ~:~ i \in \iI, 1 \lle r \neq r' \lle \fkv_i ].
\end{equation}
Its Poisson structure is determined by $\{ \dfo_{i,r}^{\pm 1}, w_{j,s}\} = \pm \delta_{i,j} \delta_{r,s} \dfo_{i,r}^{\pm 1}$ and $\{w_{i,r}, w_{j,s}\} = \{ \dfo_{i,r}, \dfo_{j,s}\} = 0$.  As in \S\ref{ssec:quantumtorus} we extend the notation $\dfo_{i,r}$ to $i \in \I_{-1} $ by $\dfo_{i,r} = - \dfo_{\tau i, \fkv_i + 1 -r}.$ 

The next theorem is the main result of this section.
\begin{thm}
    \label{thm:ctgklo}
     Retain the notation for filtrations on $\Yt_\mu$ and $\mc A_{\bm z=0}$ as above. 
    \begin{enumerate}
    \item The open subscheme $U_\mu^\lambda \subset \overline{\cW}_\mu^\lambda$ from \eqref{eq:U} is invariant under $\sigma$. Its fixed point locus ${}^\imath U_\mu^\lambda$ is nonempty if and only if the parity condition \eqref{parity} holds, and in this case we have 
    $$
    \dim {}^\imath U_\mu^\lambda = 2 \sum_{i \in \iI} \fkv_i.
    $$

    \item The homomorphism $\Phi_\mu^{\lambda, \bm z =0} : \Yt_\mu \rightarrow \mc A_{\bm z=0}$ is filtered, and the associated graded map
    $\operatorname{gr} \Phi_\mu^{\lambda, \bm z =0} : \bC[ \iW_\mu ]\longrightarrow \operatorname{gr} \mc A_{\bm z =0}$
    defines the closure $\overline{C}_\mu^\lambda \subseteq \iWbar_\mu^\lambda$, i.e.,~the kernel of this map is the defining ideal of $\overline{C}_\mu^\lambda$. Here $C_\mu^\lambda \subseteq {}^\imath U_\mu^\lambda$ is a top-dimensional irreducible component.
    \end{enumerate}    
\end{thm}

The proof of this theorem will be given in \S\ref{ssec:pf:GKLOclassical}. By Theorem \ref{thm:ctgklo}, if $\iWbar_\mu^\lambda$ is irreducible then $\operatorname{gr} \Phi_\mu^{\lambda, \bm z =0}$ defines $\iWbar_\mu^\lambda$ itself, and hence it is important to determine when $\iWbar_\mu^\lambda$ is irreducible.

\subsubsection{Deformed iGKLO}

Although we specialize the parameters $\bm z=0$ above, this is for simplicity and not necessity. The variety $\overline{\cW}_\mu^\lambda$ admits a Beilinson-Drinfeld deformation $\overline{\cW}_\mu^{\underline{\lambda}}$ defined in \cite[\S 2(x)]{BFN19}. It is a closed subvariety $\overline{\cW}_\mu^{\underline{\lambda}} \subset \cW_\mu \times \bA^{|\lambda|}$, where $\bA^{|\lambda|} = \bA^{\sum_i \bw_i}$ is the affine space having points $\bm z = (z_{i,s})_{i \in \I, 1 \lle s \lle \bw_i}$. By definition, the fiber over $\bm z \in \bA^{|\lambda|}$ is
    \begin{equation}
    \label{eq:BDfiber}
    \overline{\cW}^{\underline{\lambda}, \bm z}_\mu = \cW_\mu \cap \overline{ G[z] z^{\lambda, \bm z} G[z]},
    \end{equation}
where $z^{\lambda, \bm z} = \prod_{i, s} (z-z_{i,s})^{\varpi_i^\vee} \in G(\!(z^{-1})\!)$. In particular, the fiber over $0 \in \bA^{|\lambda|}$ is exactly $\overline{\cW}_\mu^\lambda$. 

Now suppose that $\tau \lambda = \lambda$ and $\mu$ is even spherical.  Then the involution $\sigma$ of $G(\!(z^{-1})\!)$ induces an involution of $\overline{\cW}_\mu^{\underline{\lambda}}$, permuting the above fibers with $\sigma( \overline{\cW}_\mu^{\underline{\la}, \bm z}) = \overline{\cW}_\mu^{\underline{\la}, - \bm z}$.  Define an involution $\sigma$ of $\bA^{|\lambda|}$ by 
    $$\sigma(z_{i,s}) = \left\{ \begin{array}{cl} - z_{i, \bw_i - s+1}, & \text{ if } i \in \I_0, \\ - z_{\tau i, s}, & \text{ if } i \in \I \setminus \I_0. \end{array}\right.  $$
Then the map $\overline{\cW}_\mu^{\underline{\lambda}} \rightarrow \bA^{|\lambda|}$ is $\sigma$-equivariant and there is an embedding of fixed point loci
   \begin{align} \label{isliceDeform}
    \iWbar_\mu^{\underline{\lambda}} \subset \iW_\mu \times {}^\imath \bA^{|\lambda|}.
    \end{align}
Similarly to \eqref{eq:wcoords1} below, one sees that the fixed point locus ${}^\imath \bA^{|\lambda|}$ has coordinates $\{z_{i,s} : i \in \iI, 1 \lle s \lle \fkw_i\}$. Thus, a point $\bm z= (z_{i,s})_{i \in \iI, 1 \lle s \lle \fkw_i} \in {}^\imath \bA^{|\lambda|}$ is exactly as in \S \ref{ssec:quantumtorus}. The fixed point locus ${}^\imath U_\mu^{\underline{\lambda}}$ is an open subvariety of $\iWbar_\mu^{\underline{\lambda}}$. 
    
Let the filtration on $\Yt_\mu[\bm z]$ be defined as above for $\Yt_\mu$, and with the variables $z_{i,s}$ in  degree 1. We can now formulate a deformed version of Theorem \ref{thm:ctgklo}.
    \begin{thm} \label{thm:ctgklo:deform}
   The map $\Phi_\mu^\lambda : \Yt_\mu[\bm z] \rightarrow \mc A$ is filtered, and its associated graded map 
    $$
    \operatorname{gr} \Phi_\mu^\lambda : \bC[\iW_\mu \times {}^\imath\bA^{|\lambda|}] \longrightarrow \operatorname{gr} \mc A
    $$
defines the closure $\overline{C}_\mu^{\underline{\lambda}} \subseteq \iWbar_\mu^{\underline{\lambda}}$. Here $C_\mu^{\underline{\lambda}} \subseteq {}^\imath U_\mu^{\underline{\lambda}} $ is a top-dimensional irreducible component of dimension $\sum_{i \in \iI} (2 \fkv_i + \fkw_i)$.     
\end{thm}

Recall from Definition \ref{dfn:truncated stY} that we define the truncated shifted iYangian $\Yt_\mu^\lambda$ as the image of the map $\Phi_\mu^\lambda$.  Recall also from \S \ref{ssec:commsub} that we define a filtration $F_{\mu_1}^\bullet \Yt_\mu^\la$ as the quotient filtration from $\Yt_\mu[\bm z]$.

Then $\Yt_\mu^\la$ inherits natural filtrations: firstly as a quotient of $\Yt_\mu[\bm z]$, and second as a subalgebra of $\mc A$. Since $\Phi_\mu^\lambda$ is filtered, the first filtration is contained in the second. There is a canonical surjection, cf.~the proof of \cite[Theorem~5.8]{LWW25sty}:
$$
\gr \Yt_\mu^\la = \gr \operatorname{Im}(\Phi_\mu^\la) \twoheadrightarrow  \operatorname{Im}(\gr~ \Phi_\mu^\la) \cong \bC[ \overline{C}_\mu^{\underline{\la}}].
$$
We expect that it is an isomorphism:
\begin{conj}  \label{conj:TSTY:islices}
    The truncated shifted iYangian $\Yt_\mu^\lambda$ quantizes the variety $\overline{C}_\mu^{\underline{\lambda}} \subseteq \iWbar_\mu^{\underline{\la}}$. In particular, if $\iWbar_\mu^\la$ is irreducible then $\Yt_\mu^\la$ quantizes $\iWbar_\mu^{\underline{\la}}$.
\end{conj}

Proving this conjecture would be a question of comparing two filtrations on $\Yt_\mu^\la$: the quotient filtration from $\Yt_\mu[\bm z]$ as defined above, and the subspace filtration inherited from $\mc A$. The former filtration is contained within the latter, and if they are equal, then the conjecture holds.  

For truncated shifted (untwisted) Yangians, the analogous conjecture is proven for $\mu$ dominant in \cite[Corollary B.28]{BFN19} and for general $\mu$ in \cite[Theorem 3.13]{W19}.  

\begin{rem}
    A weak form of the above conjecture states that $\Yt_\mu^\la$ quantizes a scheme supported on $\overline{C}_\mu^{\underline{\la}}$.  A similar result was proven for  truncated shifted Yangians in \cite[Theorem 4.8]{KWWY14}.
\end{rem}

\subsection{Proof of Theorem \ref{thm:ctgklo} on ${}^\imath\overline{\cW}_\mu^\lambda$}
  \label{ssec:pf:GKLOclassical}

First, we observe that the involution $\sigma$ on $\overline{\cW}_\mu^\lambda$ maps 
\begin{equation}
    \label{eq:sigmaingklo}
    \begin{array}{cc}
    \sigma\big( \gkloa_i(z) \big) = (-1)^{\bv_i}\gkloa_{\tau i}(-z), & \quad \sigma\big( \gklob_i(z) \big) = (-1)^{\bv_i}\gkloc_{\tau i}(-z), \\
    \sigma\big( \gkloc_i(z) \big) = (-1)^{\bv_i} \gklob_{\tau i}(-z), & \quad \sigma\big( \gklod_i(z)\big) = (-1)^{\bv_i}\gklod_{\tau i}(-z).
    \end{array}
\end{equation}
Indeed, one sees this from the identity $\sigma(n^+ h z^\mu n^-) = \sigma(n^-) \sigma(h) z^\mu \sigma(n^+)$ for $n^+ h z^\mu n^- \in \iW_\mu$, together with \eqref{eq:sigma:gz} and the definitions \eqref{eq:gkloa}--\eqref{eq:gkloc}. See also \cite[Proposition~2.7]{FZ99}. Therefore, viewed as functions on $\iWbar_\mu^\lambda$ by restriction:
    \begin{equation}
        \label{cor:involutioningklo}
        \gkloa_i(z) = (-1)^{\bv_i} \gkloa_{\tau i}(-z), \quad \gkloc_i(z) = (-1)^{\bv_i} \gklob_{\tau i}(-z), \quad \gklod_i(z) = (-1)^{\bv_i} \gklod_{\tau i}(-z).
    \end{equation}
\begin{rem}
    The algebra $\bC[\iWbar_\mu^\lambda]$ is Poisson generated by the coefficients of $\gkloa_i(z)$ and $\gklob_i(z)$ for $i \in \I$.  This follows from Theorem \ref{thm:iWmuPoisson}(2) together with equations \eqref{eq: Drinfeld gens iWmu} and \eqref{eq:changeofgens}.
\end{rem}
    

Consequently, if we define an involution $\sigma$ on $\bA^{(\lambda-\mu)}$ by $\sigma\big(\gkloa_i(z)\big) = (-1)^{\bv_i} \gkloa_{\tau i}(-z)$, then the map $\overline{\cW}_\mu^\lambda \rightarrow \bA^{(\lambda-\mu)}$ from \S \ref{sec:classGKLO1} is $\sigma$-equivariant and induces a map on fixed point loci
\begin{equation}
    \iWbar_\mu^\lambda \longrightarrow {}^\imath \bA^{(\lambda-\mu)}.
\end{equation}
Applying  Lemma \ref{lem:symdirac}, we note that ${}^\imath \bA^{(\lambda-\mu)}$ is an affine space  with coordinates
\begin{equation}
    \{ \gkloa_i^{(s)} : i \in \I_0, 1 \lle s \lle \bv_i \text{ with } s \text{ even} \} \cup \{ \gkloa_i^{(s)} : i \in \I_1, 1 \lle s \lle \bv_i\}.
\end{equation}
Here, we view the $\gkloa_i^{(s)}$ as functions on ${}^\imath \bA^{(\lambda-\mu)}$ by restriction from $\bA^{(\lambda-\mu)}$. Recalling the notation $\fkv_i$ from \eqref{ell_theta}, we see that ${}^\imath \bA^{(\lambda-\mu)}$ has dimension $\sum_{i \in \iI} \fkv_i$.

We will next extend $\sigma$ to an involution of $\bA^{|\lambda-\mu|}$, such that the map $\bA^{|\lambda-\mu|} \rightarrow \bA^{(\lambda-\mu)}$ defined by \eqref{eq:rootsofgkloa} is $\sigma$-equivariant and well behaved (in the sense of Lemma \ref{lem:fixedcoloureddiv} and Corollary \ref{cor:Aflat}).   For $i \in \I$ and $1 \lle r \lle \bv_i$, we set
\begin{equation}
    \label{sigma:wir}
    \sigma(w_{i,r}) = - w_{\tau i, \tau r}, \qquad \text{ where }\quad \tau r:=\bv_i + 1 - r.
\end{equation}
%
This defines an involution $\sigma$ on the space $\bA^{|\lambda-\mu|}$. 

Using Lemma \ref{lem:symdirac}, there are various choices of coordinates on the fixed point scheme ${}^\imath \bA^{|\lambda-\mu|} $ defined by restriction from $\bA^{|\lambda-\mu|}$. For our next result, a good choice is  to use coordinates
\begin{equation}
    \label{eq:wcoords1}
    \{ w_{i,r} : i \in \iI, 1 \lle r \lle \fkv_i\}.
\end{equation}
The restriction of any other $w_{i,r}$ to ${}^\imath \bA^{|\lambda-\mu|}$ can be expressed in terms of these coordinates, using \eqref{sigma:wir}. Note that for any $i\in\I_0$ such that $\bv_i = 2\fkv_i+1$ is odd (or equivalently if $\theta_i = 1$, in the notation of \S \ref{ssec:quantumtorus}) the restriction $w_{i, \fkv_i+1} = 0$.   Hence we have $\dim {}^\imath \bA^{|\lambda-\mu|} =\sum_{i \in \iI} \fkv_i$.  Consider now the following group:
\begin{equation}
    {}^\imath S_{\lambda-\mu} =\prod_{i \in \I_0} \big( S_{\fkv_i} \ltimes ( \bZ/ 2 \bZ)^{\fkv_i}\big) \times  \prod_{i \in \I_1} S_{\fkv_i},
\end{equation}
Then there is an action of ${}^\imath S_{\lambda-\mu}$ on the space ${}^\imath \bA^{|\lambda-\mu|}$: each symmetric group factor $S_{\fkv_i}$ permutes the corresponding coordinates $\{w_{i,r}: 1 \lle r  \lle \fkv_i\}$ in \eqref{eq:wcoords1}, while each  $\bZ/2\bZ$ factor acts on its corresponding coordinate by $w_{i,r} \mapsto - w_{i,r}$.



\begin{lem}
\label{lem:fixedcoloureddiv}
The map $\bA^{|\lambda-\mu|} \rightarrow \bA^{(\lambda-\mu)}$ from \eqref{eq:rootsofgkloa} is $\sigma$-equivariant.  The induced map ${}^\imath \bA^{|\lambda-\mu|} \rightarrow {}^\imath \bA^{(\lambda-\mu)}$ on fixed point loci identifies
$$
{}^\imath \bA^{(\lambda-\mu)} = {}^\imath \bA^{|\lambda-\mu|} / {}^\imath S_{\lambda-\mu}.
$$
The open subsets $\mathring{\bA}^{|\lambda-\mu|}$ and $\mathring{\bA}^{(\lambda-\mu)}$ are preserved by $\sigma$, and their fixed point loci ${}^\imath \mathring{\bA}^{|\lambda-\mu|}$ and ${}^\imath \mathring{\bA}^{(\lambda-\mu)}$ are nonempty if and only if the parity condition \eqref{parity} holds.


\end{lem}

\begin{proof}
    The map $\bA^{|\lambda-\mu|} \rightarrow \bA^{(\lambda-\mu)}$ is determined by \eqref{eq:rootsofgkloa}. This map is $\sigma$-equivariant, since
    $$
    \sigma\Big(\prod_{r} (z-w_{i,r}) \Big) = \prod_r (z - \sigma(w_{i,r}))= \prod_r (z+ w_{\tau i, \tau r}) = \prod_r (z+ w_{\tau i, r})
    $$
    which agrees with the action  $\sigma\big( \gkloa_i(z)\big) = (-1)^{\bv_i} \gkloa_{\tau i}(-z)$ on $\bA^{(\lambda-\mu)}$.

    Now consider the restriction ${}^\imath \bA^{|\lambda-\mu|} \rightarrow {}^\imath \bA^{(\lambda-\mu)}$ of this map to fixed point loci. We use the coordinates \eqref{eq:wcoords1} on ${}^\imath \bA^{|\lambda-\mu|}$. If $i \in \I_1$, then we have $\gkloa_i(z) = \prod_{r=1}^{\fkv_i}(z-w_{i,r})$.     The coefficients of this polynomial generate the invariants for the action of $S_{\fkv_i}$ on $\bC[w_{i,1},\ldots,w_{i,\fkv_i}]$. Meanwhile, if $i \in \I_0$ then we have
    $$
    \gkloa_i(z) = z^{\theta_i}\prod_{r=1}^{\fkv_i} (z^2-w_{i,r}^2).
    $$
    The coefficients of this polynomial generate the invariants for the action of $S_{\fkv_i} \ltimes (\bZ/2\bZ)^{\fkv_i}$ on $\bC[w_{i,1},\ldots,w_{i,\fkv_i}]$. It follows that ${}^\imath \bA^{(\lambda-\mu)} = {}^\imath \bA^{|\lambda-\mu|} / {}^\imath S_{\lambda-\mu}$.
    
    Finally, observe that if there exist $i, j \in \I_0$ which are connected and have $\theta_i \theta_j = 1$, then this means that $\gkloa_i(z)$ and $\gkloa_j(z)$ both have odd degree.  Since $\gkloa_i(z) = - \gkloa_i(-z)$ and $\gkloa_j(z) = - \gkloa_j(-z)$, this means that these two polynomials must  have a common zero $z=0$. Therefore ${}^\imath \mathring{\bA}^{(\lambda-\mu)}$ is empty.  If no such $i,j$ exist, then there is no obstruction. The case of ${}^\imath \mathring{\bA}^{|\lambda-\mu|}$ is similar.
\end{proof}

\begin{cor} \label{cor:Aflat}
    The map ${}^\imath \bA^{|\lambda-\mu|} \rightarrow {}^\imath \bA^{(\lambda-\mu)}$ is finite and faithfully flat.
\end{cor}
\begin{proof}
    The group ${}^\imath S_{\lambda-\mu}$ is a product of Weyl groups of types B/C and A, acting on the product  $\bA^{|\lambda-\mu|}$ of affine spaces in the standard ways. Thus the quotient map ${}^\imath \bA^{|\lambda-\mu|} \rightarrow  {}^\imath \bA^{|\lambda-\mu|} / {}^\imath S_{\lambda-\mu}$ has the desired properties, by Chevalley's Theorem on Weyl group invariants. The claim now follows from the previous lemma.
\end{proof}

Recall that we define the open subset $U_\mu^\lambda \subseteq \overline{\cW}_\mu^\lambda$ as the preimage of $\mathring{\bA}^{(\lambda-\mu)} \subseteq \bA^{(\lambda-\mu)}$.  It  follows from the previous lemma that $U_\mu^\lambda$ is preserved by the involution $\sigma$.  We will study its fixed point locus ${}^\imath U_\mu^\lambda$, via the simpler space  $X_\mu^\lambda$ from Theorem \ref{prop:gklobirat}. 

\begin{lem}
    \label{lem:sigmaonX}
    There is a Poisson involution $\sigma$ of $X_\mu^\lambda$ defined by
    \begin{align}  \label{sigma:yir}
    \sigma(\ow_{i,r}) = - \ow_{\tau i, \tau r}, \qquad \sigma(\ox_{i,r}^\pm) = (-1)^{\bv_i} \ox_{\tau i, \tau r}^\mp
    \end{align}
    such that the map $X_\mu^\lambda \rightarrow \bA^{|\lambda-\mu|}$ is $\sigma$-equivariant.
\end{lem}
\begin{proof}
    Follows by Definition \ref{def:X} and Theorem \ref{prop:gklobirat}.
\end{proof}

Let ${}^\imath X_\mu^\lambda \subseteq X_\mu^\lambda$ denote the corresponding fixed point locus.  Viewed as functions on ${}^\imath X_\mu^\lambda$ by restriction, we thus have
\begin{equation}
    \label{eq:iXcoordrels}
    w_{i,r} = - w_{\tau i, \tau r}, \quad y_{i,r}^+ = (-1)^{\bv_i} y_{\tau i, \tau r}^-.
\end{equation}
Note that as $\sigma$ is a Poisson involution on $X_\mu^\lambda$, its fixed point locus ${}^\imath X_\mu^\lambda$ inherits a Poisson structure by (doubled) Dirac reduction.  The same is true of ${}^\imath U_\mu^\lambda$, and the map ${}^\imath X_\mu^\lambda \rightarrow {}^\imath U_\mu^\lambda$ is Poisson since Dirac reduction is functorial.

Recalling  the notations of \S \ref{ssec:quantumtorus}, observe that for any $i\in\I$, when restricted to ${}^\imath X_\mu^\lambda$ we have
\begin{equation}
    \gkloa_i(z) = \W_i(z).
\end{equation}

\begin{prop}
\label{prop:classicalgklo1}
Retain the Assumption \eqref{parity}. Then the space ${}^\imath X_\mu^\lambda$ has coordinates 
 $$
\{ w_{i,r}, \ox_{i,r}^- \mid i \in \I, 1 \lle r \lle \bv_i\},
 $$
 and is defined scheme-theoretically in these coordinates by the relations $(1)$--$(4)$:
\begin{enumerate}
    \item 
    $w_{\tau i, \tau r} = - w_{i,r}$, for $i\in \I$ and $1\lle r \lle \bv_i$;
    \item 
    \begin{equation}
        \label{eq:fixedgklo1}
        y_{i,r}^- y_{\tau i, \tau r}^- = (-1)^{\bv_i + 1} w_{i,r}^{\bw_i} \prod_{j \leftrightarrow i, \ j \in \I} \W_j(w_{i,r}),
        \qquad \text{for } i \in \iI,\, 1 \lle r \lle \fkv_i;
    \end{equation}
    \item or each $i \in \I_0$ with $\theta_i = 1$:
    \begin{equation}
        \label{eq:fixedgklo2}
        (y_{i,\fkv_i+1}^-)^2 = \left\{ \begin{array}{cl}         (-1)^{\sum_{j \leftrightarrow i, j \in \iI} \fkv_j} \prod_{j \leftrightarrow i,~j \in \iI} \prod_{s = 1}^{\fkv_j} w_{j,s}^2, & \text{ if } \bw_i= 0
        \\
        0, & \text{ if } \bw_i > 0;
        \end{array} \right.
    \end{equation}
    
    \item 
    $w_{i,r} - w_{j,s} \neq 0$, for all pairs $(i,r) \neq (j,s)$ with $i,j \in \I$ and $c_{ij} \neq 0$.
\end{enumerate}
In terms of these coordinates, the Poisson structure on ${}^\imath X_\mu^\lambda$ is determined by
\begin{align*}
\{ w_{i,r}, w_{j,s} \} &= 0, \qquad \{ w_{i,r}, y_{j,s}^-\} = (\delta_{(i,r), (j,s)} - \delta_{(\tau i, \tau r), (j,s)}) y_{i,r}^-,
\\
\{ y_{i,r}^-, y_{j,s}^-\} &= \frac{-\delta_{i\leftrightarrow j}}{w_{i,r}-w_{j,s}} y_{i,r}^- y_{j,s}^- + \delta_{(i,r), (\tau j, \tau s)} (-1)^{\bv_i} \frac{\partial}{\partial x}\Big( x^{\bw_i} \prod_{j \leftrightarrow i, j \in \I} \W_j(x) \Big) \Big|_{x = w_{i,r}}.
\end{align*}
\end{prop}

\begin{proof}
    The same map from Lemma \ref{lem:sigmaonX} defines an involution of the space $\bA^{3|\lambda-\mu|}$ from Definition~\ref{def:X}, and its fixed point locus ${}^\imath \bA^{3|\lambda-\mu|}$ is defined simply by imposing \eqref{eq:iXcoordrels}. Put differently, the fixed point locus is an affine space with coordinates $w_{i,r}, \ox_{i,r}^-$ for $i\in \I$, $1\lle r \lle \bv_i$, with the relation (1).  Now consider the open subset $\mathring\bA^{3|\lambda-\mu|} \subseteq \bA^{3|\lambda-\mu|}$ where $w_{i,r} - w_{j,s} \neq 0$ for all pairs $(i,r) \neq (j,s)$ with $c_{ij} \neq 0$. Then its fixed point variety ${}^\imath \mathring{\bA}^{3|\lambda-\mu|} \subseteq {}^\imath \bA^{3|\lambda-\mu|}$ is also open, defined by imposing condition (4).  Now observe by Definition \ref{def:X} that $X_\mu^\lambda \subseteq \mathring \bA^{3|\lambda-\mu|}$ is a  closed subscheme. It follows that ${}^\imath X_\mu^\lambda \subseteq {}^\imath \mathring{\bA}^{3|\lambda-\mu|}$ is also closed and is defined by imposing the defining relations \eqref{eq:gklorel3} of $X_\mu^\lambda$.  Keeping in mind the relations \eqref{eq:iXcoordrels}, this is equivalent to conditions (2) and (3) above, proving the first claim.

    The Poisson structure is calculated using (doubled) Dirac reduction via Remark \ref{rem: Dirac reduction Poisson}, cf.~the proof of Lemma \ref{lem:iWmuDrinfeld}.  For example, let us compute the bracket $\{y_{i,r}^-, y_{j,s}^-\}$ in $\bC[{}^\imath X_\mu^\lambda]$.  We  lift $y_{i,r}^- \in \bC[{}^\imath X_\mu^\la]$ to the $\sigma$--invariant element $\tfrac{1}{2} (y_{i,r}^- + (-1)^{\bv_i} y_{\tau i, \tau r}^+) \in \bC[X_\mu^\lambda]$, and compute the Poisson bracket of these lifts using Theorem \ref{prop:gklobirat}:
    \begin{align*}
        & 2 \left\{ \tfrac{1}{2}(y_{i,r}^- + (-1)^{\bv_i} y_{\tau i, \tau r}^+), \tfrac{1}{2}(y_{j,s}^- + (-1)^{\bv_j} y_{\tau j, \tau s})\right\} \\
        =~& \tfrac{1}{2} \frac{- \delta_{i \leftrightarrow j}}{w_{i,r} -w_{j,s}} y_{i,r}^- y_{j,s}^- + \tfrac{1}{2} \delta_{(i,r), (\tau j, \tau s)} (-1)^{\bv_i}  \frac{\partial}{\partial x} \Big( x^{\bw_i} \prod_{k \leftrightarrow i, k \in \I} \W_k(x) \Big) \Big|_{x = w_{i,r}} \\
        &\hskip 3.2cm- \tfrac{1}{2} \delta_{(\tau i, \tau r), (j,s)} (-1)^{\bv_j} \frac{\partial}{\partial v} \Big( v^{\bw_j} \prod_{ \ell \leftrightarrow j, \ell \in \I} \W_\ell(v) \Big) \Big|_{v = w_{j,s}} \\&\hskip 3.25cm+ \tfrac{1}{2} (-1)^{\bv_i + \bv_j} \frac{\delta_{\tau i \leftrightarrow \tau j}}{w_{\tau i, \tau r}-w_{\tau j, \tau s}}y_{\tau i, \tau r}^+ y_{\tau j, \tau s}^+.
    \end{align*}
    Upon restricting to ${}^\imath X_\mu^\lambda$, the first and last terms above become equal because of \eqref{eq:iXcoordrels}. Let us show that the second and third terms also become equal upon restriction. If $ (\tau i, \tau r)= (j, s)$ then we know that $w_{j,s} = - w_{i,r}$, $\bv_j = \bv_i$, and $\bw_i = \bw_j$.  There is also a bijection between the sets $\{\ell \in \I : \ell \leftrightarrow j\}$ and $\{k \in \I : k \leftrightarrow i\}$ defined by $k = \tau \ell$, under which the corresponding polynomials satisfy $\W_\ell(v) = (-1)^{\bv_k} \W_k(-v)$. With this in mind, and making the substitution $x = - v$ of dummy variables, we get:
    \begin{align*}
    - (-1)^{\bv_j} \frac{\partial}{\partial v}&\Big( v^{\bw_j} \prod_{ \ell \leftrightarrow j, \ell \in \I} \W_\ell(v) \Big) \Big|_{v = w_{j,s}}\\& = (-1)^{\bv_i} \frac{\partial}{\partial x}\Big( (-x)^{\bw_i} \prod_{ k \leftrightarrow i, k \in \I} (-1)^{\bv_k} \W_k(x) \Big) \Big|_{x=w_{i,r}}.
    \end{align*}
    Within  the derivative on the right side there is a sign of $(-1)^{\bw_i + \sum_{k \leftrightarrow i} \bv_k} = (-1)^{\mu_i} = 1$ since $\mu$ is even. This gives the claimed equality of the second and third terms.
\end{proof}

 As in \S \ref{ssec:quantumtorus}, for each $i\in \I_1$ we fix a choice of $\zeta_i \in \bN$ with $1 \lle \zeta_i \lle \fkv_i = \bv_i$, and extend it to $i \in \I_1\cup\I_{-1}$ by $\zeta_{\tau i} = \fkv_i - \zeta_i$. Recall also our conventions from \S \ref{ssec:GKLOfpl}.  The following result defines the classical iGKLO homomorphism.
\begin{lem}
\label{lem:classicalgklo2}
    Retain the Assumption \eqref{parity} and fix an orientation of the Dynkin diagram $\I$ satisfying the conditions from \S\ref{ssec:GKLOqs}. Then there is a Poisson homomorphism 
    \begin{equation} \label{eq:XtoA}
    \bC[{}^\imath X_\mu^\lambda] \longrightarrow \operatorname{gr} \mc A_{\bm z = 0} 
    \end{equation}
    defined by (1)--(3) below:
    \begin{enumerate}
        \item 
        $w_{i,r} \mapsto w_{i,r}$, for $i\in \I$ and $ 1\lle r \lle \bv_i$;
        \item 
        $$
        y_{i,r}^- \mapsto \left\{ \begin{array}{cl} - w_{i,r}^{\fkw_i} \prod_{j \rightarrow i} \W_j(w_{i,r}) \dfo_{i,r}^{-1} & \text{if } 1 \lle r \lle \zeta_i, \\ - w_{i,r}^{\fkw_i} \prod_{\tau j \leftarrow \tau i} \W_{\tau j}(w_{\tau i, r}) \dfo_{\tau i, r}, & \text{if } \zeta_i < r \lle \bv_i, \end{array} \right.
        $$
        for $i \in \I_1 \cup \I_{-1}$ and $1 \lle r \lle \bv_i$; 
        \item 
        $$
        y_{i,r}^- \mapsto \left\{ 
        \begin{array}{cl} - w_{i,r}^{\fkw_i} \prod_{j \rightarrow i} \W_j(w_{i,r}) \prod_{\substack{j \leftarrow i, \\ j \in \I_0}} \overline{W}_j^-(w_{i,r})\dfo_{i,r}^{-1}, & \text{if } 1 \lle r \lle \fkv_i, \\ 
        - w_{i,r}^{\fkw_i+\varsigma_i} \prod_{\substack{j \rightarrow  i, \\ j \in \I_{\pm 1}}} \W_{j}(w_{i, \fkv_i-r}) \prod_{\substack{j \leftarrow i, \\ j \in \I_0}} W_j(w_{i,\fkv_i-r}) \dfo_{i, \fkv_i-r}, & \text{if } \bv_i - \fkv_i< r \lle  \bv_i, \\ 
        \sqrt{(-1)^{\sum_{j\leftrightarrow i, j\in \iI} \fkv_i}}  \prod_{j \leftrightarrow i} W_j(0), & \begin{array}{c}\text{if } r = \fkv_i + 1,  \text{ } \theta_i = 1 \\ \text{and } \bw_i=0, \end{array} \\ 0, &  \begin{array}{c}\text{if } r = \fkv_i + 1,  \text{ } \theta_i = 1 \\ \text{and } \bw_i>0, \end{array} \end{array} \right.
        $$
        for $i \in \I_0$ and $1 \lle r \lle \bv_i$.
    \end{enumerate}
\end{lem}

\begin{proof}
After unwinding the definitions it is straightforward to see that this assignment satisfies the relations (1)--(4) from Proposition \ref{prop:classicalgklo1}, and thus defines an algebra homomorphism $\bC[{}^\imath X_\mu^\lambda] \rightarrow \operatorname{gr} \mc A_{\bm z=0}$.  One proves that it is a Poisson map by a direct case-by-case calculation, essentially an application of the following basic identity: denoting $P = (w_{i,r} - w_{j,s})^A (w_{i,r} + w_{j,s})^B \dfo_{i,r}^{-1}$ and $Q = (w_{j,s} - w_{i,r})^C ( w_{j,s}+ w_{i,r})^D \dfo_{j,s}^{-1}$ where $A,B,C,D \in \bN$, one finds that
$$
\{ P, Q\} = \frac{-A-C}{w_{i,r} - w_{j,s}}PQ + \frac{B-D}{w_{i,r} + w_{j,s}} PQ.
$$
We leave the details to the reader.
%
\end{proof}

\begin{prop}
\label{prop:classicalgklo3}
    The space ${}^\imath X_\mu^\lambda$ is nonempty if and only if the parity condition \eqref{parity} holds.  
    In this case $\dim {}^\imath X_\mu^\lambda = 2\sum_{i \in \iI} \fkv_i$ and the map $\bC[{}^\imath X_\mu^\lambda] \rightarrow \operatorname{gr} \mc A_{\bm z = 0} $ from \eqref{eq:XtoA} defines a top-dimensional irreducible component of ${}^\imath X_\mu^\lambda$, i.e.,~the kernel of this map is its defining ideal.
\end{prop}

\begin{proof}
    If the condition \eqref{parity} does not hold, then  ${}^\imath X_\mu^\lambda$ is empty by Lemma \ref{lem:fixedcoloureddiv}.  On the other hand, if \eqref{parity} holds, then the map $\bC[{}^\imath X_\mu^\lambda] \rightarrow \operatorname{gr} \mc A_{\bm z = 0}$ from \eqref{eq:XtoA} shows that ${}^\imath X_\mu^\lambda$ is nonempty.  Moreover, this ring homomorphism becomes surjective if we localize at all $w_{i,r}$ and $w_{i,r} \pm w_{j,s}$, which shows that $\dim {}^\imath X_\mu^\lambda \gge \dim \operatorname{gr} \mc A_{\bm z=0} = 2 \sum_{i \in \iI} \fkv_i$.

    Next, we prove the opposite inequality $\dim {}^\imath X_\mu^\lambda \lle 2 \sum_{i \in \iI} \fkv_i$. Recall the coordinates on ${}^\imath X_\mu^\lambda$ from Proposition  \ref{prop:classicalgklo1}. Because of the relations among the $w_{i,r}$, we need only those elements from \eqref{eq:wcoords1}, of which there are $\sum_{i \in \iI}\fkv_i$. In the relation \eqref{eq:fixedgklo1} observe that either $\ox_{i,r}^- \neq 0$, in which case we can uniquely solve for $\ox_{\tau i, \tau r}^-$, or else $\ox_{i,r}^- = 0$, in which case there is no constraint on $\ox_{\tau i, \tau r}^-$. This dichotomy applies for each $y_{i,r}^-$ with $i\in \iI$ and $1\lle r \lle \fkv_i$, and in each case only one of  $y_{i,r}^-$ or $y_{\tau i, \tau r}^-$ is needed as a coordinate. Moreover, the relation \eqref{eq:fixedgklo2} shows that any  remaining coordinates $\ox_{i,\fkv_i+1}^-$ are finite over the $w_{i,r}$, and thus may ignored  for computing dimension. This shows that ${}^\imath X_\mu^\lambda$ is covered by finitely many locally-closed subvarieties, each of which has dimension at most $2 \sum_{i \in \iI} \fkv_i$. This gives the claimed upper bound on $\dim {}^\imath X_\mu^\lambda$.
    
    Finally, since $\operatorname{gr} \mc A_{\bm z=0}$ is a domain of the correct dimension, we conclude that it corresponds to an irreducible component of ${}^\imath X_\mu^\lambda$ of top dimension.
\end{proof}

Armed with this understanding of ${}^\imath X_\mu^\lambda$, we now return to studying the open subvariety ${}^\imath U_\mu^\lambda \subseteq \iWbar_\mu^\lambda.$

\begin{lem}
    \label{lem:iXU:flat}
    The map $X_\mu^\lambda \rightarrow U_\mu^\lambda$ is $\sigma$-equivariant. The corresponding map ${}^\imath X_\mu^\lambda \rightarrow {}^\imath U_\mu^\lambda$ on fixed point loci identifies
    $$
    {}^\imath X_\mu^\lambda = {}^\imath U_\mu^\lambda \times_{{}^\imath \bA^{(\lambda-\mu)}} {}^\imath \bA^{|\lambda-\mu|} = {}^\imath U_\mu^\lambda \times_{{}^\imath \mathring{\bA}^{(\lambda-\mu)}} {}^\imath \mathring{\bA}^{|\lambda-\mu|}.
    $$
    In particular, the map ${}^\imath X_\mu^\lambda \rightarrow {}^\imath U_\mu^\lambda$ is finite and faithfully flat.
\end{lem}
\begin{proof}
Under the action of $\sigma$  on $X_\mu^\lambda$, we have
$$
\sigma \left(\sum_r \Big( \prod_{s \neq r} \frac{z-\ow_{i,s}}{\ow_{i,r} - \ow_{i,s}}\Big) \ox_{i,r}^\pm \right)  = (-1)^{\bv_i} \sum_r \Big( \prod_{s \neq r} \frac{-z-\ow_{\tau i,\tau s}}{\ow_{\tau i, \tau r} - \ow_{\tau i, \tau s}}\Big) \ox_{\tau i,\tau r}^\mp.
$$
This matches with $\sigma( \gklob_i(z)) = (-1)^{\bv_i} \gkloc_{\tau i}(-z)$ and $\sigma(\gkloc_i(z)) = (-1)^{\bv_i} \gklob_{\tau i}(-z)$ from \eqref{eq:sigmaingklo}.  Using Theorem \ref{prop:gklobirat}(2) we conclude that $X_\mu^\lambda \rightarrow U_\mu^\lambda$ is $\sigma$-equivariant.

Recall from Theorem \ref{prop:gklobirat} that $X_\mu^\lambda = U_\mu^\lambda \times_{\bA^{(\lambda-\mu)}} \bA^{|\lambda-\mu|}$.  The maps defining this fiber product are all $\sigma$-equivariant, so by Lemma \ref{lem:fixedfp} we conclude that the fixed point scheme ${}^\imath X_\mu^\lambda$ is the fiber product ${}^\imath U_\mu^\lambda \times_{{}^\imath \bA^{(\lambda-\mu)}} {}^\imath \bA^{|\lambda-\mu|}$ of corresponding fixed point loci.  Since the map ${}^\imath \bA^{|\lambda-\mu|} \rightarrow {}^\imath \bA^{(\lambda-\mu)}$ is finite and faithfully-flat by Lemma \ref{lem:fixedcoloureddiv}, its base change ${}^\imath X_\mu^\lambda \rightarrow {}^\imath U_\mu^\lambda$ is also finite and faithfully flat.
\end{proof}

\begin{proof}[Proof of Theorem \ref{thm:ctgklo}]
    For Part (1),  first note that the locus $U_\mu^\lambda$ is preserved by $\sigma$ because of Lemma \ref{lem:fixedcoloureddiv}.  Since the map ${}^\imath X_\mu^\lambda \rightarrow {}^\imath U_\mu^\lambda$ is finite and faithfully flat by Lemma \ref{lem:iXU:flat}, the rest of Part (1) follows from Proposition \ref{prop:classicalgklo3}.

    For Part (2), it is a tedious but straightforward calculation to show that the filtered degrees of the generators $B_i^{(r)}, H_i^{(r)}$ are preserved under $\Phi_\mu^{\la, \bm z=0}$. The map $\Phi_\mu^{\la, \bm z=0}$ is therefore filtered by the same argument as in \cite[\S B(vii)]{BFN19}.   Next we look at the composed map   
    $$
    {}^\imath X_\mu^\lambda \rightarrow {}^\imath U_\mu^\lambda \hookrightarrow \iWbar_\mu^\lambda \hookrightarrow \iW_\mu,
    $$
    which corresponds to a Poisson algebra homomorphism $\bC[\iW_\mu] \rightarrow \bC[{}^\imath X_\mu^\lambda]$. Composed with Lemma \ref{lem:classicalgklo2}, we obtain a Poisson map $\bC[\iW_\mu] \rightarrow \operatorname{gr} \mc A_{\bm z=0}$.  Its effect on the Poisson generators $b_i^{(r)}, h_i^{(s)} \in \bC[\iW_\mu]$ is given by substituting Lemma \ref{lem:classicalgklo2} into the formulas from Remark \ref{rem:ogcgklo}.  

    We can obtain a second Poisson map $\bC[\iW_\mu]\rightarrow \operatorname{gr} \mc A_{\bm z=0}$ from the iGKLO homomorphism: the map $\Phi_\mu^\lambda$ is filtered by Part (1), and we know that $\operatorname{gr}^{F_{\mu_1}^\bullet} \Yt_\mu \cong \bC[\iW_\mu]$ by Theorem \ref{thm:iWmuPoisson}.  Inspecting leading terms in Theorem \ref{thm:GKLOquasisplit}, we see that both homomorphisms agree on the Poisson generators $b_i^{(r)}, h_i^{(s)} \in \bC[\iW_\mu]$.  Thus, the two Poisson maps are equal.

    Finally, recall that the map from Proposition \ref{prop:classicalgklo3} defines a top-dimensional irreducible component of ${}^\imath X_\mu^\lambda$, and thus a top-dimensional irreducible component $C_\mu^\lambda \subseteq {}^\imath U_\mu^\lambda$ under the finite and faithfully flat map ${}^\imath X_\mu^\lambda \rightarrow {}^\imath U_\mu^\lambda$. The map $\bC[\iW_\mu] \rightarrow \operatorname{gr} \mc A_{\bm z=0}$ therefore defines its closure $\overline{C}_\mu^\lambda \subseteq \iWbar_\mu^\lambda$, completing the proof of Part (2).
\end{proof}

\section{Affine Grassmannian 
$\mathrm{i}$slices of type AI and nilpotent Slodowy slices}
\label{sec:slodowy}

In this section, we restrict ourselves to affine Grassmannian slices for $G$ of $\mathrm{PGL}$ type, and the diagram automorphism $\tau = \operatorname{id}$, corresponding to the Satake diagram of type AI.  We identify the affine Grassmannian islices with specific nilpotent Slodowy slices of type BCD.

\subsection{Nilpotent orbits and transversal slices ABC$^+$}
\label{ssec:recollectionnilp}

Let $\fka$ be a reductive Lie algebra over $\bC$, and let $\mc N_{\fka}$ denote its nilpotent cone. Then $\mc N_{\fka}\subset\fka$ is an irreducible Poisson subvariety.  It decomposes as a union of finitely many nilpotent orbits under the action of the adjoint type group corresponding to $\fka$, which are also precisely the symplectic leaves of $\mc N_{\fka}$. We recall aspects of this theory in classical types, following \cite{CM93}.

For $\fksl_N$, nilpotent orbits are parametrized by the set $\ParN$ of partitions of $N$: each $\la \in \ParN$ corresponds to the orbit $\mathbb{O}_\la \subset \mc N_{\fksl_N}$ through any nilpotent element $e_\la \in \fksl_N$ with Jordan type $\la$.  Write $\overline{\mathbb{O}}_\la$ for the closure.  Then there are decompositions
\begin{equation}
    \mc N_{\fksl_N} = \bigsqcup_{\la \in \ParN} \mathbb{O}_\la, \quad \quad \overline{\mathbb{O}}_\la = \bigsqcup_{\substack{\la' \in \ParN, \\\la' \unlhd \la} } \mathbb{O}_{\la'}.
\end{equation}
Here $\la' \unlhd  \la$ denotes the dominance order on partitions.

To treat the classical types uniformly, we shall denote $\epsilon=+$ for orthogonal types and $\epsilon = -$ for symplectic.  Let $J_\epsilon$ be an invertible $N\times N$ matrix which is symmetric if $\epsilon=+$ and skew-symmetric if $\epsilon=-$, and define an involution $\sigma_\epsilon$ of $\fksl_N$ by 
\begin{equation}
\label{eq:tinvo}
    \sigma_\epsilon: X \mapsto - J_\epsilon^{-1} X^T J_\epsilon.
\end{equation}
Denoting by $\fksl_N^\epsilon = (\fksl_N)^{\sigma_\epsilon}$ the corresponding fixed point subalgebra, we thus have
\begin{align} \label{signosp}
\fksl_N^\epsilon = 
\begin{cases}
\mathfrak{so}_N, & \text{ if } \epsilon=+, \\ 
\mathfrak{sp}_N, & \text{ if } \epsilon=-. 
\end{cases}
\end{align}
The involution $\sigma_\epsilon$ preserves each nilpotent orbit $\mathbb{O}_\la$ of $\fksl_N$. We will consider the corresponding fixed point loci:
\begin{equation}
    \mathbb{O}_\la^\epsilon := (\mathbb{O}_\la)^{\sigma_\epsilon}.
\end{equation}

Recall that a partition is {\em orthogonal} (resp.~{\em symplectic}) if all its even (resp.~odd) parts occur with even multiplicity.  Denote by $\text{Par}_\epsilon(N)$ the set of orthogonal (resp. symplectic) partitions of $N$, for $\epsilon =+$ (resp. $-$).  
The following summarizes results of Gerstenhaber and of Hesselink, see \cite[Theorems 5.1.6,  6.2.5 \& 6.3.3]{CM93}:
\begin{prop}
\label{prop:cloc}
    The fixed point locus $\mathbb{O}_\la^\epsilon$ is nonempty if and only if $\la \in \ParNepsilon$, and in this case, $\mathbb{O}_\la^\epsilon$ is an $O_N$-orbit for $\epsilon=+$ and a $Sp_N$-orbit for $\epsilon=-$. There are decompositions
    $$
(\mc{N}_{\fksl_N})^{\sigma_\epsilon} = \mc N_{\fksl_N^\epsilon} = \bigsqcup_{\la \in \ParNepsilon} \mathbb{O}_\la^\epsilon, 
\quad \quad 
(\overline{\mathbb{O}}_\la)^{\sigma_\epsilon} = \bigsqcup_{\substack{\nu \in \ParNepsilon, \\ \nu \unlhd \la}} \mathbb{O}_{\nu}^\epsilon.
    $$
    Moreover, for any $\la \in \ParN$, the variety $(\overline{\mathbb{O}}_\la)^{\sigma_\epsilon}$ is the closure of the orbit $\mathbb{O}_{\la'}^\epsilon$ where $\la' \in \ParNepsilon$ is the unique maximal element satisfying $\la' \unlhd \la$. In particular, if $\la \in \ParNepsilon$ then $\la' = \la$.
\end{prop}

\begin{rem}
    \label{rem:ve}
    For $\la \in \operatorname{Par}_+(N)$, the $O_N$-orbit $\mathbb{O}_\la^+$ is connected except for $\la$ \emph{very even}, in which case it has two components (each of which is an $SO_N$-orbit). See \cite[Theorem 5.1.4]{CM93} for details.
\end{rem}

Fix a partition $\la \in \ParNepsilon$ where $\epsilon \in \{\emptyset, +,-\}$, and choose an element $e \in \mathbb{O}_\la^\epsilon$; it is understood that the superscript $\emptyset$ denotes type A and can be dropped.  By the Jacobson-Morozov Theorem there exists a corresponding $\fksl_2$-triple $\{e,h,f\}\subset \fksl_N^\epsilon$, which is unique up to conjugation by a result of Kostant \cite[\S 3]{CM93}.  We consider the corresponding Slodowy slice in $\fksl_N^\epsilon$, defined by:
\begin{equation}
    \mathcal{S}_\la^\epsilon = e + \operatorname{Ker}( \operatorname{ad} f).
\end{equation}
Then $\mathcal{S}_\la^\epsilon$ is a transversal slice to the nilpotent orbit $\mathbb{O}_\la^\epsilon$ at $e \in \fksl_N^\epsilon$, and more generally it intersects all nilpotents in $\fksl_N^\epsilon$ transversely \cite[\S 2.2]{GG02}. It also follows that for any inclusion $\mathbb{O}^\epsilon_{\mu} \subseteq \overline{\mathbb{O}_{\la}^\epsilon}$, there are nonempty intersections
\begin{equation}
    \mathbb{O}^\epsilon_{\la} \cap \mc S_{\mu}^\epsilon \neq \emptyset, 
    \quad \quad \overline{\mathbb{O}^\epsilon_{\la}} \cap \mc S_{\mu}^\epsilon \neq \emptyset.
\end{equation}
All of these spaces inherit Poisson structures from $\fksl_N^\epsilon$ by \cite[\S 3]{GG02}.  Note that $\mathcal{S}_\la^\epsilon$ is uniquely determined from $\mathbb{O}^\epsilon_\la$, up to conjugation.

\subsection{Results of Lusztig and Mirkovi\'c-Vybornov}
Throughout this section we will consider affine Grassmannian slices $\overline{\cW}_\mu^\la$ where $G$ is of $\mathrm{PGL}$ type, and their relation to nilpotent orbits in $\fksl_N$.

First consider $G = \mathrm{PGL}_N$. Lusztig \cite[\S 2]{Lu81} constructed an isomorphism 
\begin{equation}
    \label{eq:nilp1}
    \mathcal{N}_{\fksl_N} \cong \overline{\cW}_0^{(N)},
\end{equation}
given explicitly by $X \in \mathcal{N}_{\fksl_N} \mapsto I + z^{-1}X \in G_1[\![z^{-1}]\!] = \cW_0$; here the partition $(N)$ in the superscript stands for the coweight $N \varpi_1^\vee$. Moreover, for each partition $\la =(\la_1,\ldots, \la_N) \in \ParN$ this map restricts to isomorphisms of strata $\mathbb{O}_\la \cong \cW_0^\la$ and their closures $\overline{\mathbb{O}}_\la \cong \overline{\cW}_0^\lambda$, where $\la$ is identified with the coweight $\lambda = \la_1 \varepsilon_1^\vee +\ldots+\la_N\varepsilon_N^\vee$.  

More generally, let $G=PGL_n$. Let $\la=(\la_1,\ldots,\la_{n-1})$ and $\mu=(\mu_1,\ldots,\mu_n)$ be partitions of $N$ such that $\la\unrhd \mu$. Denote by $\varepsilon_i^\vee$, for $1\le i\le n$, the standard orthonormal basis for the coweight lattice of $GL_n$. Their images span the coweight lattice of $PGL_n$, and  are related to the fundamental coweights $\varpi_i^\vee$ by $\varpi_i^\vee=\varepsilon^\vee_1+\varepsilon^\vee_2+\ldots+\varepsilon^\vee_i$ for $1 \leq i < n$. (Note that $\varpi_n^\vee = \varepsilon_1^\vee + \ldots + \varepsilon_n^\vee = 0$ when viewed as a coweight of $PGL_n$.) Then $\la$ and $\mu$ give rise to the following dominant coweights of $PGL_n$, which we denote again by $\la$ and $\mu$ by a slight abuse of notation: 
\begin{align}  \label{weight=par}
\la =\sum_{i=1}^{n-1}\la_i\varepsilon_i^\vee, 
\qquad
\mu=\sum_{i=1}^n\mu_i\varepsilon_i^\vee. 
\end{align}

Label the nodes of the  Dynkin diagram of $G$ by  $\I =\{1,\ldots,n-1\}$ in the standard way.  Recall our notation $\bw_i = \langle \lambda, \alpha_i\rangle$ for $i\in \I$. Then we have $N = \sum_{i=1}^{n-1} i \bw_i$. One checks that $\la$ can be expressed in exponential notation followed by a transpose $t$ by letting
\begin{equation}
    \label{eq:MViso1}
    \la = \big(1^{\bw_1} 2^{\bw_2} \cdots (n-1)^{\bw_{n-1}}\big)^t.
\end{equation}
For example, for $n=4$, $\bw_1=4, \bw_2=3,$ $\bw_3=2$, the Young diagram of $\la$ is
\begin{align} \label{YDiagram}
 &   \begin{ytableau}
    *(brown) & *(brown) & *(lightgray)  & *(lightgray) & *(lightgray)  & &&& \\
    *(brown) & *(brown) & *(lightgray)  & *(lightgray) & *(lightgray)  & \none &\none &\none & \none \\
    *(brown) & *(brown) & \none &\none &\none & \none  &\none &\none & \none 
\end{ytableau}   
\end{align}
On the other hand, $\mu = (\mu_1, \cdots, \mu_n) \in \ParN$ satisfies
\begin{equation}
    \label{eq:MViso2}
    \langle \mu, \alpha_{i}\rangle = \mu_{i} - \mu_{i+1}, \quad \text{ for } 1\lle i\lle n-1.
\end{equation}
    
\begin{prop} {\rm (Mirkovi\'c-Vybornov \cite[Proposition 4.3.4]{MV22})}
\label{MViso}
With notation as above, there is an isomorphism
$$
\mc N_{\fksl_N} \cap \mc S_{\mu} \cong \overline{\cW}_\mu^{(N)}.
$$
It restricts to isomorphisms of strata and their closures:
$$
    {\mathbb{O}}_{\la} \cap \mathcal{S}_{\mu} \cong {\cW}_\mu^\lambda, \quad \quad \overline{\mathbb{O}}_{\la} \cap \mathcal{S}_{\mu} \cong \overline{\cW}_\mu^\lambda.
$$
\end{prop}
Strictly speaking, we will use an alternate formulation of this isomorphism  by Cautis-Kamnitzer \cite[\S 3.3]{CK18}.  Note that the above isomorphisms are Poisson  by \cite[Theorem A]{WWY20}.    
\begin{rem}
    Actually, the above works use transverse slices which are generally different from Slodowy's $\mathcal{S}_\mu$. This is not crucial however: all reasonable slices (``MV slices'') are Poisson isomorphic by \cite[Theorem 5.5]{WWY20}, and these isomorphisms are easily seen to restrict to isomorphisms of intersections with nilpotent orbits and their closures.  Similarly the precise choice of $\fksl_2$-triple $\{e,h,f\} \in \fksl_N^\epsilon$ used to define $\mc S_{\mu}^\epsilon$ is irrelevant, up to an isomorphism induced by conjugation.
\end{rem}

\subsection{$\mathrm{i}$Slices and $O_N$-orbit closures}

 Let $G=\mathrm{PGL}_N$, and on the loop group $G(\!(z^{-1})\!)$ we have $\sigma( g(z)) = g(-z)^T $ as in Example \ref{eg:transpose}. Transporting $\sigma$ via the isomorphism $\mathcal{N}_{\fksl_N} \cong \overline{\cW}_0^{(N)}$ from \eqref{eq:nilp1}, we obtain the involution $\sigma:X \mapsto -X^T$ on $\mathcal{N}_{\fksl_N}$. Its fixed point locus is $\mc N_{\mathfrak{so}_N}$, as follows from \cite[Proposition 1.1.3]{CM93}. By restriction, we thus obtain  an isomorphism of  fixed point loci:
$$
    \mathcal{N}_{\mathfrak{so}_N} \cong \iWbar_0^{(N)}.
$$

Recall that the isomorphism $\mc N_{\fksl_N} \cong \overline{\cW}_0^{(N)}$ restricts to isomorphisms of strata and their closures. Since our involutions preserve strata, we obtain isomorphisms of fixed point loci: 
\begin{equation}
    \label{eq:MViso3}
    {\mathbb{O}}_\la^+ \cong  \iW_0^\lambda, \quad \quad 
    (\overline{\mathbb{O}}_\la)^\sigma \cong  \iWbar_0^\lambda.
\end{equation}
By Proposition \ref{prop:cloc} the locus $\mathbb{O}_\la^+$ is nonempty if and only if the partition $\la$ is orthogonal.  Recall also that $(\overline{\mathbb{O}}_\la)^\sigma$ is the closure of an orbit $\mathbb{O}_{\la'}^+$ (with $\la' = \la$ if $\la$ is orthogonal, but $\la'\lhd\la$ otherwise), and thus $\iWbar_0^\la$ is the closure of a corresponding stratum $\iW_0^{\la'}$. We also deduce that the variety $\iWbar_0^\la$ is reducible for some $\la$ (see Remark \ref{rem:ve}), or even non-normal since this is true of classical nilpotent orbits by \cite[Theorem 1]{KP82}.

\subsection{$\mathrm{i}$Slices and nilpotent Slodowy slices}

Next, we move on to the setting of Proposition \ref{MViso} with $G=\mathrm{PGL}_n$.  First, we focus on the special case where $\la = (N)$.  In  this case, from Proposition~\ref{MViso} we have
\begin{align} \label{eq:MViso}
\mathcal{N}_{\fksl_N} \cap \mathcal{S}_{\mu} \cong \overline{\cW}_\mu^{(N)}.
\end{align}
To consider the fixed point locus $\iWbar_\mu^{(N)}$, recall from Lemma \ref{lem: gradings and shifts iWmu} that we need to assume that $\mu$ is even (and spherical, but this is automatic since $\tau=\id$). 
Inspecting Equation \eqref{eq:MViso2}, we observe that the coweight $\mu$ being even  corresponds to the partition $\mu$ having all parts of the same parity.  Following \cite[\S 4.1]{T23}, we let $\epsilon = +$ if all parts of $\mu$ are odd, and $\epsilon = -$ if all parts of $\mu$ are even.  Observe that $\mu$ is orthogonal when $\epsilon=+$, and symplectic when $\epsilon=-$, just as in \S \ref{ssec:recollectionnilp}.

Define the following subsets of $\ParN$:
\begin{align} \label{eq:3Par}
    \ParNepsilon^\diamond &: =\{\la \in \ParNepsilon \mid  \text{all parts are odd (resp. even) for $\epsilon=+$ (resp. $\epsilon=-$)} \},
    \\
    \ParN_{\lle n} &:=\{\la \in \ParN \mid \ell(\la)\lle n \},
    \\
    \ParNepsilon^\diamond_{\lle n} &:= \ParNepsilon^\diamond \cap \ParN_{\lle n}. 
\end{align}
Given $\mu \in \ParNepsilon$, we choose $J_\epsilon$ in \eqref{eq:tinvo} for the involution $\sigma_\epsilon$ of $\fksl_N$ exactly as in \cite[\S 4.2]{T23} (the assumption that $\mu \in \ParNepsilon^\diamond$ can be relaxed for now). This allows Topley \cite[\S 4.2]{T23} to choose $e \in \fksl_N^\epsilon$ of Jordon form $\mu$ and construct an $\fksl_2$-triple $\{e,h,f\}$ which is fixed pointwise by $\sigma_\epsilon$, i.e., lies inside $\fksl_N^\epsilon$.  We shall use this $\fksl_2$-triple throughout the rest of this section.
Then we see that the involution $\sigma_\epsilon$ on $\fksl_N$ in \eqref{eq:tinvo} restricts to an involution on the Slodowy slice $ \mathcal{S}_{\mu}$, and the corresponding fixed point locus is the Slodowy slice $ \mathcal{S}_{\mu}^\epsilon$ in the classical Lie algebra $\fksl_N^\epsilon$: \begin{align} \label{fixSlodowy}
(\mathcal{S}_{\mu})^{\sigma_\epsilon} = \mathcal{S}_{\mu}^\epsilon.
\end{align}
While we did not find this formulation in literature, an isomorphism (instead of equality) of this form seems implicit in \cite[(4.24)]{T23} by passing through the Dirac reduction. 

\begin{prop} \label{prop:onlypi2}
Let $\mu \in \ParNepsilon^\diamond_{\lle n}$ be identified with an even dominant coweight for $G = \mathrm{PGL}_n$. Then the isomorphism $\overline{\cW}_\mu^{(N)} \cong \mathcal{N}_{\fksl_N} \cap \mathcal{S}_{\mu}$ in \eqref{eq:MViso} is $\sigma_\epsilon$-equivariant, giving rise to a Poisson isomorphism of fixed-point loci:
    \begin{align} \label{iMV:N}
    \iWbar_\mu^{(N)} \cong \mathcal{N}_{\fksl_N^\epsilon} \cap \mathcal{S}_{\mu}^\epsilon. 
\end{align}
\end{prop}

\begin{proof}
    In brief, by \cite[Theorem 4.3(d)]{WWY20} the isomorphism \eqref{eq:MViso}  can be understood as the classical limit of the Brundan-Kleshchev isomorphism \cite[Theorem 10.1]{BK06} between a finite W-algebra and a quotient of a shifted Yangian.  Topley has proven \cite[Proposition 4.7]{T23} that the classical Brundan-Kleshchev isomorphism is $\sigma_\epsilon$-equivariant. Putting these two pieces together proves our claim.
    
    More precisely,  following the conventions of \cite[Theorem 4.3]{WWY20} the Brundan-Kleshchev isomorphism is an isomorphism of filtered algebras:
    \begin{equation}
        \label{eq:BKiso}
        Y_\mu^{(N)} \xrightarrow{\sim} W(\mathfrak{gl}_N, \mu).
    \end{equation}
    The left side is a truncated shifted Yangian for $\fksl_n$, while the right side is a finite W-algebra for $\mathfrak{gl}_N$.  These algebras have non-trivial centers, and we should pass to central quotients as in \cite[Theorem 4.9]{WWY20}. (The precise choice of central quotient is not important, since all choices yield the same associated graded algebras.)
    Taking a central quotient on the left side, we obtain an algebra $Y_\mu^{(N)}(\mathbf{R})$ whose associated graded algebra is the coordinate ring $\bC[\overline{\cW}_\mu^{(N)}]$ by \cite[Theorem~ 2.5]{WWY20}. On the right side, we obtain a central quotient $W(\mathfrak{gl}_N, \mu)_{\mathbf{R}}$ defined as in \cite[\S 3.3.2]{WWY20}. 
    
    Now consider  the classical finite W-algebra $S(\mathfrak{gl}_N, \mu) = \operatorname{gr} W(\mathfrak{gl}_N, \mu)$.  By \cite[Theorem~ 4.1]{GG02}, there is an isomorphism $S(\mathfrak{gl}_N, \mu) \cong \bC[\mathcal{S}'_{\mu}]$, where $\mathcal{S}'_{\mu}$ denotes the Slodowy slice defined in $\mathfrak{gl}_N$.  Consider also the quotient algebra $S(\mathfrak{gl}_N, \mu)_0 = S(\mathfrak{gl}_N, \mu) / J$, where $J$ is the ideal generated by the coefficients of the characteristic polynomial.  Then we have $S(\mathfrak{gl}_N, \mu)_0 = \operatorname{gr} W(\mathfrak{gl}_N, \mu)_{\mathbf{R}}$, as follows from \cite[(4.5)]{WWY20} combined with \cite[Lemma 3.7]{BK08}. There is an isomorphism
    \begin{equation*}
    S(\mathfrak{gl}_N, \mu)_0 \cong \bC[ \mathcal{N}_{\fksl_N} \cap \mathcal{S}_{\mu}]
    \end{equation*}
    as in \cite[Remark 3.18]{WWY20}. This is related to the well-known fact that if $K \subset \bC[\mathfrak{gl}_N]$ denotes the ideal generated by the coefficients of the characteristic polynomial, then $\bC[\mathcal{N}_{\fksl_N}] = \bC[\mathfrak{gl}_N]/K$.
    
    Altogether, the isomorphism \eqref{eq:MViso} is induced by a composition of algebra isomorphisms:
    \begin{equation}
    \label{eq:MVisoviaBK}
    \xymatrix{
    \bC[\overline{\cW}^{(N)}_\mu] \ar[r]^-{\sim} & S(\mathfrak{gl}_N, {\mu})_0 \ar[r]^-{\sim} & \bC[\mathcal{N}_{\fksl_N} \cap \mathcal{S}_{\mu}].
    }
    \end{equation}

    Finally, \eqref{eq:tinvo} naturally induces an involution on $S(\mathfrak{gl}_N, \mu)$, see \cite[\S 2.3]{T23}.  By  \cite[Proposition~4.7]{T23}, the classical limit of \eqref{eq:BKiso} is $\sigma_\epsilon$-equivariant.  Note that Topley works with classical shifted Yangians for $\mathfrak{gl}_n$, whereas we work with $\mathfrak{sl}_n$.  These algebras are easily compared, as in (the classical limit) of \cite[\S 4.2]{WWY20}.  Tracing through conventions, one can see that  \cite[(3.58)]{T23} matches with Lemma \ref{lem: gradings and shifts iWmu}(4).  As a result, we see that the first map in \eqref{eq:MVisoviaBK} is $\sigma_\epsilon$-equivariant.  By construction, the second isomorphism in \eqref{eq:MVisoviaBK} is also $\sigma_\epsilon$-equivariant.  This completes the proof.
\end{proof}

We now return to the case of general dominant coweights  $\la \gge \mu$ for $\mathrm{PGL}_n$, which correspond to partitions  $\la\in \ParN_{\lle n-1}$ and $\mu\in \ParN_{\lle n}$ such that $\la \unrhd \mu$. Since we wish to consider fixed points we must assume that the coweight $\mu$ is even, which corresponds to $\mu \in \ParNepsilon^\diamond_{\lle n}$. Note  however that there is no need to impose any additional condition on $\la$.  

The following is the main result of this section.

\begin{thm}
\label{thm:loci=sliceBCD}
Let $\tau =\id$. Let $\la \in \ParN_{\lle n-1}$ and $\mu \in\ParNepsilon^\diamond_{\lle n}$ be such that $\la \unrhd \mu$; they are identified with dominant coweights for $\mathrm{PGL}_n$ such that $\la\gge \mu$. Then the isomorphism $\iWbar_\mu^{(N)} \cong \mc N^\epsilon_{\fksl_N} \cap \mathcal{S}^\epsilon_{\mu}$ from \eqref{iMV:N} restricts to Poisson isomorphisms
    \[
    \iW_\mu^{\la}\cong {\mathbb{O}^{\epsilon}_{\la}} \cap \mathcal{S}_{\mu}^\epsilon, \quad \quad \iWbar_\mu^{\la}\cong (\overline{\mathbb{O}}_{\la})^{\sigma_\epsilon} \cap \mathcal{S}_{\mu}^\epsilon.
    \]
    Moreover,
    \begin{enumerate}
        \item $\iW_\mu^\la$ is nonempty if and only if $\la \in \ParNepsilon$.  
        \item The variety $\iWbar_\mu^\la$ is the closure of its stratum $\iW_\mu^{\la'} \cong \mathbb{O}^\epsilon_{\la'} \cap \mc S^\epsilon_{\mu}$, where $\la'$ is the unique maximal element $\la' \in \ParNepsilon$ satisfying $\la'\unlhd \la$.  
    \end{enumerate}
\end{thm}

\begin{proof}
    By Proposition \ref{prop:onlypi2}, the isomorphism $\overline{\cW}_\mu^{(N)} \cong \mc N_{\fksl_N} \cap \mathbb{O}_{\mu}$ is $\sigma_\epsilon$-equivariant.  We also know from Proposition \ref{MViso} that this isomorphism restricts to isomorphisms of strata and their closures.  Since these strata and closures are preserved by $\sigma_\epsilon$, by restriction we obtain isomorphisms of their respective fixed point loci. This proves  the first claim, and the remaining claims then follow from Proposition \ref{prop:cloc}.
\end{proof}

\begin{eg}  \label{ex:rank1:islice}
    We return once again to $G= \operatorname{PGL}_2$ and $\tau = \id$, as in Examples \ref{eg:SL2slices}  and  \ref{eg:SL2slices2}.  Let $\la \gge \mu$ be dominant coweights for $G$, with $\mu$ even (or equivalently, with $\la$ even).  We extract partitions $\la = (\bw)$ and $\mu = (\bw-\bv, \bv)$.  We have $\epsilon = +$ for $\bv$ odd, and $\epsilon=-$ for $\bv $ even, and an isomorphism
    $\iWbar_\mu^\la \cong \mc N_{\fksl_N^\epsilon} \cap \mc S_{\mu}^\epsilon,$ where the right-hand side is identified with the space of matrices from Example \ref{eg:SL2slices2}.
\end{eg}

\section{$\mathrm{i}$Coulomb branches and affine Grassmannian $\mathrm{i}$slices}
  \label{sec:Coulomb}

In this section, starting with quiver gauge theory data under some mild $\tau$-symmetry/parity conditions suggested in earlier sections, we produce a new symplectic representation with new gauge group and flavor symmetry group. We speculate that the resulting iCoulomb branch (typically not of cotangent type) provides a normalization of the affine Grassmannian islices constructed earlier.

\subsection{Quiver datum}

Let $Q=(\I,\Omega)$ be an ADE quiver, and an arrow $h$ in $\Omega$ sends a vertex $h'$ in $\I$ to another vertex $h''$ in $\I$, that is, we denote $h' \stackrel{h}{\rightarrow} h''$. We have 

$\blacktriangleright$ the double quiver $\Qbar =(\I,\Omega\cup \overline{\Omega})$, with a reversed arrow $\overline{h}$ for each $h\in \Omega$;

$\blacktriangleright$ the framed quiver $Q^f =(\I \cup \I', \Omega \cup \Omega_{\I})$, where $\I' =\{i'|i\in \I\}$ is an isomorphic copy of $\I$ and $\Omega_\I$ consists of arrows from $i'$ to $i$, one for each $i\in\I$;

$\blacktriangleright$ the framed double quiver $\Qfbar =(\I \cup \I', \Omega \cup \overline{\Omega} \cup \Omega_{\I}\cup \overline\Omega_{\I})$. 

Let $(V_i; W_i)_{i\in \I}$ be a representation of the framed double quiver $\Qfbar$, with dimension vectors $\bv=(\bv_i)_{i\in \I}$ and $\bw=(\bw_i)_{i\in \I}$; it is understood that $W_i$ is attached to the vertex $i'$, for $i\in\I$. We have a symplectic vector space:
\begin{align}  \label{E_VW}
    E_{V,W} 
    = \bigoplus_{h \in \Omega \cup \overline{\Omega}} \Hom (V_{h'}, V_{h''}) \bigoplus \bigoplus_{i\in \I} \big(\Hom (W_i, V_i)
\oplus \Hom (V_i, W_i) \big), 
\end{align} 
which can be identified with the cotangent space $T^* (\oplus_{h \in \Omega} \Hom (V_{h'}, V_{h''}) \bigoplus \oplus_{i\in \I} \Hom (W_i, V_i) )$. 
Denote the gauge group by $GL_{\bv} =GL(V)=\prod_{i\in \I} GL(\bv_i)$ and the flavor symmetry group by $GL_{\bw} =GL(W)=\prod_{i\in \I} GL(\bw_i)$. 

Associated to such data, one attaches a Nakajima quiver variety $\mc M_H(\bv, \bw)$  as well a Coulomb branch $\mc M_C(\bv, \bw)$ (of cotangent type) corresponding to type A quiver gauge theories \cite{BFN19}.

\subsection{iCoulomb branches}

Following earlier sections, we let $\tau$ be a bijection on $\I$ such that $\tau^2=\id$. This induces a bijection $\ddot{\tau}$ on $\Omega \cup \overline{\Omega}$ such that, for $h\in \Omega \cup \overline{\Omega}$,  
\begin{enumerate}
\item 
    $(\ddot{\tau} h)' ={\tau}(h'')$ and $(\ddot{\tau} h)'' ={\tau} (h')$;
\item
$\ddot{\tau} h =\overline{h}$ whenever both ends of $h$ are fixed by $\tau$, i.e., $\tau (h')=h', \tau(h'')=h''$;
\item 
$\ddot{\tau} h=h$ whenever $\tau \big(h'\big)= h''$.
\end{enumerate}
Note that $\ddot{\tau}$ is always a non-identity involution, and this makes the double quiver equipped with $\ddot{\tau}$, denoted by 
$(\Qbar,\tau,\ddot{\tau})$ or $(\Qbar,\ddot{\tau})$,  a symmetric quiver \`a la \cite{DW02}. Examples of symmetric quivers  with two neighboring vertices fixed by $\tau$, see Case (2) above (e.g., $\tau=\id$), were not explicitly considered {\em loc. cit.}, as they do not occur if one considers ADE quivers $Q$ (instead of $\Qbar$); but they occur frequently and form an important family in our setting (see Example \ref{ex:iquivers}) as we consider symmetric quivers on the double quivers $\Qbar$. 

Applying the above construction to the double quiver $\Qfbar$, we obtain a symmetric quiver $(\Qfbar, \ddot{\tau})$.




Recall $\I_0=\{i\in \I\mid \tau i=i\}$, and the partition $\I =\I_1 \cup \I_0 \cup \I_{-1}$ so that  $\I_1$ (and resp. $\I_{-1}$) a set of representatives of $\tau$-orbits in $\I$ of length $2$. We shall choose $\I_1$ such that the underlying Dynkin subdiagram on $\I_1$ is connected. Recall $\iI =\I_1 \cup \I_0$. 
Introduce the following distinguished subsets of the arrow set $\Omega$:
\begin{align} \label{Omega3}
\begin{split}
\Omega_0 &=\{h\in \Omega\mid h',h''\in \I_0\} \\
\Omega_{1\texttt{-}0} &=\{h\in\Omega\mid h'\in \I_1, h''\in \I_0\}, \\
\Omega_{\texttt{qs}} &=\{h \in \Omega \mid \tau \big(h'\big)= h''\}. 
\end{split}
\end{align}
Denote 
\begin{align}
    \label{I1:qs} \I_1^{\texttt{qs}} =\{i\in \I_1 \mid i, \tau i \text{ are connected by an edge  }h\in \Omega\}. 
\end{align}
There is a bijection 
$\I_1^{\texttt{qs}} \rightarrow\Omega_{\texttt{qs}}$, sending $i$ to the edge $h$ connecting $i$ and $\tau i$.

Throughout this section, we make the following $\tau$-symmetry and ``not-$2$-odd" parity assumption:
\begin{align} \label{symmetricVW:I1}
\begin{cases}
\dim V_i =\dim V_{\tau i}, \quad \dim W_i =\dim W_{\tau i}, &
    \qquad \text{ for } i \in \I_1\cup \I_{-1};
    \\
\dim V_i \text{ or } \dim W_i  \text{ must be even}, &
    \qquad \text{ for } i \in \I_0;
    \\
\dim V_{h'} \text{ or } \dim V_{h''} \text{ must be even}, & \qquad \text{ for } h\in \Omega_0.
\end{cases}
\end{align}
The third assumption here is a rephrasing of the assumption \eqref{parity}, while the second one comes from Remark \ref{rem:parity}.

Denote by $\Qfbar_1$ the full subquiver of $\Qfbar$ with vertex set $\I_1 \cup \I_1'$, where $\I_1'$ is the subset of $\I'$ which is in natural bijection with $\I_1$. Let 
\begin{align} \label{VW:A}
(V^{\texttt{A}},W^{\texttt{A}}) :=(V_i;W_i)_{i\in \I_1}
\end{align}
be the representation of $\Qfbar_1$, and $E_{V^{\texttt{A}},W^{\texttt{A}}}$ be the corresponding symplectic vector space; cf. \eqref{E_VW}. Note that the gauge group for the subquiver $\Qfbar_1$ is type A. 

We fix a bipartite partition of $\I_0$:
\begin{align} \label{bipartite}
\I_0 =\I_0^\oplus \cup \I_0^\ominus
\end{align}
such that each arrow in $Q$ connects a vertex in $\I_0^\oplus$ and a vertex in $\I_0^\ominus$. (There are exactly two bipartite partitions for $\I$ of ADE type and rank $\gge 2$.) 
We shall denote 
\begin{equation} \label{parityI}
\epsilon_i=\begin{cases}
     +, & \text{ if } i \in \I_0^\oplus
    \\
    -, & \text{ if } i \in \I_0^\ominus.
\end{cases}
\end{equation}

For $U=V$ or $U=W$, $U_i^+$ denotes an orthogonal vector space (i.e., equipped with a non-degenerate symmetric bilinear form) of the same dimension as $U_i$, and $U_i^-$ denotes a symplectic vector space of dimension  $2 \lfloor \frac12 \dim U_i \rfloor$.
For $i\in \I_0$, we define 
\begin{align} \label{VWimath}
V_i^\imath := V_i^{\epsilon_i} =
    \begin{cases}
        V_i^+, & \text{ for } i\in \I_0^\oplus
        \\
        V_i^-  ,& \text{ for } i\in \I_0^\ominus, 
    \end{cases}
    \qquad\qquad
W_i^\imath := W_i^{\overline{\epsilon}_i} =
    \begin{cases}
        W_i^-, & \text{ for } i\in \I_0^\oplus
        \\
        W_i^+  ,& \text{ for } i\in \I_0^\ominus.
    \end{cases}
\end{align}  
We also define $G^{\epsilon_i}(V_i^\imath)$ to be the $\epsilon_i$-isometry group of $V^\imath$, which means special orthogonal for $\epsilon_i=+$ and symplectic for $\epsilon_i=-$. 
Similarly, we define $G^{\overline{\epsilon}_i}(W_i^\imath)$ to be the $\overline{\epsilon}_i$-isometry group of $W^\imath$.

Define the new gauge group and flavor symmetry group
\begin{align} \label{iGroups}
\begin{split}
    G^\imath(V^\imath) &:= \prod_{i\in\I_1} GL(V_i) \times \prod_{i\in\I_0} G^{\epsilon_i}(V_i^\imath),
    \\
    G^\imath(W^\imath) &:=   \prod_{i\in\I_1} GL(W_i) \times \prod_{i\in\I_0} G^{\overline{\epsilon}_i}(W_i^\imath).
    \end{split}
\end{align}

Given a quiver representation $(V,W)$ of $\Qfbar$ subject to the $\tau$-symmetry and parity condition \eqref{symmetricVW:I1}, with the above preparation we can now construct a new vector space 
\begin{align} \label{iEVW}
    \iEVW  =&\, E_{V^{\texttt{A}},W^{\texttt{A}}}
    \bigoplus 
    \big(
    \oplus_{h \in \Omega_{1\text{-}0}} \Hom (V_{h'}, V_{h''}^\imath) 
    \oplus
    \oplus_{h \in \overline{\Omega}_{1\text{-}0}} \Hom (V_{h'}^\imath, V_{h''}) \big)
    \notag \\
    &    \bigoplus
    \big( \oplus_{h \in \Omega_0} \Hom (V_{h'}^\imath, V_{h''}^\imath) 
   \oplus 
    \oplus_{i\in \I_0} \Hom (W_i^\imath, V_i^\imath) \big)
    \\
    & \bigoplus_{i \in \I_1^{\texttt{qs}}} (\wedge^2V_i\oplus \wedge^2 V_i^*).
    \notag 
\end{align}

\begin{rem} 
\label{rem:SSX}
If $\I_1^{\texttt{qs}} \neq \emptyset$, then $\I_1^{\texttt{qs}}$ is a singleton and this only occurs in the Satake diagram $(\I,\tau)$ of type AIII$_{2r}$; see Table \ref{tab:Satakediag}. This case corresponds to the last rank one framed iquiver in Table~ \ref{tab:framed osp:rk1} with ${\wedge^2V_i\oplus \wedge^2 V_i^*}$ attached. We learned how to include the last summand in \eqref{iEVW} when $\I_1^{\texttt{qs}} \neq \emptyset$ from \cite{SSX25}.
\end{rem}

Given (natural) representations $M_a$ of groups $G_a$, for $a=1,2$, $G_1\times G_2$ acts on $\Hom (M_1,M_2)$ naturally by $\big((g_1,g_2).f \big)(m_1) =g_2. f(g_1^{-1}m_1)$, for $g_a\in G_a$ and $m_1\in M_1$. All the smaller components of the first three big summands of $\iEVW$ in \eqref{iEVW} are of the form $\Hom (M_1,M_2)$ and so the (suitable components of) groups $G^\imath(V^\imath), G^\imath(W^\imath)$ act on them naturally. Only $GL(V_i)$ for $i\in \I_1^{\texttt{qs}}$ acts nontrivially on the $i$th component $\wedge^2V_i\oplus \wedge^2 V_i^*$ of the fourth summand in \eqref{iEVW}.
In this way we have defined the action of the groups $G^\imath(V^\imath)$ and $G^\imath(W^\imath)$ on $\iEVW$.

\begin{lem}
\label{lem:iEVW}
   The vector space $\iEVW$ in \eqref{iEVW} is naturally symplectic and endowed with actions of $G^\imath(V^\imath)$ and $G^\imath(W^\imath)$. 
\end{lem}

\begin{proof}
Clearly the first and fourth summands $E_{V^{\texttt{A}},W^{\texttt{A}}}$ (see \eqref{VW:A})
and $\wedge^2V_i\oplus \wedge^2 V_i^* \cong T^* (\wedge^2V_i)$ of the vector space $\iEVW$ are symplectic. The second summand $\big(
    \oplus_{h \in \Omega_{1\text{-}0}} \Hom (V_{h'}, V_{h''}^\imath) 
    \oplus
    \oplus_{h \in \overline{\Omega}_{1\text{-}0}} \Hom (V_{h'}^\imath, V_{h''}) \big)$ is isomorphic to $T^*\big(\oplus_{h \in \Omega_{1\text{-}0}} \Hom (V_{h'}, V_{h''}^\imath)\big)$, and hence symplectic too.

By the $\imath$-fication rule, exactly one of the two vector spaces $W_i^\imath$ and $V_i^\imath$ in \eqref{VWimath}, for $i\in\I_0$, is symplectic while the other one is orthogonal. Hence $\Hom (W_i^\imath, V_i^\imath)$ is naturally symplectic. 
By the bipartite partition requirement \eqref{bipartite}, exactly one of the two vector spaces $V_{h'}^\imath$ and $V_{h''}^\imath$, for $h \in \Omega_0$, is symplectic while the other one is orthogonal. Hence, $\Hom (V_{h'}^\imath, V_{h''}^\imath)$ is symplectic as well. Therefore, the third summand
$\big( \oplus_{h \in \Omega_0} \Hom (V_{h'}^\imath, V_{h''}^\imath) 
   \oplus 
\oplus_{i\in \I_0} \Hom (W_i^\imath, V_i^\imath) \big)$ of $\iEVW$ is naturally symplectic. This proves that $\iEVW$ is symplectic. 
\end{proof}

By Lemma \ref{lem:iEVW}, we can define a Coulomb branch $\mathcal M_C^\imath(V^\imath,W^\imath)$. Note that this Coulomb branch is not of cotangent type in most cases, and thus for its definition one should follow one of the competing mathematical definitions \cite{BDFRT,TCB,BF24}. Here we gloss over verifying necessary homotopy vanishing conditions in these constructions, and we wonder if the $\tau$-symmetry and not-$2$-odd condition \eqref{symmetricVW:I1} help.
 To distinguish from $\mathcal M_C(V,W)$ arising from type A quiver gauge theories, we shall refer to $\mathcal M_C^\imath(V^\imath,W^\imath)$ as {\em iCoulomb branches}.

 \begin{rem} \label{rem:ivector:rank}
 We reinterpret \eqref{ell_theta}--\eqref{varsigma} as defining new dimension vectors $(\fkv_i)_{i\in \iI}$ and $(\fkw_i)_{i\in \iI}$ from $\bv=(\bv_i)_{i\in\I}$ and $\bw=(\bw_i)_{i\in\I}$, respectively. Then the ranks of the component groups of the gauge group $G^\imath(V^\imath)$ and the flavor symmetry group $G^\imath(W^\imath)$ (see \eqref{VWimath}--\eqref{iGroups}) are exactly given by these dimension vectors $(\fkv_i)_{i\in \iI}$ and $(\fkw_i)_{i\in \iI}$, respectively. See Remark \ref{rem:SameDimension} for more numerology!
 \end{rem}

\begin{rem}
    The discussions and the formulation of $\iEVW$ in this section are valid for general quivers $Q$ without loops under some mild assumptions; for example, they are valid for affine quivers (except the cyclic quiver $A_{2r}^{(1)}$ due to \eqref{bipartite}). We have restricted ourselves to ADE type in order to match with earlier sections and to be more specific with the subsets defined in \eqref{Omega3}.
\end{rem}

\begin{eg} \label{ex:iquivers}
The construction of $\iEVW$ is much simplified in two cases. 
    \begin{enumerate}
    \item (Split type, i.e., $\tau=\id$).  In this case, we have $\I_0=\I$, $\Omega_0=\Omega$ and thus
    \[
    \iEVW = \bigoplus_{h \in \Omega} \Hom (V_{h'}^\imath, V_{h''}^\imath) 
   \bigoplus \bigoplus_{i\in \I} \Hom (W_i^\imath, V_i^\imath).
\]
In this case, the gauge group is purely ortho-symplectic. 
\item In another extreme case with $\I_0=\emptyset$, we have
\[
\iEVW =E_{V^{\texttt{A}},W^{\texttt{A}}} \bigoplus \bigoplus_{i\in\I_1^{\texttt{qs}}} ( \wedge^2V_i\oplus \wedge^2 V_i^* ). 
\]
In and only in this case, the gauge group is purely type A and $\iEVW$ is of cotangent type; the new case with $\I_1^{\texttt{qs}} \neq \emptyset$ is studied in \cite{SSX25}.
    \end{enumerate}
\end{eg}

\subsection{$\imath$-fication on (framed) Satake double quivers}

We can visualize the construction of $\iEVW$ in \eqref{iEVW} via a new iquiver ${}^\imath\Qbar =(\iI, {}^\imath \Omega)$ and the corresponding framed iquiver ${}^\imath\Qfbar =(\iI \cup \iI', {}^\imath \Omega \cup {}^\imath \Omega_{{}^\imath\I })$, by applying $\imath$-fication rules to $\Qbar$ and $\Qfbar$, respectively. We first focus on $\imath$-fications of quivers.

Recall the genuinely quasi-split ADE Satake diagrams $(\I, \tau)$ from Table \ref{tab:Satakediag}; the split Satake diagrams (with $\tau=\id$) formally look the same as ADE Dynkin diagrams. By Satake double quivers, we simply mean double quivers enhanced with a diagram involution $\tau$. Rank one framed Satake double quivers, which are building blocks of general framed Satake double quivers, are listed in Table \ref{tab:framed Satake:rk1}. In the first diagram where no $\tau$ is visible, it is understood that $\tau =\id$. We denote nodes in the quiver $Q$ by $\bigcirc$ and nodes on the framing by $\Box$. 

\begin{table}[h]  
\begin{center}
\centering
\setlength{\unitlength}{0.6mm}		
\begin{picture}(70,35)(0,5)
\put(-38,30){$\bigcirc$}
\put(-38,6){\large $\Box$}
    
\put(-35,36){\small$i$}
\put(-35,0){\small$i$}
				
{\color{blue}
\put(-33.5,12.5){\vector(0,1){15.5}} 
\put(-35.5,28){\vector(0,-1){15.5}} }
%
    \put(19,30){$\bigcirc$}
	\put(42,30){$\bigcirc$}
	\put(19,6){\large $\Box$}
	\put(42.3,6){\large $\Box$}

	\put(21,0){\small$i$}
    \put(21,36){\small$i$}
	\put(43,0){\small$\tau i$}
    \put(43,36){\small$\tau i$}



{\color{purple}
    \qbezier(26,34)(33,39)(42,34)
    \put(26,34){\vector(-3,-1){0.5}}
    \put(42,34){\vector(3,-1){0.5}}
	\put(32,37){$\tau$}
    \qbezier(26,10)(33,15)(42,10)
    \put(26,10){\vector(-3,-1){0.5}}
    \put(42,10){\vector(3,-1){0.5}}
	\put(32,13){$\tau$}
    }
				
{\color{blue}
\put(21.7,12.5){\vector(0,1){15.5}}
\put(19.7,28){\vector(0,-1){15.5}}
\put(44.7,12.5){\vector(0,1){15.5}}
\put(42.7,28){\vector(0,-1){15.5}}
}
%
    \put(89,30){$\bigcirc$}
	\put(112,30){$\bigcirc$}
	\put(89,6){\large $\Box$}
	\put(112.3,6){\large $\Box$}

	\put(91,0){\small$i$}
    \put(91,36){\small$i$}
	\put(113,0){\small$\tau i$}
    \put(113,36){\small$\tau i$}


	\put(96,30.5){\vector(1,0){15}}
	\put(111,32.5){\vector(-1,0){15}}

{\color{purple}
    \qbezier(96,34)(103,39)(112,34)
    \put(96,34){\vector(-3,-1){0.5}}
    \put(112,34){\vector(3,-1){0.5}}
	\put(102,37){$\tau$}
    \qbezier(96,10)(103,15)(112,10)
    \put(96,10){\vector(-3,-1){0.5}}
    \put(112,10){\vector(3,-1){0.5}}
	\put(102,13){$\tau$}
    }
				
{\color{blue}
\put(91.7,12.5){\vector(0,1){15.5}}
\put(89.7,28){\vector(0,-1){15.5}}
\put(114.7,12.5){\vector(0,1){15.5}}
\put(112.7,28){\vector(0,-1){15.5}}
}
\end{picture}
\end{center}
\vspace{0.5cm}
\caption{Rank one framed Satake double quivers}
\label{tab:framed Satake:rk1}
\end{table}

Rank two Satake double quivers are listed in Table \ref{tab:Satake:rk2}. To build rank two {\em framed} Satake double quivers, we simply need to copy and paste suitable frames from Table \ref{tab:framed Satake:rk1} to Table \ref{tab:Satake:rk2}.

\begin{table}[h]  
\begin{center}
\centering
\setlength{\unitlength}{0.6mm}		
\begin{picture}(70,40)(0,1)
%
%
%
    \put(-44,30){$\bigcirc$}
    \put(-21,30){$\bigcirc$}

    \put(-42,23.9){\small$j$}
    \put(-18,23.8){\small$i$}

	\put(-37,30.5){\vector(1,0){15}}
	\put(-22,32.5){\vector(-1,0){15}}
%
%
%
    \put(43,30){$\bigcirc$}
    \put(66,30){$\bigcirc$}
    \put(89,30){$\bigcirc$}
	\put(112,30){$\bigcirc$}

    \put(45,24){\small$\tau i$}
    \put(68,24){\small$\tau j$}
    \put(91,23.9){\small$j$}
    \put(114,23.7){\small$i$}

	\put(50,30.5){\vector(1,0){15}}
	\put(96,30.5){\vector(1,0){15}}
	\put(65,32.5){\vector(-1,0){15}}
	\put(111,32.5){\vector(-1,0){15}}

{\color{purple}
    \qbezier(73,35)(80,40)(89,35)
    \put(73,35){\vector(-3,-1){0.5}}
    \put(89,35){\vector(3,-1){0.5}}
    \qbezier(50,35)(76,52)(112,35)
    \put(50,35){\vector(-3,-1){0.5}}
    \put(112,35){\vector(3,-1){0.5}}
    \put(79,39){$\tau$}
    }
%
%
%
    \put(-44,0){$\bigcirc$}
    \put(-21,0){$\bigcirc$}
	\put(2,0){$\bigcirc$}

    \put(-42,-6){\small$\tau j$}
    \put(-19,-6.3){\small$i$}
    \put(4,-6.1){\small$j$}

	\put(-37,0.5){\vector(1,0){15}}
	\put(-14,0.5){\vector(1,0){15}}
	\put(-22,2.5){\vector(-1,0){15}}
	\put(1,2.5){\vector(-1,0){15}}

{\color{purple}
    \qbezier(-37,5)(-21,12)(2,5)
    \put(-37,5){\vector(-3,-1){0.5}}
    \put(2,5){\vector(3,-1){0.5}}
    \put(-21,9){$\tau$}
    }
%
%
%
%
    \put(43,0){$\bigcirc$}
    \put(66,0){$\bigcirc$}
    \put(89,0){$\bigcirc$}
	\put(112,0){$\bigcirc$}

    \put(45,-5){\small$\tau i$}
    \put(68,-6){\small$\tau j$}
    \put(91,-6.1){\small$j$}
    \put(114,-6.3){\small$i$}

	\put(50,0.5){\vector(1,0){15}}
	\put(73,0.5){\vector(1,0){15}}
	\put(96,0.5){\vector(1,0){15}}
	\put(65,2.5){\vector(-1,0){15}}
	\put(88,2.5){\vector(-1,0){15}}
	\put(111,2.5){\vector(-1,0){15}}

{\color{purple}
    \qbezier(73,5)(80,10)(89,5)
    \put(73,5){\vector(-3,-1){0.5}}
    \put(89,5){\vector(3,-1){0.5}}
    \qbezier(50,5)(76,22)(112,5)
    \put(50,5){\vector(-3,-1){0.5}}
    \put(112,5){\vector(3,-1){0.5}}
    \put(79,9){$\tau$}
    }
\end{picture}
\end{center}
\vspace{0.5cm}
\caption{Rank two Satake double quivers}
\label{tab:Satake:rk2}
\end{table}

In  Tables \ref{tab:framed osp:rk1}--\ref{tab:osp:rk2}, we present the iquivers obtained by $\imath$-fication of Satake double quivers from Tables \ref{tab:framed Satake:rk1}--\ref{tab:Satake:rk2}, respectively. Here and below we set \begin{align*} 
\epsilon\in \{\pm \}, \quad \text{ and }\;\;
\bar{\epsilon}:=-\epsilon,
\end{align*}
and we add signs to (framed or not) nodes, e.g., $\oplus, \ominus, \boxplus,$ or $\boxminus$, in iquivers obtained from modifying a $\tau$-fixed node in Satake double quivers. Note that the signs are always opposite in neighboring connected nodes.

\begin{table}[h]  
\begin{center}
\centering
\setlength{\unitlength}{0.6mm}		
\begin{picture}(70,35)(0,5)
\put(-26.2,29.8){\small $\bar{\epsilon}$}
\put(-28,30){$\bigcirc$}
\put(-28,6){\large $\Box$}
\put(-26.2,7.2){\small ${\epsilon}$}    
\put(-25,36){\small$i$}
\put(-25,0){\small$i$}				
{\color{blue}
\linethickness{.25mm}
\put(-24.5,28){\line(0,-1){15.5}}  
}
%
%
    \put(39,30){$\bigcirc$}
	\put(39,6){\large $\Box$}

	\put(41,0){\small$i$}
    \put(41,36){\small$i$}
				
{\color{blue}
\put(43.7,12.5){\vector(0,1){15.5}}
\put(41.7,28){\vector(0,-1){15.5}}
}
%
%
\put(112,30){$\bigcirc$}
\put(112,6){\large $\Box$}

\put(115,0){\small$i$}
\put(115,36){\small$i$}
\put(112.8,30){$\purple\circlearrowleft$}

{\color{blue}
\put(116.5,12.5){\vector(0,1){15.5}}
\put(114.5,28){\vector(0,-1){15.5}}
}
\end{picture}
\end{center}
\vspace{0.5cm}
\caption{Rank one framed iquivers}
\label{tab:framed osp:rk1}
\end{table}

\begin{table}[h]  
\begin{center}
\centering
\setlength{\unitlength}{0.6mm}		
\begin{picture}(70,40)(0,1)
%
%
%
    \put(-24,30){$\bigcirc$}
    \put(-1,30){$\bigcirc$}
    \put(-22.4,29.8){\small $\bar{\epsilon}$}
    \put(1,30.3){\small ${\epsilon}$}

    \put(-22,23.9){\small$j$}
    \put(1,23.8){\small$i$}

{    \linethickness{.25mm}
	\put(-17,31.5){\line(1,0){15}}
}
%
%
%
    \put(63,30){$\bigcirc$}
    \put(86,30){$\bigcirc$}
    \put(65,23.9){\small$j$}
    \put(89,23.8){\small$i$}

	\put(70,30.5){\vector(1,0){15}}
\put(85,32.5){\vector(-1,0){15}}
%
%
%
    \put(-24,0){$\bigcirc$}
    \put(-1,0){$\bigcirc$}
    \put(-22,-6.1){\small$j$}
    \put(1,-6.2){\small$i$}
    \put(1,0.5){\small ${\epsilon}$}

	\put(-17,0.5){\vector(1,0){15}}
	\put(-2,2.5){\vector(-1,0){15}}

%
%
%
%
    \put(63,0){$\bigcirc$}
    \put(86,0){$\bigcirc$}

    \put(65,-6.1){\small$j$}
    \put(89,-6.2){\small$i$}
\put(86.8,0){$\purple\circlearrowleft$}
\put(86.8,0){$\purple{\circlearrowright}$}

	\put(70,0.5){\vector(1,0){15}}
	\put(85,2.5){\vector(-1,0){15}}
\end{picture}
\end{center}
\vspace{0.5cm}
\caption{Rank two iquivers}
\label{tab:osp:rk2}
\end{table}

Following the $\imath$-fication rules on rank one and rank two Satake double quivers, we can assemble them together naturally to obtain modifications of arbitrary Satake double quivers of split and quasi-split ADE type.

Now we illustrate diagrammatically how we modify the representations $(V,W)$ of framed double quivers to become $(V^\imath, W^\imath)$ of the corresponding frame iquivers; this provides the input for the symplectic vector space $\iEVW$ in \eqref{iEVW}. It suffices to do so on rank one framed Satake double quivers. By attaching vector spaces $V_i,W_i$ to the nodes in Table \ref{tab:framed Satake:rk1} and $V_i^\imath, W_i^\imath$ defined in \eqref{VWimath} to the nodes in Table \ref{tab:framed osp:rk1}, we obtain Table \ref{tab:framed Satake:rk1:VW_i} and Table \ref{tab:framed osp:rk1:VWi} below. The $\purple\circlearrowleft$ inside $\bigcirc$
in Table \ref{tab:framed osp:rk1:VWi} is used to indicate the summand $\wedge^2V_i\oplus \wedge^2 V_i^*$ in $\iEVW$ in \eqref{iEVW} which corresponds to the case when $\I_1^{\texttt{qs}} =\{i\}$ in Remark \ref{rem:SSX}. Applying the $\imath$-fication rule to Table \ref{tab:framed Satake:rk1:VW_i} yields Table \ref{tab:framed osp:rk1:VWi}. Note that no arrow is indicated in the first diagram in Table \ref{tab:framed osp:rk1:VWi} thanks to a natural identification $\Hom(V_i^\imath,W_i^\imath) \cong \Hom(W_i^\imath,V_i^\imath)$.

\begin{table}[h]  
\begin{center}
\centering
\setlength{\unitlength}{0.6mm}		
\begin{picture}(70,35)(0,5)
\put(-38,30){$\bigcirc$}
\put(-38,6){\large $\Box$}
    
\put(-37,37){\small$V_i$}
\put(-37,0){\small$W_i$}
				
{\color{blue}
\put(-33.5,12.5){\vector(0,1){15.5}} 
\put(-35.5,28){\vector(0,-1){15.5}} }
%
    \put(19,30){$\bigcirc$}
	\put(42,30){$\bigcirc$}
	\put(19,6){\large $\Box$}
	\put(42.3,6){\large $\Box$}

	\put(19,0){\small$W_i$}
    \put(19,37){\small$V_i$}
	\put(43,0){\small$W_{\tau i}$}
    \put(43,37){\small$V_{\tau i}$}

{\color{purple}
    \qbezier(26,34)(33,39)(42,34)
    \put(26,34){\vector(-3,-1){0.5}}
    \put(42,34){\vector(3,-1){0.5}}
	\put(32,37){$\tau$}
    \qbezier(26,10)(33,15)(42,10)
    \put(26,10){\vector(-3,-1){0.5}}
    \put(42,10){\vector(3,-1){0.5}}
	\put(32,13){$\tau$}
    }
				
{\color{blue}
\put(21.7,12.5){\vector(0,1){15.5}}
\put(19.7,28){\vector(0,-1){15.5}}
\put(44.7,12.5){\vector(0,1){15.5}}
\put(42.7,28){\vector(0,-1){15.5}}
}
%
    \put(89,30){$\bigcirc$}
	\put(112,30){$\bigcirc$}
	\put(89,6){\large $\Box$}
	\put(112.3,6){\large $\Box$}

	\put(89,0){\small$W_i$}
    \put(89,37){\small$V_i$}
	\put(113,0){\small$W_{\tau i}$}
    \put(113,37){\small$V_{\tau i}$}

	\put(96,30.5){\vector(1,0){15}}
	\put(111,32.5){\vector(-1,0){15}}

{\color{purple}
    \qbezier(96,34)(103,39)(112,34)
    \put(96,34){\vector(-3,-1){0.5}}
    \put(112,34){\vector(3,-1){0.5}}
	\put(102,37){$\tau$}
    \qbezier(96,10)(103,15)(112,10)
    \put(96,10){\vector(-3,-1){0.5}}
    \put(112,10){\vector(3,-1){0.5}}
	\put(102,13){$\tau$}
    }
				
{\color{blue}
\put(91.7,12.5){\vector(0,1){15.5}}
\put(89.7,28){\vector(0,-1){15.5}}
\put(114.7,12.5){\vector(0,1){15.5}}
\put(112.7,28){\vector(0,-1){15.5}}
}
\end{picture}
\end{center}
\vspace{0.5cm}
\caption{Representations of rank one framed Satake double quivers}
\label{tab:framed Satake:rk1:VW_i}
\end{table}

\begin{table}[h]  
\begin{center}
\centering
\setlength{\unitlength}{0.6mm}		
\begin{picture}(70,35)(0,5)
\put(-26.2,29.8){\small $\bar{\epsilon}$}
\put(-28,30){$\bigcirc$}
\put(-28,6){\large $\Box$}
\put(-26.2,7.2){\small ${\epsilon}$}    
\put(-27,37){\small$V_i^\imath$}
\put(-27,0){\small$W_i^\imath$}				
{\color{blue}
\linethickness{.25mm}
\put(-24.5,28){\line(0,-1){15.5}}  
}
%
%
    \put(39,30){$\bigcirc$}
	\put(39,6){\large $\Box$}

	\put(40,0){\small$W_i$}
    \put(40,37){\small$V_i$}
				
{\color{blue}
\put(43.7,12.5){\vector(0,1){15.5}}
\put(41.7,28){\vector(0,-1){15.5}}
}
%
%
\put(112,30){$\bigcirc$}
\put(112,6){\large $\Box$}

\put(113,0){\small$W_i$}
\put(113,37){\small$V_i$}
\put(112.8,30){$\purple\circlearrowleft$}
\put(83,33){\tiny $\purple{\wedge^2V_i\oplus \wedge^2 V_i^*}$}

{\color{blue}
\put(116.5,12.5){\vector(0,1){15.5}}
\put(114.5,28){\vector(0,-1){15.5}}
}
\end{picture}
\end{center}
\vspace{0.5cm}
\caption{Representations of rank one framed iquivers}
\label{tab:framed osp:rk1:VWi}
\end{table}

\subsection{Discussions and examples}

\begin{eg} \label{ex:slN}
    In Table \ref{tab:ConeslN}, we specify the dimension vectors $\bv$ on $\bigcirc$ and $\bw$ on $\Box$ on framed double quiver of type A$_{N-1}$. It is well known that the corresponding Coulomb branch is identified with the nilpotent cone $\mathcal N_{\fksl_N}$ for $\fksl_N$, see for example \cite[\S 2(v)]{BFNro}.

\begin{table}[h]  
\begin{center}
\centering
\setlength{\unitlength}{0.6mm}		
\begin{picture}(30,35)(0,5)
	\put(-64,30){$\bigcirc$}
	\put(-43,30){$\bigcirc$}
	\put(-22,30){$\bigcirc$}
	\put(-1,30){$\bigcirc$}
	\put(49,30){$\bigcirc$}
	\put(70,30){$\bigcirc$}
	\put(71,6){$\Box$}

	\put(-62,36){\small$1$}
	\put(-41,36){\small$2$}
	\put(-20,36){\small$3$}
	\put(1,36){\small$4$}
    \put(41,36){\small${N-2}$}
	\put(66,36){\small${N-1}$}
	\put(71,0){\small${N}$}
				
	\put(-57,30.5){\vector(1,0){14}}
	\put(-36,30.5){\vector(1,0){14}}
	\put(-15,30.5){\vector(1,0){14}}
	\put(6,30.5){\vector(1,0){14}}
	\put(23,30){$\cdots$}
	\put(33.5,30.5){\vector(1,0){15}}
	\put(56,30.5){\vector(1,0){14}}
				
	\put(-43,32.5){\vector(-1,0){14}}
	\put(-22,32.5){\vector(-1,0){14}}
	\put(-1,32.5){\vector(-1,0){14}}
	\put(20,32.5){\vector(-1,0){14}}
	\put(23,30){$\cdots$}
	\put(48,32.5){\vector(-1,0){15}}
	\put(70,32.5){\vector(-1,0){14}}
				
	\color{blue}\put(74.3,11.5){\vector(0,1){16.5}}
	\color{blue}\put(72.5,28){\vector(0,-1){16.5}}
\end{picture}
\end{center}
\vspace{0.5cm}
\caption{Nilpotent cone for $\mathfrak{sl}_N$}
\label{tab:ConeslN}
\end{table}
\end{eg}

\begin{eg} \label{ex:so_2n}
Fix $N=2n$ in Table \ref{tab:ConeslN}, and specify the bipartite partition of $\I$ with the leftmost node to be in $\I^\ominus$. Note that the leftmost node $\bigcirc$ got removed under the $\imath$-fication  (see \eqref{VWimath}--\eqref{iGroups}) since the corresponding $V^\imath_i =0$.
Then applying the $\imath$-fication rules, we obtain  the (orth-symplectic) iquiver in Table \ref{tab:cone:so_2n}, where the new orthogonal/symplectic (or $\oplus/\ominus$, $\boxplus/\boxminus$ diagramatically) vector spaces $V_i^\imath$ and $W_i^\imath$ and their dimensions are specified. 

The corresponding iCoulomb branch $\mathcal M_C^\imath(V^\imath,W^\imath)$ is indeed the nilpotent cone $\mathcal N_{\mathfrak{so}_{2n}}$ for $\mathfrak{so}_{2n}$. In physics this is well-known due to the foundational work of Gaiotto-Witten \cite[Figure 54]{GaioW09}, and is verified mathematically in the upcoming work of Finkelberg-Hanany-Nakajima \cite{FHN25}.

\begin{table}[h]  
\begin{center}
\centering
\setlength{\unitlength}{0.6mm}			
\begin{picture}(30,35)(0,5)
	\put(-64,30){\large $\oplus$}
	\put(-43,30){\large $\ominus$}
	\put(-22,30){\large $\oplus$}
	\put(-1,30){\large $\ominus$}
	\put(49,30){\large $\oplus$}
	\put(70,30){\large $\ominus$}
	\put(70.3,10){$\boxplus$}

	\put(-63,36){\small$2$}
	\put(-42,36){\small$2$}
	\put(-21,36){\small$4$}
	\put(0,36){\small$4$}
	\put(40,36){\small${2n-2}$}
	\put(64,36){\small${2n-2}$}
	\put(69,4){\small${2n}$}
				
    \linethickness{.2mm}
	\put(-57.6,32){\line(1,0){14.5}}
	\put(-36.6,32){\line(1,0){14.5}}
	\put(-15.6,32){\line(1,0){14.5}}
	\put(5.6,32){\line(1,0){15}}
	\put(23,30){$\cdots$}
	\put(33.6,32){\line(1,0){15}}
	\put(55.6,32){\line(1,0){14.5}}
				
	\color{blue}
\linethickness{.25mm}
    \put(72.9,15.5){\line(0,1){13}}
			\end{picture}
\end{center}
\vspace{0.5cm}
\caption{Nilpotent cone for $\mathfrak{so}_{2n}$}
\label{tab:cone:so_2n}
\end{table}
\end{eg}

\begin{eg}  \label{ex:so_2n+1}
Fix $N=2n+1$ in Table \ref{tab:ConeslN}, and specify the bipartite partition of $\I$ with the leftmost node to be in $\I^\ominus$. The leftmost node $\bigcirc$ is again removed under $\imath$-fication  (see \eqref{VWimath}--\eqref{iGroups}) since the corresponding $V^\imath_i =0$.
Then applying the $\imath$-fication rules, we obtain  the (ortho-symplectic) iquiver in Table \ref{tab:cone:so_2n+1}, where the new orthogonal/symplectic vector spaces $V_i^\imath$ and $W_i^\imath$ and their dimensions are specified. 

The corresponding iCoulomb branch $\mathcal M_C^\imath(V^\imath,W^\imath)$ is indeed the nilpotent cone $\mathcal N_{\mathfrak{so}_{2n+1}}$; we refer the reader again to Gaiotto-Witten \cite{GaioW09} and the forthcoming work of Finkelberg-Hanany-Nakajima \cite{FHN25}.

 \begin{table}[h]  
\begin{center}
\centering
\setlength{\unitlength}{0.6mm}			
\begin{picture}(30,35)(0,5)
	\put(-64,30){\large $\oplus$}
	\put(-43,30){\large $\ominus$}
	\put(-22,30){\large $\oplus$}
	\put(-1,30){\large $\ominus$}
	\put(49,30){\large $\oplus$}
	\put(70.5,30){\large $\ominus$}
	\put(92,30){\large $\oplus$}
	\put(92.2,10){$\boxminus$}

	\put(-63,36){\small$2$}
	\put(-42,36){\small$2$}
	\put(-21,36){\small$4$}
	\put(0,36){\small$4$}
	\put(40,36){\small $2n-2$}
	\put(64,36){\small $2n-2$}
	\put(91,36){\small $2n$}
	\put(91,4){\small $2n$}
				
    \linethickness{.2mm}
	\put(-57.6,32){\line(1,0){14.5}}
	\put(-36.6,32){\line(1,0){14.5}}
	\put(-15.6,32){\line(1,0){14.5}}
	\put(5.6,32){\line(1,0){15}}
	\put(23.6,30){$\cdots$}
	\put(34.1,32){\line(1,0){15}}
	\put(55.6,32){\line(1,0){14.5}}
	\put(77.1,32){\line(1,0){14.5}}
				
	\color{blue}
    \linethickness{.25mm}
\put(94.9,15.5){\line(0,1){13}}
			\end{picture}
\end{center}
\vspace{0.5cm}
\caption{Nilpotent cone for $\mathfrak{so}_{2n+1}$}
\label{tab:cone:so_2n+1}
\end{table}
\end{eg}

We continue to assume that $V,W$ satisfy the $\tau$-symmetry/parity conditions \eqref{symmetricVW:I1}. 

By \cite[Theorem 3.10]{BFN19},  generalized affine Grassmannian slices $\overline{\cW}_\mu^\lambda$ are realized as Coulomb branches $\mc M_C(V, W)$ for specified dimension vectors $\bv, \bw$ or equivalently for specific groups $GL(V), GL(W)$. Such group datum give rise to new groups $G^\imath(V^\imath)$ and $G^\imath(W^\imath)$ as in \eqref{iGroups} and then the corresponding iCoulomb branch $\mathcal M_C^\imath(V^\imath,W^\imath)$.

\begin{conj}  \label{conj:iCBr:islices}
\begin{enumerate}
    \item 
    The iCoulomb branch $\mathcal M_C^\imath(V^\imath,W^\imath)$ is a normalization of a top-dimensional component of the affine Grassmannian islice ${}^\imath\overline{\cW}_\mu^\lambda$.
    \item 
    Truncated shifted iYangians are (subalgebras of) quantized iCoulomb branches.  
    \item 
    Truncated shifted affine iquantum groups are (subalgebras of) K-theoretic iCoulomb branch algebras.  
\end{enumerate}
\end{conj}

\begin{rem}  \label{rem:SameDimension}
By Theorem \ref{thm:ctgklo}(2) and Remark \ref{rem:ivector:rank}, the dimension of ${}^\imath\overline{\cW}_\mu^\lambda$ is equal to twice the rank of the group $G^\imath(V^\imath)$, which is the known dimension of the iCoulomb branch $\mathcal M_C^\imath(V^\imath,W^\imath)$. In addition, the deformation space ${}^\imath \bA^{|\lambda|}$ in \eqref{isliceDeform} has dimension equal to $\sum_{i\in\iI} \fkw_i$, the rank of the group $G^\imath(W^\imath)$ by Remark \ref{rem:ivector:rank} again. This is consistent with Conjecture \ref{conj:iCBr:islices}. 
\end{rem}

\begin{rem}  \label{rem:physics}
 For type AI, in light of Theorem \ref{thm:loci=sliceBCD}, Conjecture \ref{conj:iCBr:islices}(1) can be rephrased as a connection between iCoulomb branches and a family of nilpotent Slodowy slices of type BCD. This generally agrees with physical predictions \cite{GaioW09,CDT13,CHMZ15,HK16,CHZ17,CHananyK19}, modulo some subtleties discussed below.
\end{rem}

\begin{rem} \label{rem:subtle}
Conjecture \ref{conj:iCBr:islices} as formulated is perhaps too general to take literally.  For example, Cabrera-Hanany-Zhong \cite{CHZ17} computed the monopole formulas for Coulomb branches of certain orthosymplectic quivers of type AI and identified these with Hilbert series for affinizations of classical nilpotent orbits; combining these with Theorem \ref{thm:loci=sliceBCD}, it is not compatible with our conjecture. There are also ambiguities in our conjecture, the most immediate being a sign ambiguity arising from the choice of bipartite partition \eqref{bipartite}. There is also a subtle issue of choosing $G^{+}(V_i^\imath)$ or $G^{+}(W_i^\imath)$ to be orthogonal or special orthogonal, see e.g.~\cite{CHZ17}.
Nevertheless, we hope that it points toward a meaningful connection between seemingly unrelated subjects which can be made rigorous, in particular putting previous physical expectations in a general framework of Satake diagrams. 
\end{rem}

A natural sufficient condition for the third assumption of \eqref{symmetricVW:I1}
 is that $\bv=(\bv_i)_{i\in\I}$ is parity-alternating. Under such a parity-alternating condition, we can fix the ambiguity of  \eqref{bipartite} by letting $\I^\oplus$ (resp. $\I^\ominus$) correspond to even (resp. odd) $\bv_i$. In this case, the type B group does not appear as a component group of $G^\imath(V^\imath)$. 

Let us specialize to the type A$_{N-1}$ quiver with $\I=\{1,\ldots, N-1\}$. We denote $\la =\sum_{a=1}^N\la_a\epsilon^\vee_a$, $\mu=\sum_{a=1}^N\mu_a\epsilon^\vee_a$, and write $\la-\mu=\sum_{i\in\I} \bv_i \alpha_i^\vee =\sum_{a=l}^u c_a \epsilon^\vee_a$, with $c_a\in \Z$ and $c_l,c_u\ne 0$. Then one checks that 
\begin{align*}
    \text{(a) $\bv$ is parity-alternating} \Leftrightarrow \text{ (b) }
    \begin{cases}
        \text{$c_a \, (l<a<u)$ are odd, and}
        \\
        \text{$\#\{l \lle a\lle u\mid c_a \text{ is odd} \}$ is even.}    
    \end{cases}
\end{align*}
Actually the set of $\bv$ satisfying (a) is in bijection with the set of $(c_a\mid 1\lle a\lle N)$ satisfying (b) together with $\sum_{a} c_a=0$.

For example, let $\la=(N)$. In this case, if $\mu\vdash N$ has all parts odd, then $\bv$ is parity-alternating. 

Take another example with $\mu=(1^N)$. In this case, the partition $\la=(\la_1, \la_2, \ldots) \vdash N$ satisfies that $\la_i$ are even for all $i\gge 2$ if and only if $\bv$ is parity-alternating. 
In Examples~ \ref{ex:slN}, \ref{ex:so_2n} and \ref{ex:so_2n+1}, we have $\la=(N)$ and $\mu=(1^N)$.

We end with raising a question of geometric symmetric pairs. 

\begin{problem}
Is $\mathcal M_C(V^\imath,W^\imath)$ a normalization of a Dirac reduction of $\mc M_C(V,W)$? 
\end{problem} 
We may refer to $\big(\mc M_C(V, W)$, $\mathcal M_C^\imath(V^\imath,W^\imath)\big)$ as a geometric symmetric pair if the answer to the above problem is affirmative. It might be of interest to study geometric symmetric pairs under symplectic duality (see also \cite{Li19} on Higgs side, and the recent work of Nakajima \cite{Nak25}). If Conjecture \ref{conj:iCBr:islices} holds for islices corresponding to the nilpotent Slodowy slices of type BCD covered by Theorem \ref{thm:loci=sliceBCD}, then we obtain a large class of geometric symmetric pairs. Already in this case, understanding the Poisson involution $\sigma$ within the framework of iCoulomb branches could be a good starting point.

\bibliographystyle{amsalpha}
\bibliography{references}

@article{Lu25,
    author = {Lu, K.},
    title = {{Minimalistic presentation and coideal structure of twisted Yangians}},
   JOURNAL = {Commun. Math. Phys.},
   volume = {407},
   number = {5},
   pages = {article no. 99},
      YEAR = {2026},
NOTE ={},
    URL = {}
}

@article {Mac02,
    AUTHOR = {MacKay, N.},
     TITLE = {Rational {$K$}-matrices and representations of twisted
              {Y}angians},
   JOURNAL = {J. Phys. A},
  FJOURNAL = {Journal of Physics. A. Mathematical and General},
    VOLUME = {35},
      YEAR = {2002},
    NUMBER = {37},
     PAGES = {7865--7876},
      ISSN = {0305-4470,1751-8121},
   MRCLASS = {81R12 (17B37 81R50 81T40)},
  MRNUMBER = {1945798},
MRREVIEWER = {Sonia\ Natale},
       DOI = {10.1088/0305-4470/35/37/302},
       URL = {https://doi.org/10.1088/0305-4470/35/37/302},
}

@article{LWW25sty,
    author = {Lu, Kang and Wang, Weiqiang and Weekes, Alex},
    title = {Shifted twisted {Y}angians of quasi-split {ADE} types},
   JOURNAL = {arXiv preprint \arxiv{2512.19998}},
      YEAR = {2025},
NOTE ={},
    URL = {}
}

@article {FJLS17,
    AUTHOR = {Fu, Baohua and Juteau, Daniel and Levy, Paul and Sommers,
              Eric},
     TITLE = {Generic singularities of nilpotent orbit closures},
   JOURNAL = {Adv. Math.},
  FJOURNAL = {Advances in Mathematics},
    VOLUME = {305},
      YEAR = {2017},
     PAGES = {1--77},
      ISSN = {0001-8708,1090-2082},
   MRCLASS = {14B05 (17B08 37B05)},
  MRNUMBER = {3570131},
MRREVIEWER = {Arvid\ Siqveland},
       DOI = {10.1016/j.aim.2016.09.010},
       URL = {https://doi.org/10.1016/j.aim.2016.09.010},
}

@article {B23,
    AUTHOR = {Bellamy, Gwyn},
     TITLE = {Coulomb branches have symplectic singularities},
   JOURNAL = {Lett. Math. Phys.},
  FJOURNAL = {Letters in Mathematical Physics},
    VOLUME = {113},
      YEAR = {2023},
    NUMBER = {5},
     PAGES = {Paper No. 104, 8},
      ISSN = {0377-9017,1573-0530}
}

@article {BF17,
    AUTHOR = {Braverman, Alexander and Finkelberg, Michael},
     TITLE = {Twisted zastava and {$q$}-{W}hittaker functions},
   JOURNAL = {J. Lond. Math. Soc. (2)},
  FJOURNAL = {Journal of the London Mathematical Society. Second Series},
    VOLUME = {96},
      YEAR = {2017},
    NUMBER = {2},
     PAGES = {309--325},
      ISSN = {0024-6107,1469-7750},
   MRCLASS = {17B10 (17B37 33C67)},
  MRNUMBER = {3708952},
MRREVIEWER = {Vivek\ Sahai},
       DOI = {10.1112/jlms.12057},
       URL = {https://doi.org/10.1112/jlms.12057},
}

@article{BF24,
    author = {Bielawski, Roger and Foscolo, Lorenzo},
    title = {Hypertoric varieties, {W}-{H}ilbert schemes, and {C}oulomb branches},
    journal = {arXiv preprint \arxiv{2304.08125}},
    year = 2024,
    url = {https://doi.org/abs/2304.08125}
}

@article {CHananyK19,
    AUTHOR = {Cabrera, Santiago and Hanany, Amihay and Kalveks, Rudolph},
     TITLE = {Quiver theories and formulae for {S}lodowy slices of classical
              algebras},
   JOURNAL = {Nuclear Phys. B},
  FJOURNAL = {Nuclear Physics. B. Theoretical, Phenomenological, and
              Experimental High Energy Physics. Quantum Field Theory and
              Statistical Systems},
    VOLUME = {939},
      YEAR = {2019},
     PAGES = {308--357},
      ISSN = {0550-3213},
   MRCLASS = {17B08 (14D21 16G20)},
  MRNUMBER = {3894300},
MRREVIEWER = {Ada Boralevi},
       DOI = {10.1016/j.nuclphysb.2018.12.022},
       URL = {https://doi.org/10.1016/j.nuclphysb.2018.12.022},
}

@article {FLLLW,
    AUTHOR = {Fan, Zhaobing and Lai, Chun-Ju and Li, Yiqiang and Luo, Li and
              Wang, Weiqiang},
     TITLE = {Affine flag varieties and quantum symmetric pairs},
   JOURNAL = {Mem. Amer. Math. Soc.},
  FJOURNAL = {Memoirs of the American Mathematical Society},
    VOLUME = {265},
      YEAR = {2020},
    NUMBER = {1285},
     PAGES = {v+123},
      ISSN = {0065-9266,1947-6221},
      ISBN = {978-1-4704-4175-3; 978-1-4704-6138-6},
   MRCLASS = {17B37 (14F43 14M15 20G25)},
  MRNUMBER = {4080913},
MRREVIEWER = {Zhihua\ Chang},
       DOI = {10.1090/memo/1285},
       URL = {https://doi.org/10.1090/memo/1285},
}

@article {GaioW09,
    AUTHOR = {Gaiotto, Davide and Witten, Edward},
     TITLE = {{$S$}-duality of boundary conditions in {$\mathscr N=4$} super
              {Y}ang-{M}ills theory},
   JOURNAL = {Adv. Theor. Math. Phys.},
  FJOURNAL = {Advances in Theoretical and Mathematical Physics},
    VOLUME = {13},
      YEAR = {2009},
    NUMBER = {3},
     PAGES = {721--896},
      ISSN = {1095-0761},
   MRCLASS = {81T60 (81T13 81T30)},
  MRNUMBER = {2610576},
MRREVIEWER = {Douglas J. Smith},
       DOI = {10.4310/atmp.2009.v13.n3.a5},
       URL = {https://doi.org/10.4310/atmp.2009.v13.n3.a5},
}

@article {NW23,
    AUTHOR = {Nakajima, Hiraku and Weekes, Alex},
     TITLE = {Coulomb branches of quiver gauge theories with symmetrizers},
   JOURNAL = {J. Eur. Math. Soc. (JEMS)},
  FJOURNAL = {Journal of the European Mathematical Society (JEMS)},
    VOLUME = {25},
      YEAR = {2023},
    NUMBER = {1},
     PAGES = {203--230},
      ISSN = {1435-9855,1435-9863},
   MRCLASS = {14D21 (16G20 17B37 81T13)},
  MRNUMBER = {4556783},
MRREVIEWER = {\^Angela\ Mestre},
       DOI = {10.4171/jems/1176},
       URL = {https://doi.org/10.4171/jems/1176},
}

@article{SuW24,
  title={{Affine $\imath$quantum groups and Steinberg varieties of type C}},
  author={Su, Changjian and Wang, Weiqiang},
  journal={arXiv preprint \arxiv{2407.06865}},
  year={2024}
}

@article{BDFRT,
    author = {Braverman, A. and Dhillon, G. and Finkelberg, M. and Raskin, S. and Travkin, R.},
    title = {{C}oulomb branches of noncotangent type (with appendices by {G}urbir {D}hillon
and {T}heo {J}ohnson-{F}reyd)},
   JOURNAL = {arXiv preprint \arxiv{2201.09475}},
      YEAR = {2022},
    URL = {
https://doi.org/abs/2201.09475}
}

@article {BR17,
    AUTHOR = {Belliard, Samuel and Regelskis, Vidas},
     TITLE = {Drinfeld {J} presentation of twisted {Y}angians},
   JOURNAL = {SIGMA Symmetry Integrability Geom. Methods Appl.},
  FJOURNAL = {SIGMA. Symmetry, Integrability and Geometry. Methods and
              Applications},
    VOLUME = {13},
      YEAR = {2017},
     PAGES = {Paper No. 011, 35},
      ISSN = {1815-0659}
}

@article{BFN18,
	AUTHOR = {Braverman, Alexander and Finkelberg, Michael and Nakajima, Hiraku},
     TITLE = {Towards a mathematical definition of {C}oulomb branches of 3-dimensional {$\mathcal N=4$} gauge theories, {II}},
   JOURNAL = {Adv. Theor. Math. Phys.},
  FJOURNAL = {Advances in Theoretical and Mathematical Physics},
    VOLUME = {22},
      YEAR = {2018},
    NUMBER = {5},
     PAGES = {1071--1147}
}

@article{BFN19,
	AUTHOR = {Braverman, Alexander and Finkelberg, Michael and Nakajima, Hiraku},
     TITLE = {Coulomb branches of {$3d$} {$\mathcal N=4$} quiver gauge theories
              and slices in the affine {G}rassmannian},
      NOTE = {with two appendices by Braverman, Finkelberg, Joel Kamnitzer,
              Ryosuke Kodera, Nakajima, Ben Webster and Alex Weekes},
   JOURNAL = {Adv. Theor. Math. Phys.},
  FJOURNAL = {Advances in Theoretical and Mathematical Physics},
    VOLUME = {23},
      YEAR = {2019},
    NUMBER = {1},
     PAGES = {75--166}
}

@article {BFNro,
    AUTHOR = {Braverman, Alexander and Finkelberg, Michael and Nakajima,
              Hiraku},
     TITLE = {Ring objects in the equivariant derived {S}atake category
              arising from {C}oulomb branches},
      NOTE = {Appendix by Gus Lonergan},
   JOURNAL = {Adv. Theor. Math. Phys.},
  FJOURNAL = {Advances in Theoretical and Mathematical Physics},
    VOLUME = {23},
      YEAR = {2019},
    NUMBER = {2},
     PAGES = {253--344},
      ISSN = {1095-0761,1095-0753},
   MRCLASS = {14F08 (14M15 16G20 17B45 81T13)},
  MRNUMBER = {4033353},
MRREVIEWER = {Dave\ Auckly},
       DOI = {10.4310/ATMP.2019.v23.n2.a1},
       URL = {https://doi-org.ezproxy.usherbrooke.ca/10.4310/ATMP.2019.v23.n2.a1},
}

@article {BG03,
    AUTHOR = {Brown, Kenneth A. and Gordon, Iain},
     TITLE = {Poisson orders, symplectic reflection algebras and
              representation theory},
   JOURNAL = {J. Reine Angew. Math.},
  FJOURNAL = {Journal f\"ur die Reine und Angewandte Mathematik. [Crelle's
              Journal]},
    VOLUME = {559},
      YEAR = {2003},
     PAGES = {193--216}
}

@article {BK06,
    AUTHOR = {Brundan, Jonathan and Kleshchev, Alexander},
     TITLE = {Shifted {Y}angians and finite {$W$}-algebras},
   JOURNAL = {Adv. Math.},
  FJOURNAL = {Advances in Mathematics},
    VOLUME = {200},
      YEAR = {2006},
    NUMBER = {1},
     PAGES = {136--195},
      ISSN = {0001-8708},
   MRCLASS = {17B37},
  MRNUMBER = {2199632},
MRREVIEWER = {Chengming Bai},
       DOI = {10.1016/j.aim.2004.11.004},
       URL = {https://doi.org/10.1016/j.aim.2004.11.004},
}

@article {BK08,
    AUTHOR = {Brundan, Jonathan and Kleshchev, Alexander},
     TITLE = {Representations of shifted {Y}angians and finite
              {$W$}-algebras},
   JOURNAL = {Mem. Amer. Math. Soc.},
  FJOURNAL = {Memoirs of the American Mathematical Society},
    VOLUME = {196},
      YEAR = {2008},
    NUMBER = {918},
     PAGES = {viii+107},
      ISSN = {0065-9266,1947-6221},
      ISBN = {978-0-8218-4216-4},
   MRCLASS = {17B37 (17B10)},
  MRNUMBER = {2456464},
MRREVIEWER = {Serge\ M.\ Skryabin},
       DOI = {10.1090/memo/0918},
       URL = {https://doi.org/10.1090/memo/0918},
}

@article {BKLW18,
    AUTHOR = {Bao, Huanchen and Kujawa, Jonathan and Li, Yiqiang and Wang,
              Weiqiang},
     TITLE = {Geometric {S}chur duality of classical type},
   JOURNAL = {Transform. Groups},
  FJOURNAL = {Transformation Groups},
    VOLUME = {23},
      YEAR = {2018},
    NUMBER = {2},
     PAGES = {329--389},
      ISSN = {1083-4362,1531-586X},
   MRCLASS = {17B37 (14M15)},
  MRNUMBER = {3805209},
MRREVIEWER = {Alexandr\ Nikolaevich\ Zubkov},
       DOI = {10.1007/s00031-017-9447-4},
       URL = {https://doi.org/10.1007/s00031-017-9447-4},
}

@article {BW18KL,
    AUTHOR = {Bao, Huanchen and Wang, Weiqiang},
     TITLE = {A new approach to {K}azhdan-{L}usztig theory of type {$B$} via quantum symmetric pairs},
   JOURNAL = {Astérisque},
  FJOURNAL = {Astérisque},
      YEAR = {2018},
    VOLUME=  {402},
    PAGES = {vii+134},
      ISSN = {0303-1179},
      ISBN = {978-2-85629-889-3},
   MRCLASS = {17B37},
URL={https://smf.emath.fr/system/files/2017-08/smf_ast_402.pdf}
}

@article {BW18iCB,
    AUTHOR = {Bao, Huanchen and Wang, Weiqiang},
     TITLE = {Canonical bases arising from quantum symmetric pairs},
   JOURNAL = {Invent. Math.},
  FJOURNAL = {Inventiones Mathematicae},
    VOLUME = {213},
      YEAR = {2018},
    NUMBER = {3},
     PAGES = {1099--1177},
      ISSN = {0020-9910},
   MRCLASS = {17B10},
       DOI = {10.1007/s00222-018-0801-5},
       URL = {https://doi.org/10.1007/s00222-018-0801-5},
}

@article {CK18,
    AUTHOR = {Cautis, Sabin and Kamnitzer, Joel},
     TITLE = {Categorical geometric symmetric {H}owe duality},
   JOURNAL = {Selecta Math. (N.S.)},
  FJOURNAL = {Selecta Mathematica. New Series},
    VOLUME = {24},
      YEAR = {2018},
    NUMBER = {2},
     PAGES = {1593--1631}
}

@article {CDT13,
    AUTHOR = {Chacaltana, Oscar and Distler, Jacques and Tachikawa, Yuji},
     TITLE = {Nilpotent orbits and codimension-2 defects of 6d {$\scr
              N=(2,0)$} theories},
   JOURNAL = {Internat. J. Modern Phys. A},
  FJOURNAL = {International Journal of Modern Physics A. Particles and
              Fields. Gravitation. Cosmology. Astrophysics. Accelerator
              Physics},
    VOLUME = {28},
      YEAR = {2013},
    NUMBER = {3-4},
     PAGES = {1340006, 54},
      ISSN = {0217-751X,1793-656X},
   MRCLASS = {81T40 (81R10)},
  MRNUMBER = {3027737},
MRREVIEWER = {Stanislav\ Z.\ Pakuliak},
       DOI = {10.1142/S0217751X1340006X},
       URL = {https://doi-org.ezproxy.usherbrooke.ca/10.1142/S0217751X1340006X},
}

@book {CP94,
    AUTHOR = {Chari, Vyjayanthi and Pressley, Andrew},
     TITLE = {A guide to quantum groups},
 PUBLISHER = {Cambridge University Press, Cambridge},
      YEAR = {1994},
     PAGES = {xvi+651},
      ISBN = {0-521-43305-3},
   MRCLASS = {17B37 (16W30 81R50)}
}

@article {CG06,
    AUTHOR = {Ciccoli, Nicola and Gavarini, Fabio},
     TITLE = {A quantum duality principle for coisotropic subgroups and
              {P}oisson quotients},
   JOURNAL = {Adv. Math.},
  FJOURNAL = {Advances in Mathematics},
    VOLUME = {199},
      YEAR = {2006},
    NUMBER = {1},
     PAGES = {104--135},
      ISSN = {0001-8708,1090-2082}
}

@article {CG14,
    AUTHOR = {Ciccoli, Nicola and Gavarini, Fabio},
     TITLE = {A global quantum duality principle for subgroups and
              homogeneous spaces},
   JOURNAL = {Doc. Math.},
  FJOURNAL = {Documenta Mathematica},
    VOLUME = {19},
      YEAR = {2014},
     PAGES = {333--380},
      ISSN = {1431-0635,1431-0643}
}

@book {CM93,
    AUTHOR = {Collingwood, David H. and McGovern, William M.},
     TITLE = {Nilpotent orbits in semisimple {L}ie algebras},
    SERIES = {Van Nostrand Reinhold Mathematics Series},
 PUBLISHER = {Van Nostrand Reinhold Co., New York},
      YEAR = {1993},
     PAGES = {xiv+186},
      ISBN = {0-534-18834-6}
}

@article {CHMZ15,
    AUTHOR = {Cremonesi, Stefano and Hanany, Amihay and Mekareeya, Noppadol
              and Zaffaroni, Alberto},
     TITLE = {{$T^\sigma_\rho(G)$} theories and their {H}ilbert series},
   JOURNAL = {J. High Energy Phys.},
  FJOURNAL = {Journal of High Energy Physics},
      YEAR = {2015},
    NUMBER = {1},
     PAGES = {150, front matter+54},
      ISSN = {1126-6708,1029-8479},
   MRCLASS = {81T13 (81T60)},
  MRNUMBER = {3313892},
MRREVIEWER = {Rafael\ Hern\'andez Redondo},
       DOI = {10.1007/JHEP01(2015)150},
       URL = {https://doi-org.ezproxy.usherbrooke.ca/10.1007/JHEP01(2015)150},
}

@article {CHZ17,
    AUTHOR = {Cabrera, Santiago and Hanany, Amihay and Zhong, Zhenghao},
     TITLE = {Nilpotent orbits and the {C}oulomb branch of {$T^\sigma(G)$}
              theories: special orthogonal vs orthogonal gauge group
              factors},
   JOURNAL = {J. High Energy Phys.},
  FJOURNAL = {Journal of High Energy Physics},
      YEAR = {2017},
    NUMBER = {11},
     PAGES = {079, front matter+28},
      ISSN = {1126-6708,1029-8479},
   MRCLASS = {81T13},
  MRNUMBER = {3747228},
       DOI = {10.1007/jhep11(2017)079},
       URL = {https://doi-org.ezproxy.usherbrooke.ca/10.1007/jhep11(2017)079},
}

@article {Dr85,
    AUTHOR = {Drinfeld, V. G.},
     TITLE = {Hopf algebras and the quantum {Y}ang-{B}axter equation},
   JOURNAL = {Dokl. Akad. Nauk SSSR},
  FJOURNAL = {Doklady Akademii Nauk SSSR},
    VOLUME = {283},
      YEAR = {1985},
    NUMBER = {5},
     PAGES = {1060--1064},
      ISSN = {0002-3264}
}

@article {Dr87,
    AUTHOR = {Drinfeld, V. G.},
     TITLE = {A new realization of {Y}angians and of quantum affine
              algebras},
   JOURNAL = {Dokl. Akad. Nauk SSSR},
  FJOURNAL = {Doklady Akademii Nauk SSSR},
    VOLUME = {296},
      YEAR = {1987},
    NUMBER = {1},
     PAGES = {13--17},
      ISSN = {0002-3264},
   MRCLASS = {17B65 (16A24 17B45 58F07 81E99)},
  MRNUMBER = {914215},
MRREVIEWER = {J. S. Joel},
}

@inproceedings {Dr87b,
    AUTHOR = {Drinfeld, V. G.},
     TITLE = {Quantum groups},
 BOOKTITLE = {Proceedings of the {I}nternational {C}ongress of
              {M}athematicians, {V}ol. 1, 2 ({B}erkeley, {C}alif., 1986)},
     PAGES = {798--820},
 PUBLISHER = {Amer. Math. Soc., Providence, RI},
      YEAR = {1987},
      ISBN = {0-8218-0110-4}
}

@article {DW02,
    AUTHOR = {Derksen, Harm and Weyman, Jerzy},
     TITLE = {Generalized quivers associated to reductive groups},
   JOURNAL = {Colloq. Math.},
  FJOURNAL = {Colloquium Mathematicum},
    VOLUME = {94},
      YEAR = {2002},
    NUMBER = {2},
     PAGES = {151--173},
      ISSN = {0010-1354},
   MRCLASS = {16G20 (14L35)},
  MRNUMBER = {1967372},
MRREVIEWER = {Iain G. Gordon},
       DOI = {10.4064/cm94-2-1},
       URL = {https://doi.org/10.4064/cm94-2-1},
}

@article {E92,
    AUTHOR = {Edixhoven, Bas},
     TITLE = {N\'eron models and tame ramification},
   JOURNAL = {Compositio Math.},
  FJOURNAL = {Compositio Mathematica},
    VOLUME = {81},
      YEAR = {1992},
    NUMBER = {3},
     PAGES = {291--306},
      ISSN = {0010-437X,1570-5846},
   MRCLASS = {14K15 (11G10)},
  MRNUMBER = {1149171},
MRREVIEWER = {Min\ Ho\ Lee},
       URL = {http://www.numdam.org/item?id=CM_1992__81_3_291_0},
}

@book {F94,
    AUTHOR = {Fernandes, Rui Antonio Loja},
     TITLE = {Completely integrable bi-{H}amiltonian systems},
      NOTE = {Thesis (Ph.D.)--University of Minnesota},
 PUBLISHER = {ProQuest LLC, Ann Arbor, MI},
      YEAR = {1994},
     PAGES = {88},
       URL =
              {http://gateway.proquest.com/openurl?url_ver=Z39.88-2004&rft_val_fmt=info:ofi/fmt:kev:mtx:dissertation&res_dat=xri:pqdiss&rft_dat=xri:pqdiss:9415479},
}

@unpublished{FHN25,
    author = {Finkelberg, Michael and Hanany, Amihay and Nakajima, Hiraku},
    title = {private communication},
      YEAR = {2025}
}

@article {FKPRW,
    AUTHOR = {Finkelberg, Michael and Kamnitzer, Joel and Pham, Khoa and
              Rybnikov, Leonid and Weekes, Alex},
     TITLE = {Comultiplication for shifted {Y}angians and quantum open
              {T}oda lattice},
   JOURNAL = {Adv. Math.},
  FJOURNAL = {Advances in Mathematics},
    VOLUME = {327},
      YEAR = {2018},
     PAGES = {349--389}
}

@article {FKMM99,
    AUTHOR = {Finkelberg, Michael and Kuznetsov, Alexander and Markarian,
              Nikita and Mirkovi\'c, Ivan},
     TITLE = {A note on a symplectic structure on the space of
              {$G$}-monopoles},
   JOURNAL = {Comm. Math. Phys.},
  FJOURNAL = {Communications in Mathematical Physics},
    VOLUME = {201},
      YEAR = {1999},
    NUMBER = {2},
     PAGES = {411--421}
}

@incollection {FT19,
    AUTHOR = {Finkelberg, Michael and Tsymbaliuk, Alexander},
     TITLE = {Multiplicative slices, relativistic {T}oda and shifted quantum
              affine algebras},
 BOOKTITLE = {Representations and nilpotent orbits of {L}ie algebraic
              systems},
    SERIES = {Progr. Math.},
    VOLUME = {330},
     PAGES = {133--304},
 PUBLISHER = {Birkh\"{a}user/Springer, Cham},
      YEAR = {2019},
   MRCLASS = {17B37 (81R10 81T13)},
  MRNUMBER = {3971731},
MRREVIEWER = {Kyungyong Lee},
       DOI = {10.1007/978-3-030-23531-4\_6},
       URL = {https://doi.org/10.1007/978-3-030-23531-4_6},
}

@article {FT19b,
    AUTHOR = {Finkelberg, Michael and Tsymbaliuk, Alexander},
     TITLE = {Shifted quantum affine algebras: integral forms in type {$A$}},
NOTE = {with appendices by Alexander Tsymbaliuk and Alex Weekes},
   JOURNAL = {Arnold Math. J.},
  FJOURNAL = {Arnold Mathematical Journal},
    VOLUME = {5},
      YEAR = {2019},
    NUMBER = {2-3},
     PAGES = {197--283}
}

@article {F73,
    AUTHOR = {Fogarty, John},
     TITLE = {Fixed point schemes},
   JOURNAL = {Amer. J. Math.},
  FJOURNAL = {American Journal of Mathematics},
    VOLUME = {95},
      YEAR = {1973},
     PAGES = {35--51},
      ISSN = {0002-9327,1080-6377},
   MRCLASS = {14L99},
  MRNUMBER = {332805},
MRREVIEWER = {G.\ Horrocks},
       DOI = {10.2307/2373642},
       URL = {https://doi.org/10.2307/2373642},
}

@article {FZ99,
    AUTHOR = {Fomin, Sergey and Zelevinsky, Andrei},
     TITLE = {Double {B}ruhat cells and total positivity},
   JOURNAL = {J. Amer. Math. Soc.},
  FJOURNAL = {Journal of the American Mathematical Society},
    VOLUME = {12},
      YEAR = {1999},
    NUMBER = {2},
     PAGES = {335--380},
      ISSN = {0894-0347,1088-6834},
   MRCLASS = {20G20 (15A23)},
  MRNUMBER = {1652878},
       DOI = {10.1090/S0894-0347-99-00295-7},
       URL = {https://doi.org/10.1090/S0894-0347-99-00295-7},
}

@article {GG02,
    AUTHOR = {Gan, Wee Liang and Ginzburg, Victor},
     TITLE = {Quantization of {S}lodowy slices},
   JOURNAL = {Int. Math. Res. Not.},
  FJOURNAL = {International Mathematics Research Notices},
      YEAR = {2002},
    NUMBER = {5},
     PAGES = {243--255}
}

@article {G02,
    AUTHOR = {Gavarini, Fabio},
     TITLE = {The quantum duality principle},
   JOURNAL = {Ann. Inst. Fourier (Grenoble)},
  FJOURNAL = {Universit\'e{} de Grenoble. Annales de l'Institut Fourier},
    VOLUME = {52},
      YEAR = {2002},
    NUMBER = {3},
     PAGES = {809--834},
      ISSN = {0373-0956,1777-5310},
   MRCLASS = {17B37 (20G42)},
  MRNUMBER = {1907388},
MRREVIEWER = {Olivier\ G.\ Schiffmann},
       DOI = {10.5802/aif.1902},
       URL = {https://doi.org/10.5802/aif.1902},
}

@article {GKLO,
    AUTHOR = {Gerasimov, A. and Kharchev, S. and Lebedev, D. and Oblezin,
              S.},
     TITLE = {On a class of representations of the {Y}angian and moduli
              space of monopoles},
   JOURNAL = {Comm. Math. Phys.},
  FJOURNAL = {Communications in Mathematical Physics},
    VOLUME = {260},
      YEAR = {2005},
    NUMBER = {3},
     PAGES = {511--525}
}

@book {GW10,
    AUTHOR = {G\"ortz, Ulrich and Wedhorn, Torsten},
     TITLE = {Algebraic geometry {I}},
    SERIES = {Advanced Lectures in Mathematics},
      NOTE = {Schemes with examples and exercises},
 PUBLISHER = {Vieweg + Teubner, Wiesbaden},
      YEAR = {2010},
     PAGES = {viii+615},
      ISBN = {978-3-8348-0676-5},
   MRCLASS = {14-01},
  MRNUMBER = {2675155},
MRREVIEWER = {C\'icero\ Carvalho},
       DOI = {10.1007/978-3-8348-9722-0},
       URL = {https://doi.org/10.1007/978-3-8348-9722-0},
}

@article {GuayR16,
    AUTHOR = {Guay, Nicolas and Regelskis, Vidas},
     TITLE = {Twisted {Y}angians for symmetric pairs of types {B}, {C}, {D}},
   JOURNAL = {Math. Z.},
  FJOURNAL = {Mathematische Zeitschrift},
    VOLUME = {284},
      YEAR = {2016},
    NUMBER = {1-2},
     PAGES = {131--166},
      ISSN = {0025-5874},
   MRCLASS = {17B67 (16T25 81R10)},
  MRNUMBER = {3545488},
MRREVIEWER = {Chengming Bai},
       DOI = {10.1007/s00209-016-1649-2},
       URL = {https://doi.org/10.1007/s00209-016-1649-2},
}

@article {Harada,
    AUTHOR = {Harada, Megumi},
     TITLE = {The symplectic geometry of the {G}el'fand-{C}etlin-{M}olev basis for representations of {${\rm
              Sp}(2n,{\Bbb C})$}},
   JOURNAL = {J. Symplectic Geom.},
  FJOURNAL = {The Journal of Symplectic Geometry},
    VOLUME = {4},
      YEAR = {2006},
    NUMBER = {1},
     PAGES = {1--41}
}

@article {HK16,
    AUTHOR = {Hanany, Amihay and Kalveks, Rudolph},
     TITLE = {Quiver theories for moduli spaces of classical group nilpotent
              orbits},
   JOURNAL = {J. High Energy Phys.},
  FJOURNAL = {Journal of High Energy Physics},
      YEAR = {2016},
    NUMBER = {6},
     PAGES = {130, front matter+60},
      ISSN = {1126-6708,1029-8479},
   MRCLASS = {81V45},
  MRNUMBER = {3538768},
       DOI = {10.1007/JHEP06(2016)130},
       URL = {https://doi-org.ezproxy.usherbrooke.ca/10.1007/JHEP06(2016)130},
}

@article {KPW22,
    AUTHOR = {Kamnitzer, Joel and Pham, Khoa and Weekes, Alex},
     TITLE = {Hamiltonian reduction for affine {G}rassmannian slices and
              truncated shifted {Y}angians},
   JOURNAL = {Adv. Math.},
  FJOURNAL = {Advances in Mathematics},
    VOLUME = {399},
      YEAR = {2022},
     PAGES = {Paper No. 108281, 52},
      ISSN = {0001-8708,1090-2082},
   MRCLASS = {14M15 (17B10)},
  MRNUMBER = {4385132},
MRREVIEWER = {Felipe\ Zald\'ivar},
       DOI = {10.1016/j.aim.2022.108281},
       URL = {https://doi.org/10.1016/j.aim.2022.108281},
}

@article {KWWY14,
    AUTHOR = {Kamnitzer, Joel and Webster, Ben and Weekes, Alex and Yacobi,
              Oded},
     TITLE = {Yangians and quantizations of slices in the affine
              {G}rassmannian},
   JOURNAL = {Algebra Number Theory},
  FJOURNAL = {Algebra \& Number Theory},
    VOLUME = {8},
      YEAR = {2014},
    NUMBER = {4},
     PAGES = {857--893},
      ISSN = {1937-0652},
   MRCLASS = {17B37 (14D24 14M15 20G15 53D55)},
  MRNUMBER = {3248988},
MRREVIEWER = {Christian Ohn},
       DOI = {10.2140/ant.2014.8.857},
       URL = {https://doi.org/10.2140/ant.2014.8.857},
}

@article {KP82,
    AUTHOR = {Kraft, Hanspeter and Procesi, Claudio},
     TITLE = {On the geometry of conjugacy classes in classical groups},
   JOURNAL = {Comment. Math. Helv.},
  FJOURNAL = {Commentarii Mathematici Helvetici},
    VOLUME = {57},
      YEAR = {1982},
    NUMBER = {4},
     PAGES = {539--602},
      ISSN = {0010-2571,1420-8946},
   MRCLASS = {14L35 (14B05 14D25)},
  MRNUMBER = {694606},
MRREVIEWER = {I.\ Dolgachev},
       DOI = {10.1007/BF02565876},
       URL = {https://doi.org/10.1007/BF02565876},
}

@article {Li19,
    AUTHOR = {Li, Yiqiang},
     TITLE = {Quiver varieties and symmetric pairs},
   JOURNAL = {Represent. Theory},
  FJOURNAL = {Representation Theory. An Electronic Journal of the American
              Mathematical Society},
    VOLUME = {23},
      YEAR = {2019},
     PAGES = {1--56},
   MRCLASS = {16S30 (14J50 14L35 53D05)},
  MRNUMBER = {3900699},
MRREVIEWER = {Kevin D. Coulembier},
       DOI = {10.1090/ert/522},
       URL = {https://doi.org/10.1090/ert/522},
}

@unpublished{LWW,
    author = {Lu, Kang and Wang, Weiqiang and Weekes, Alex},
    title = {Shifted twisted {Y}angians and affine {G}rassmannian islices, {II}},
   JOURNAL = {},
      YEAR = {2026},
NOTE ={in prepartion},
    URL = {}
}

@article {LW21,
    AUTHOR = {Lu, Ming and Wang, Weiqiang},
     TITLE = {A {D}rinfeld type presentation of affine {$\imath$}quantum
              groups {I}: {S}plit {ADE} type},
   JOURNAL = {Adv. Math.},
  FJOURNAL = {Advances in Mathematics},
    VOLUME = {393},
      YEAR = {2021},
     PAGES = {Paper No. 108111, 46},
      ISSN = {0001-8708},
   MRCLASS = {17B37 (17B67)},
  MRNUMBER = {4340233},
MRREVIEWER = {Andrea Appel},
       DOI = {10.1016/j.aim.2021.108111},
       URL = {https://doi.org/10.1016/j.aim.2021.108111},
}

@article{LWZ25,
    author = {Lu, Kang and Wang, Weiqiang and Zhang, Weinan},
    title = {A {D}rinfeld type presentation of twisted {Y}angians},
   JOURNAL = {Represent. Theory},
  FJOURNAL = {Representation Theory. An Electronic Journal of the American
              Mathematical Society},
    VOLUME = {29},
      YEAR = {2025},
     PAGES = {838--870}
}

@article{LWZ25degen,
    author = {Lu, Kang and Wang, Weiqiang and Zhang, Weinan},
    title = {Affine $\imath$quantum groups and twisted {Y}angians in {D}rinfeld presentations},
   JOURNAL = {{Comm. Math. Phys.}},
      YEAR = {2025},
Volume = {406},
    NUMBER = {5},
     PAGES = {Paper No. 98, pp 36}
}

@article{LZ24,
    author = {Lu, Kang and Zhang, Weinan},
    title = {A {D}rinfeld type presentation of twisted {Y}angians of quasi-split type},
   JOURNAL = {Commun. Contemp. Math},
      YEAR = {2025},
NOTE ={\arxiv{2408.06981}},
    URL = {
https://doi.org/10.1142/S021919972650001X}
}

@article{LPTTW25,
    author = {Lu, K. and Peng, Y.-N. and Tappeiner, L. and Topley, L. and Wang, W.},
    title = {{Shifted twisted Yangians and finite $W$-algebras of classical type}},
   JOURNAL = {arXiv preprint \arxiv{2505.03316}},
      YEAR = {2025},
NOTE ={},
    URL = {
https://doi.org/abs/2505.03316}
}

@article {mLWZ24,
    AUTHOR = {Lu, Ming and Wang, Weiqiang and Zhang, Weinan},
     TITLE = {Braid group action and quasi-split affine {$\imath$}quantum
              groups {II}: {H}igher rank},
   JOURNAL = {Comm. Math. Phys.},
  FJOURNAL = {Communications in Mathematical Physics},
    VOLUME = {405},
      YEAR = {2024},
    NUMBER = {6},
     PAGES = {Paper No. 142, 33},
      ISSN = {0010-3616},
   MRCLASS = {17B37 (20F36 20G42 81R50)},
  MRNUMBER = {4751724},
       DOI = {10.1007/s00220-024-05005-7},
       URL = {https://doi.org/10.1007/s00220-024-05005-7},
}

@article {Lu81,
    AUTHOR = {Lusztig, G.},
     TITLE = {Green polynomials and singularities of unipotent classes},
   JOURNAL = {Adv. Math.},
  FJOURNAL = {Advances in Mathematics},
    VOLUME = {42},
      YEAR = {1981},
    NUMBER = {2},
     PAGES = {169--178},
      ISSN = {0001-8708},
   MRCLASS = {20G40 (14M15)},
  MRNUMBER = {641425},
MRREVIEWER = {James\ E.\ Humphreys},
       DOI = {10.1016/0001-8708(81)90038-4},
       URL = {https://doi.org/10.1016/0001-8708(81)90038-4},
}

@article {Maf05,
    AUTHOR = {Maffei, Andrea},
     TITLE = {Quiver varieties of type {A}},
   JOURNAL = {Comment. Math. Helv.},
  FJOURNAL = {Commentarii Mathematici Helvetici. A Journal of the Swiss
              Mathematical Society},
    VOLUME = {80},
      YEAR = {2005},
    NUMBER = {1},
     PAGES = {1--27},
      ISSN = {0010-2571},
   MRCLASS = {14L30 (16G20)},
  MRNUMBER = {2130242},
MRREVIEWER = {Alistair Savage},
       DOI = {10.4171/CMH/1},
       URL = {https://doi.org/10.4171/CMH/1},
}

@article {MV22,
    AUTHOR = {Mirkovi\'c, Ivan and Vybornov, Maxim},
     TITLE = {Comparison of quiver varieties, loop {G}rassmannians and
              nilpotent cones in type {$A$}},
      NOTE = {With an appendix by Vasily Krylov},
   JOURNAL = {Adv. Math.},
  FJOURNAL = {Advances in Mathematics},
    VOLUME = {407},
      YEAR = {2022},
     PAGES = {Paper No. 108397, 54},
      ISSN = {0001-8708,1090-2082}
}

@book {Mol07,
    AUTHOR = {Molev, Alexander},
     TITLE = {Yangians and classical {L}ie algebras},
    SERIES = {Mathematical Surveys and Monographs},
    VOLUME = {143},
 PUBLISHER = {American Mathematical Society, Providence, RI},
      YEAR = {2007},
     PAGES = {xviii+400},
      ISBN = {978-0-8218-4374-1},
   MRCLASS = {17B37 (16W35 82B23)},
  MRNUMBER = {2355506},
MRREVIEWER = {Ian M. Musson},
       DOI = {10.1090/surv/143},
       URL = {https://doi.org/10.1090/surv/143},
}

@article {MR02,
    AUTHOR = {Molev, A. and Ragoucy, E.},
     TITLE = {Representations of reflection algebras},
   JOURNAL = {Rev. Math. Phys.},
  FJOURNAL = {Reviews in Mathematical Physics. A Journal for Both Review and
              Original Research Papers in the Field of Mathematical Physics},
    VOLUME = {14},
      YEAR = {2002},
    NUMBER = {3},
     PAGES = {317--342},
      ISSN = {0129-055X},
   MRCLASS = {16S99 (17B37 82B23)},
  MRNUMBER = {1894013},
MRREVIEWER = {Alexei P. Isaev},
       DOI = {10.1142/S0129055X02001156},
       URL = {https://doi.org/10.1142/S0129055X02001156},
}

@article {MW23,
    AUTHOR = {Muthiah, Dinakar and Weekes, Alex},
     TITLE = {Symplectic leaves for generalized affine {G}rassmannian
              slices},
   JOURNAL = {Ann. Sci. \'Ec. Norm. Sup\'er. (4)},
  FJOURNAL = {Annales Scientifiques de l'\'Ecole Normale Sup\'erieure.
              Quatri\`eme S\'erie},
    VOLUME = {56},
      YEAR = {2023},
    NUMBER = {1},
     PAGES = {287--298},
      ISSN = {0012-9593,1873-2151},
   MRCLASS = {14D24 (14J42 14M15 18A25)},
  MRNUMBER = {4637133},
}

@article {MW24,
    AUTHOR = {Muthiah, Dinakar and Weekes, Alex},
     TITLE = {Fundamental monopole operators and embeddings of {K}ac-{M}oody
              affine {G}rassmannian slices},
   JOURNAL = {Int. Math. Res. Not. IMRN},
  FJOURNAL = {International Mathematics Research Notices. IMRN},
      YEAR = {2024},
    NUMBER = {15},
     PAGES = {11159--11189},
      ISSN = {1073-7928,1687-0247},
   MRCLASS = {14D24 (14M15 16G20 17B67 20G15)},
  MRNUMBER = {4782795}
}

@incollection {Ols92,
    AUTHOR = {Olshanski\u{\i}, G.},
     TITLE = {Twisted {Y}angians and infinite-dimensional classical {L}ie
              algebras},
 BOOKTITLE = {Quantum groups ({L}eningrad, 1990)},
    SERIES = {Lecture Notes in Math.},
    VOLUME = {1510},
     PAGES = {104--119},
 PUBLISHER = {Springer, Berlin},
      YEAR = {1992},
   MRCLASS = {17B37 (17B65)},
  MRNUMBER = {1183482},
MRREVIEWER = {Janusz Grabowski},
       DOI = {10.1007/BFb0101183},
       URL = {https://doi.org/10.1007/BFb0101183},
}

@article {Nak94,
    AUTHOR = {Nakajima, Hiraku},
     TITLE = {Instantons on {ALE} spaces, quiver varieties, and
              {K}ac-{M}oody algebras},
   JOURNAL = {Duke Math. J.},
  FJOURNAL = {Duke Mathematical Journal},
    VOLUME = {76},
      YEAR = {1994},
    NUMBER = {2},
     PAGES = {365--416},
      ISSN = {0012-7094},
   MRCLASS = {53C25 (17B67 58D27 58E15)},
  MRNUMBER = {1302318},
MRREVIEWER = {Andrew Dancer},
       DOI = {10.1215/S0012-7094-94-07613-8},
       URL = {https://doi.org/10.1215/S0012-7094-94-07613-8},
}

@article{Nak25,
    author = {Nakajima, Hiraku},
    title = {Instantons on {ALE} spaces for classical groups, involutions on quiver varieties, and quantum symmetric pairs},
    year = {2025},    
    journal = {arXiv preprint \arxiv{2510.13007}}
}

@article {PR08,
    AUTHOR = {Pappas, G. and Rapoport, M.},
     TITLE = {Twisted loop groups and their affine flag varieties},
      NOTE = {With an appendix by T. Haines and Rapoport},
   JOURNAL = {Adv. Math.},
  FJOURNAL = {Advances in Mathematics},
    VOLUME = {219},
      YEAR = {2008},
    NUMBER = {1},
     PAGES = {118--198},
      ISSN = {0001-8708,1090-2082},
   MRCLASS = {22E67 (14M15 17B67 20G25)},
  MRNUMBER = {2435422},
MRREVIEWER = {Dmitry\ A.\ Timash\"ev},
       DOI = {10.1016/j.aim.2008.04.006},
       URL = {https://doi.org/10.1016/j.aim.2008.04.006},
}

@article {S16,
    AUTHOR = {Shapiro, Alexander},
     TITLE = {Poisson geometry of monic matrix polynomials},
   JOURNAL = {Int. Math. Res. Not. IMRN},
  FJOURNAL = {International Mathematics Research Notices. IMRN},
      YEAR = {2016},
    NUMBER = {17},
     PAGES = {5427--5453},
      ISSN = {1073-7928,1687-0247}
}

@article{SSX25,
    author = {Shen, Yaolong and Su, Changjian and Xiong, Rui},
    title = {Quivers with involutions and shifted twisted {Y}angians via {C}oulomb branches},
   JOURNAL = {arXiv preprint \arxiv{2510.12118}},
      YEAR = {2025},
NOTE ={},
    URL = {}
}

@misc{TT24,
    AUTHOR = {Tappeiner, Lukas and Topley, Lewis},
    TITLE = {Shifted twisted {Y}angians and {S}lodowy slices in classical {L}ie algebras},
    YEAR = 2024,
    NOTE ={\arxiv{2406.05492}},
    URL = {https://doi.org/abs/2406.05492}
}

@article{TCB,
    author = {Teleman, Constantin},
    title ={Coulomb branches for quaternionic representations},
    journal={arXiv preprint \arxiv{2209.01088}},
    year= {2022},
    note = {},
    url ={https://doi.org/abs/2209.01088}
}

@article {T23,
    AUTHOR = {Topley, Lewis},
     TITLE = {One dimensional representations of finite {$W$}-algebras,
              {D}irac reduction and the orbit method},
   JOURNAL = {Invent. Math.},
  FJOURNAL = {Inventiones Mathematicae},
    VOLUME = {234},
      YEAR = {2023},
    NUMBER = {3},
     PAGES = {1039--1107},
      ISSN = {0020-9910,1432-1297},
   MRCLASS = {17B35 (17B08 17B10 17B20 17B63)},
  MRNUMBER = {4665777},
MRREVIEWER = {William\ M.\ McGovern},
       DOI = {10.1007/s00222-023-01215-3},
       URL = {https://doi.org/10.1007/s00222-023-01215-3},
}

@article {WWY20,
    AUTHOR = {Webster, Ben and Weekes, Alex and Yacobi, Oded},
     TITLE = {A quantum {M}irkovi\'c-{V}ybornov isomorphism},
   JOURNAL = {Represent. Theory},
  FJOURNAL = {Representation Theory. An Electronic Journal of the American
              Mathematical Society},
    VOLUME = {24},
      YEAR = {2020},
     PAGES = {38--84}
}

@article{W19,
    author={Weekes, Alex},
    title={Generators for {C}oulomb branches of quiver gauge theories},
    year = 2019,
    note ={\arxiv{1903.07734}},
    url = {https://doi.org/abs/1903.07734}
}

@article {W22,
    AUTHOR = {Weekes, Alex},
     TITLE = {Quiver gauge theories and symplectic singularities},
   JOURNAL = {Adv. Math.},
  FJOURNAL = {Advances in Mathematics},
    VOLUME = {396},
      YEAR = {2022},
     PAGES = {Paper No. 108185, 21},
      ISSN = {0001-8708,1090-2082},
   MRCLASS = {14D20 (81T13)},
  MRNUMBER = {4362780},
MRREVIEWER = {Dave\ Auckly},
       DOI = {10.1016/j.aim.2022.108185},
       URL = {https://doi.org/10.1016/j.aim.2022.108185},
}

@article{W24, 
    author={Wendlandt, Curtis},
    title={The restricted quantum double of the {Y}angian}, 
    journal={Canadian Journal of Mathematics}, year={2024}, 
    pages={1–72},
    DOI={10.4153/S0008414X24000142}, 
}

@article {Xu03,
    AUTHOR = {Xu, Ping},
     TITLE = {Dirac submanifolds and {P}oisson involutions},
   JOURNAL = {Ann. Sci. \'{E}cole Norm. Sup. (4)},
  FJOURNAL = {Annales Scientifiques de l'\'{E}cole Normale Sup\'{e}rieure. Quatri\`eme
              S\'{e}rie},
    VOLUME = {36},
      YEAR = {2003},
    NUMBER = {3},
     PAGES = {403--430},
      ISSN = {0012-9593},
   MRCLASS = {53D17 (37J05 58H05 70G45 70H45)},
  MRNUMBER = {1977824},
MRREVIEWER = {Rui Loja Fernandes},
       DOI = {10.1016/S0012-9593(03)00013-2},
       URL = {https://doi.org/10.1016/S0012-9593(03)00013-2},
}

@unpublished{Z20,
    Author = {Zhou, Yehao},
    TITLE = {Note on some properties of generalized affine {G}rassmannian slices},
    YEAR = {2020},
    NOTE = {\arxiv{2011.04109}},
    URL = {https://doi.org/abs/2011.04109}
}
\end{document}